\documentclass[sort&compress,3p]{elsarticle}
\usepackage{algorithm}
\usepackage{algpseudocode}
\usepackage{amsmath}
\usepackage{amssymb}
\usepackage{amsthm}
\usepackage{array}
\usepackage{booktabs}
\usepackage{breqn}
\usepackage{changepage}
\usepackage{collcell}
\usepackage{color}
\usepackage{diagbox}
\usepackage{enumerate}
\usepackage{enumitem}
\usepackage{epstopdf}
\usepackage{float}
\usepackage{geometry}
\usepackage{graphics}
\usepackage{graphicx}
\usepackage{hyperref}
\usepackage{latexsym}
\usepackage{longtable}
\usepackage{mathtools}
\usepackage{mdframed}
\usepackage{pdflscape}
\usepackage{pgf}
\usepackage{pifont}
\usepackage{siunitx}
\usepackage{soul}
\usepackage{subcaption}
\usepackage{tabularx}
\usepackage{tikz}
\usepackage{ulem}
\usepackage[utf8]{inputenc}
\usepackage{xcolor}
\usepackage{xspace}

\graphicspath{ {./IMAGES/} }
\DeclarePairedDelimiter\abs{\lvert}{\rvert}%
\DeclarePairedDelimiter\norm{\lVert}{\rVert}%

\makeatletter
\let\OldStatex\Statex
\renewcommand{\Statex}[1][3]{%
  \setlength\@tempdima{\algorithmicindent}%
  \OldStatex\hskip\dimexpr#1\@tempdima\relax}
\makeatother

\makeatletter
\let\oldabs\abs
\def\abs{\@ifstar{\oldabs}{\oldabs*}}
\let\oldnorm\norm
\def\norm{\@ifstar{\oldnorm}{\oldnorm*}}
\makeatother

\newcommand{\I}{\mathbf{I}}

\newcommand{\J}{\mathbf{J}}

\newcommand{\A}{\mathbf{A}}
\newcommand{\Mathematica}{\texttt{Mathematica}\textsuperscript{\textregistered}\,}
\newcommand{\Matlab}{\texttt{MATLAB}\textsuperscript{\textregistered}\,}
\newcommand{\Matlode}{\texttt{MATLODE}\textsuperscript{\textregistered}\,}

\newtheorem{theorem}{Theorem}[section]
\newtheorem{remark}[theorem]{Remark}
\newtheorem{lemma}[theorem]{Lemma}
\newtheorem{definition}[theorem]{Definition}

\newtheorem{corollary}[theorem]{Corollary}

\newcommand*{\MinNumber}{0.5}%
\newcommand*{\MaxNumber}{1.8}%

\newcommand{\ApplyGradient}[1]{%
    \pgfmathsetmacro{\PercentColor}{100.0*(#1-\MinNumber)/(\MaxNumber-\MinNumber)}
    \textcolor{black!\PercentColor}{#1}
}
\newcolumntype{R}{>{\collectcell\ApplyGradient}{r}<{\endcollectcell}}

\title{EPIRK-$W$ and EPIRK-$K$ time discretization methods}

\begin{document}

\author[aff1]{Mahesh Narayanamurthi}
\ead{maheshnm@vt.edu}
\author[aff1]{Paul Tranquilli}
\ead{ptranq@vt.edu}
\author[aff1]{Adrian Sandu}
\ead{sandu@cs.vt.edu}
\author[aff2]{Mayya Tokman}
\ead{mtokman@ucmerced.edu}
\address[aff1]{Computational Science Laboratory, Department of Computer Science, Virginia Tech, Blacksburg, VA 24060}
\address[aff2]{School of Natural Sciences, University of California, Merced, CA 95343}

\thispagestyle{empty}
\setcounter{page}{0}

\makeatletter
\def\Year#1{%
  \def\yy@##1##2##3##4;{##3##4}%
  \expandafter\yy@#1;
}
\makeatother

\begin{Huge}
\begin{center}
Computational Science Laboratory Technical Report CSL-TR-\Year{\the\year}-{\tt 2} \\
\today
\end{center}
\end{Huge}
\vfil
\begin{huge}
\begin{center}
Mahesh Narayanamurthi, Paul Tranquilli, Adrian Sandu and  Mayya Tokman
\end{center}
\end{huge}

\vfil
\begin{huge}
\begin{it}
\begin{center}
``{\tt EPIRK-$W$ and EPIRK-$K$ time discretization methods}''
\end{center}
\end{it}
\end{huge}
\vfil

\begin{large}
\begin{center}
Computational Science Laboratory \\
Computer Science Department \\
Virginia Polytechnic Institute and State University \\
Blacksburg, VA 24060 \\
Phone: (540)-231-2193 \\
Fax: (540)-231-6075 \\ 
Email: \url{maheshnm@vt.edu} \\
Web: \url{http://csl.cs.vt.edu}
\end{center}
\end{large}

\vspace*{1cm}

\begin{tabular}{ccc}
\includegraphics[width=2.5in]{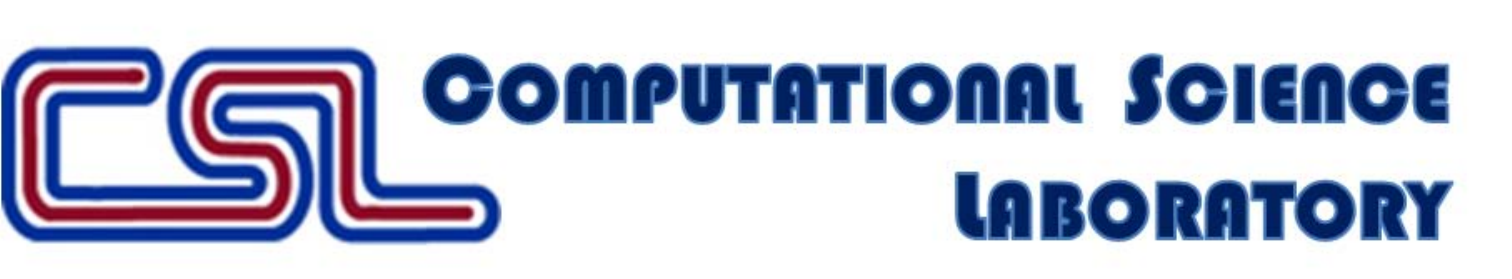}
&\hspace{2.5in}&
\includegraphics[width=2.5in]{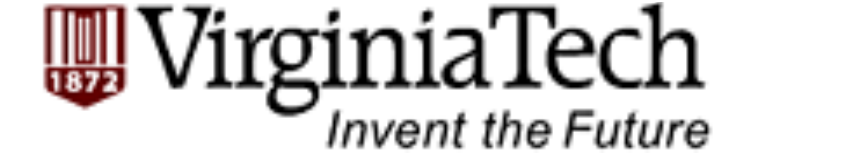} \\
{\bf\large \textit{Compute the Future}} &&\\
\end{tabular}

\newpage

\begin{abstract}
Exponential integrators are special time discretization methods where the traditional linear system solves used by implicit schemes are replaced with computing the action of matrix exponential-like functions on a vector. A very general formulation of exponential integrators is offered by the  Exponential Propagation Iterative methods of Runge-Kutta type (EPIRK) family of schemes. 
The use of Jacobian approximations is an important strategy to drastically reduce the overall computational costs of implicit schemes while maintaining the quality of their solutions.  This paper extends the EPIRK class to allow the use of inexact Jacobians as arguments of the matrix exponential-like functions. Specifically, we develop two new families of methods: EPIRK-$W$ integrators that can accommodate any approximation of the Jacobian, and EPIRK-$K$ integrators that rely on a specific Krylov-subspace projection of the exact Jacobian. Classical order conditions theories are constructed for these families. A practical EPIRK-$W$ method of order three and an EPIRK-$K$ method of order four are developed. Numerical experiments indicate that the methods proposed herein are computationally favorable when compared to existing exponential integrators.
\end{abstract}

\maketitle

\tableofcontents

\section{Introduction}

%
The following initial value problem for a system of ordinary differential equations (ODEs)
\begin{equation}
\label{eqn:ode-problem}
y' = f(t, y), \quad y(t_{0}) = y_{0}, \quad t_0 \le t \le t_F, \quad y(t) \in \mathbb{R}^N.
\end{equation}
arises in many applications including numerical solution of partial differential equations (PDEs) using a method of lines.
Different time discretization methods can then be applied to solve \eqref{eqn:ode-problem} and approximate numerical solutions $y_{n}\, \approx y(t_n)$ at discrete times $t_n$.

Runge-Kutta methods \cite{Hairer_book_I} are the prototypical one-step discretizations that use internal stage approximations to propagate the numerical solution forward in time from $t_n$ to $t_{n+1}=t_n+h$. Explicit Runge-Kutta methods \cite[Section II.2]{Hairer_book_I} perform inexpensive calculations at each step, but suffer from stability restrictions on the time step size which makes them inappropriate for solving stiff systems. Implicit Runge-Kutta methods \cite[Section II.7]{Hairer_book_I} require solution of non-linear systems of equations at each step that can be done, for instance, by using a Newton-like approach. This eases the numerical stability restrictions and allows to use large time step sizes for stiff problems, however it also increases the computational cost per step.

Rosenbrock methods \cite[Section IV.7]{Hairer_book_II} arise from linearization of diagonally implicit Runge-Kutta methods, and solve only linear systems of equations at each time step. The Jacobian of the ODE function \eqref{eqn:ode-problem}
\begin{equation}
\label{eqn:ode-jacobian}
\mathbf{J}(t,y) =  \frac{\partial f(t, y)}{\partial y}
\end{equation}
appears explicitly in the formulation of the method, and the order conditions theory relies on using an exact Jacobian during computations. For the solution of large linear systems a popular implementation uses Krylov projection-based iterative methods such as GMRES. Early truncation of the iterative method is equivalent to the use of an approximation of the Jacobian, and the overall scheme may suffer from order reduction \cite{Tranquilli_2016_JacVec}. Rosenbrock-W (ROW) methods \cite[Section IV.7]{Hairer_book_II}  mitigate this by admitting arbitrary approximations of the Jacobian ($\mathbf{A} \approx \mathbf{J}$). This is possible at the cost of a vastly increased number of order conditions, and by using many more stages to allow for enough degrees of freedom to satisfy them. Rosenbrock-K (ROK) methods  \cite{Tranquilli_2014_ROK} are built in the ROW framework by making use of a specific Krylov-based approximation of the Jacobian. This approximation leads to a number of trees that arise in ROW order conditions theory to become isomorphic to one another, leading to a dramatic reduction in the total number of order conditions.

Exponential integrators \cite{Hochbruck_1997_exp,Tokman_2006_EPI,Hochbruck_2010_exp,Tokman_2011_EPIRK, Schulze_2008_exp} are a class of timestepping methods that replace the linear system solves in Rosenbrock style methods with operations involving the action of matrix exponential-like functions on a vector. For many problems these products of a matrix exponential-like function with a vector can be evaluated more efficiently than solving the corresponding linear systems.

Both the $W$-versions \cite{Hochbruck_1998_exp}  and $K$-versions \cite{Tranquilli_2014_ExpK} of exponential methods have been established for particular formulations of exponential methods. This paper extends the theory of $W$- and $K$-type methods to the general class of Exponential Propagation Iterative methods of Runge-Kutta type (EPIRK) \cite{Tokman_2006_EPI,Tokman_2011_EPIRK}.

The remainder of the paper is laid out as follows. In Section \ref{sec:EPIRK}, we describe the general form of the EPIRK method and some simplifying assumptions that we make in this paper. In Section \ref{sec:EPIRK-W}, we derive the order conditions for an EPIRK-$W$ method and construct methods with two different sets of coefficients. In Section \ref{sec:EPIRK-K}, we extend the K-theory to EPIRK method and construct an EPIRK-$K$ method. Section \ref{sec:Implementations} addresses strategies to evaluate products involving exponential like functions of matrices and vectors  that are present in the formulation of methods discussed in this paper. Section \ref{sec:Results} presents the numerical results, and conclusions are drawn in Section \ref{sec:Conclusions}.

\section{Exponential Propagation Iterative methods of Runge-Kutta type}
\label{sec:EPIRK}

A very general formulation of exponential integrators is offered by the EPIRK family of methods, which was first introduced in \cite{Tokman_2006_EPI}. As is typical  for exponential methods, the formulation of EPIRK is derived from the following integral form of the problem \eqref{eqn:ode-problem} on the time interval $[t_n,t_{n}\,+h]$:
\begin{eqnarray}
\label{eqn:Exponential_Integrator}
y(t_{n} + h) &=& y_{n}\, + h\, \varphi_1\bigl(h\, \mathbf{J}_{n}\,\bigr)\,f(y_{n}) + h \displaystyle\int_{0}^{1} e^{h\,\mathbf{J}_{n}\,(1-\theta)} \,r\big(y(t_{n}\, + \theta h)\big)\, d\theta.
\end{eqnarray}
Equation \eqref{eqn:Exponential_Integrator} is  derived by splitting the right hand-side function of \eqref{eqn:ode-problem} into a linear term and non-linear remainder using first-order Taylor expansion
\begin{equation}
y' = f(y_{n}) + \mathbf{J}_{n}\, (y - y_{n}) + r(y), \quad r(y) := f(y) - f(y_{n}) - \mathbf{J}_{n}\, (y - y_{n}),
\label{eq:ry}
\end{equation}
and using the integrating factor method to transform this ODE into an integral equation \cite{Tokman_2011_EPIRK}. In equation \eqref{eqn:Exponential_Integrator} the Jacobian matrix \eqref{eqn:ode-jacobian} evaluated at time $t_n$ is $\mathbf{J}_n=\mathbf{J}(t_n,y_n)$, and $r(y)$ is the non-linear remainder term of the Taylor expansion of the right-hand side function.
The function $\varphi_1(z) = (e^z - 1)/z$ is the first one in the sequence of analytic functions defined by
\begin{eqnarray}
\label{eqn:phi_k}
\varphi_k(z) = \int_{0}^1 e^{z\,(1 - \theta)} \frac{\theta^{k-1}}{(k-1)!} \, d\theta
= \sum_{i=0}^\infty \frac{z^i}{(i+k)!}, \hspace{1cm} k = 1, 2, \hdots,
\end{eqnarray}
and satisfying the recurrence relation
\begin{eqnarray}
\label{eqn:phi_k_recursive_formulation_and_phi_k_0}
\varphi_{0}(z) = e^z; \quad
\varphi_{k+1}(z) = \frac{\varphi_k(z) - 1/k!}{z}, ~~k = 1, 2, \hdots; \quad \varphi_k(0) = \frac{I}{k!}.
\end{eqnarray}

As explained in \cite{Tokman_2011_EPIRK}, the construction of exponential numerical integrators requires to numerically approximate the integral term in equation  \eqref{eqn:Exponential_Integrator}, and to compute exponential-like matrix function times vector products $\varphi_k(\cdot)\,v$ that include the second term in the right hand side of \eqref{eqn:Exponential_Integrator} as well as additional terms needed by the integral approximation.

Construction of an EPIRK-type scheme begins by considering a general ansatz \cite{Tokman_2011_EPIRK} that describes an $s$-stage method:
\begin{equation}
\begin{split}
Y_i &= y_{n}\, + a_{i,1}\,  \psi_{i,1}(g_{i,1}\,h\mathbf{A}_{i,1})\, hf(y_{n}) + \displaystyle \sum_{j = 2}^{i} a_{i,j}\, \psi_{i,j}(g_{i,j}h\,\mathbf{A}_{i,j})\, h\Delta^{(j-1)}r(y_{n}), \quad i = 1, \hdots, s - 1, \\
y_{n+1} &= y_{n}\, + b_{1}\, \psi_{s,1}(g_{s,1}\,h\mathbf{A}_{s,1})\, hf(y_{n}) + \displaystyle\sum_{j = 2}^{s} b_{j}\, \psi_{s,j}(g_{s,j}h\,\mathbf{A}_{s,j})\, h\Delta^{(j-1)}r(y_{n}),
\label{eq:genEPIRK}
\end{split}
\end{equation}
where $Y_i$ are the internal stage vectors; $y_n$, also denoted as $Y_0$, is the input stage vector; and $y_{n+1}$ is the final output stage of a single step. The $j$-th forward difference $\Delta^{(j)}r(y_{n})$ is calculated using the residual values at the stage vectors:
\[
\Delta^{(1)}r(Y_i) = r(Y_{i+1}) - r(Y_i), \quad
\Delta^{(j)}r(Y_i)  = \Delta^{(j - 1)}r(Y_{i+1}) - \Delta^{(j - 1)}r(Y_{i}).
\]
Particular choices of matrices $\mathbf{A}_{ij}$, functions $\psi_{ij}$ and $r(y)$ define classes of EPIRK methods as well as allow derivation of specific schemes. In this paper we will focus on general, so called unpartitioned, EPIRK methods which can be constructed from \eqref{eq:genEPIRK} by choosing all matrices as a Jacobian matrix evaluated at $t_n$, i.e. $\mathbf{A}_{ij} = \mathbf{J}_n$, function $r(y)$ to be the remainder of the first-order Taylor expansion of $f(y)$ as in \eqref{eq:ry} and functions $\psi_{ij}$ to be linear combinations of $\varphi_{k}$'s:
\begin{eqnarray}
\psi_{i,j}(z) = \displaystyle\sum_{k=1}^{s} p_{i,j,k}\, \varphi_k(z).
\label{eqn:psi_function_definition}
\end{eqnarray}
Here we will also use the simplifying assumption from \cite[eq.~25]{Tokman_2011_EPIRK} 
\begin{eqnarray}
\label{eqn:simplified_psi_function_defn}
\psi_{i,j}(z) = \psi_{j}(z) = \displaystyle\sum_{k=1}^{j} p_{j,k}\, \varphi_k(z).
\end{eqnarray}
These choices lead to the class of general unpartitioned EPIRK schemes described by
\begin{equation}
\label{eqn:EPIRK-Formulation}
\begin{split}
Y_i &= y_{n}\, + a_{i,1}\,  \psi_{1}(g_{i,1}\,h\mathbf{J}_{n})\, hf(y_{n}) + \displaystyle \sum_{j = 2}^{i} a_{i,j}\, \psi_{j}(g_{i,j}h\,\mathbf{J}_{n})\, h\Delta^{(j-1)}r(y_{n}), \quad i = 1, \hdots, s - 1, \\
y_{n+1} &= y_{n}\, + b_{1}\, \psi_{1}(g_{s,1}\,h\mathbf{J}_{n})\, hf(y_{n}) + \displaystyle\sum_{j = 2}^{s} b_{j}\, \psi_{j}(g_{s,j}h\,\mathbf{J}_{n})\, h\Delta^{(j-1)}r(y_{n}).
\end{split}
\end{equation}
Coefficients $p_{j,k}$, $a_{i,j}\,$, $g_{i,j}\,$, and $b_{j}\,$ have to be determined from order conditions to build schemes of specific order. 
\begin{remark}
Computationally efficient strategies to evaluate exponential-like matrix function times vector products vary depending on the size of the ODE system. Approximating products $\psi(z)v$ constitutes the major portion of overall computational cost of the method.  Taylor expansion or Pade based approximations can be used to evaluate the products $\psi(z)v$ for small systems.  For large scale problems, particularly those where matrices cannot be stored explicitly and a matrix-free formulation of an integrator is a necessity, Krylov projection-based methods become the de facto choice \cite{Tokman_2011_EPIRK}.
\end{remark}

EPIRK methods of a given order can be built by solving either classical \cite{Tokman_2011_EPIRK} or stiff \cite{Rainwater_Tokman_2016} order conditions. In this paper we will focus on the classically accurate EPIRK methods.  A classical EPIRK scheme is constructed by matching required number of Taylor expansion coefficients of the exact and approximate solutions to derive order conditions.  The order conditions are then solved to compute the coefficients of the method.  To ensure that the error has the desired order the exact Jacobian matrix $\mathbf{J}_{n}$ has to be used in the Taylor expansion of the numerical solution.  If the Jacobian or its action on a vector is not computed with full accuracy, the numerical method will suffer from order reduction.

In this paper we construct EPIRK methods that retain full accuracy even with an approximated Jacobian. The theory of using inexact Jacobians in the method formulation, first proposed in \cite{Steihaug_1979}  for implicit schemes, is extended to EPIRK methods. More specifically, EPIRK-$W$ methods that admit arbitrary approximations of the Jacobian while maintaining full order of accuracy are derived in Section \ref{sec:EPIRK-W}. The drawback of $W$-methods, as mentioned in \cite[~section IV.7]{Hairer_book_II}, is the very fast growth of the number of order conditions with increasing order of the method. In such cases, significantly larger number of stages are typically needed to build higher-order methods.

In \cite{Tranquilli_2014_ROK,Tranquilli_2014_ExpK} we developed $K$-methods, versions of $W$-methods that use a specific Krylov-subspace approximation of the Jacobian, and dramatically reduce the number of necessary order conditions.  Here we construct $K$-methods for EPIRK schemes in Section \ref{sec:EPIRK-K}. The $K$-method theory enables the construction of high order methods with significantly fewer stages than $W$-methods. For example, we show that three stage EPIRK $W$-methods only exist up to order three, whereas we derive here a fourth order $K$-method. Furthermore, as shown in the the numerical results in Section \ref{sec:Results}, for some problems $K$-methods have better computational performance than traditional exponential integrators.

\section{EPIRK-$W$ methods}
\label{sec:EPIRK-W}

An EPIRK-$W$ method is formulated like a traditional EPIRK method \eqref{eqn:EPIRK-Formulation} with the only difference being the use of an inexact Jacobian ($\mathbf{A}_{n}$) in place of the exact Jacobian ($\mathbf{J}_{n}$). Using the simplification \eqref{eqn:simplified_psi_function_defn} the method reads:
\begin{equation}
\label{eqn:EPIRK-W-Formulation}
\begin{split}
Y_i &= y_{n} + a_{i,1}\,  \psi_{1}(g_{i,1}\,\,h\,\mathbf{A}_{n})\, hf(y_{n}) + \displaystyle \sum_{j = 2}^{i} a_{i,j}\, \psi_{j}(g_{i,j}\,h\,\mathbf{A}_{n})\, h\Delta^{(j-1)}r(y_{n}), \quad i = 1, \hdots, s - 1, \\
y_{n+1} &= y_{n}\, + b_{1}\, \psi_{1}(g_{s,1}\,h\,\mathbf{A}_{n})\, hf(y_{n}) + \displaystyle\sum_{j = 2}^{s} b_{j}\, \psi_{j}(g_{s,j}\,h\,\mathbf{A}_{n})\, h\Delta^{(j-1)}r(y_{n}).
\end{split}
\end{equation}
This modification requires the order conditions theory to be modified to accommodate this approximation.

\subsection{Order conditions theory for EPIRK-$W$ methods}
\label{sec:EPIRK-W-order}
The classical order conditions for EPIRK-$W$ methods result from matching the Taylor series expansion coefficients of the numerical solution $y_{n+1}$ up to a certain order with those from the Taylor series expansion of the exact solution $y(t_{n}\, + h)$. The construction of order conditions is conveniently expressed in terms of Butcher trees \cite{Hairer_book_I}. The trees corresponding to the elementary differentials of the numerical solution are the \textit{TW}-trees. \textit{TW}-trees are rooted Butcher trees with two different colored nodes - fat (empty) and meagre (full) - such that the end vertices are meagre and the fat vertices are singly branched. The \textit{TW}-trees up to order four are shown in Tables \ref{Table:TWtrees1} and \ref{Table:TWtrees2} following \cite{Tranquilli_2014_ExpK}.

\begin{table}[!htb]
\centering
\def\arraystretch{1.5}
\footnotesize
\begin{tabular}{|>{\centering\arraybackslash}m{1.2in}|>{\centering\arraybackslash}m{1in}|>{\centering\arraybackslash}m{1in}|>{\centering\arraybackslash}m{1in}|>{\centering\arraybackslash}m{1in}|}
\hline
$\tau$ & \begin{tikzpicture}[scale=.5]
      [meagre/.style={circle,draw, fill=black!100,thick},
      fat/.style={circle,draw,thick}]
      \node[circle,draw,, fill=black!100,thick] (j) at (0,0) [label=right:$j$] {};
  \end{tikzpicture} & \begin{tikzpicture}[scale=.5]
      [meagre/.style={circle,draw, fill=black!100,thick},
      fat/.style={circle,draw,thick}]
      \node[circle,draw, fill=black!100,thick] (j) at (0,0) [label=right:$j$] {};
      \node[circle,draw, fill=black!100,thick] (k) at (1,1) [label=right:$k$] {};
      \draw[-] (j) -- (k);
  \end{tikzpicture} & \begin{tikzpicture}[scale=.5]
      [meagre/.style={circle,draw, fill=black!100,thick},
      fat/.style={circle,draw,thick}]
      \node[circle,draw, thick] (j) at (0,0) [label=right:$j$] {};
      \node[circle,draw, fill=black!100,thick] (k) at (1,1) [label=right:$k$] {};
      \draw[-] (j) -- (k);
  \end{tikzpicture} & \begin{tikzpicture}[scale=.5]
      [meagre/.style={circle,draw, fill=black!100,thick},
      fat/.style={circle,draw,thick}]
      \node[circle,draw, fill=black!100,thick] (j) at (0,0) [label=right:$j$] {};
      \node[circle,draw, fill=black!100,thick] (k) at (1,1) [label=right:$k$] {};
      \node[circle,draw, fill=black!100,thick] (l) at (-1,1) [label=left:$l$] {};
      \draw[-] (j) -- (k);
      \draw[-] (j) -- (l);
  \end{tikzpicture} \\
\hline
{W-tree name} & {$\tau^{W}_{1}$} & {$\tau^{W}_{2}$} & {$\tau^{W}_{3}$} & {$\tau^{W}_{4}$}  \\
\hline
{K-tree name} & {$\tau^{K}_{1}$} & {$\tau^{K}_{2}$} & {--} & {$\tau^{K}_{3}$}  \\
\hline
$F(\tau)$ & $f^J$ & $f^J_Kf^K$ & $\A_{JK}f^K$ & $f^J_{KL}f^Kf^L$    \\
\hline
$\mathsf{a}(\tau)$ & $x_1$ & $x_2$ & $x_3$ & $x_4$  \\
\hline
$B^\#\left(hf(B(\mathsf{a},y))\right)$ & 1 & $x_1$ & $0$ & $x_1^2$  \\
\hline
$B^\#\left(h\A B(\mathsf{a},y)\right)$ & $0$ & $0$ & $x_1$ & $0$ \\
\hline
$B^\#\left(\varphi_j(h\, \A)B(\mathsf{a},y)\right)$ & $c_0x_1$ & $c_0x_2$ & $c_0x_3 + c_1x_1$ & $c_0x_4$ \\
\hline
\hline 
$\tau$ & \begin{tikzpicture}[scale=.5]
    [meagre/.style={circle,draw, fill=black!100,thick},
    fat/.style={circle,draw,thick}]
    \node[circle,draw, fill=black!100,thick] (j) at (0,0) [label=right:$j$] {};
    \node[circle,draw, fill=black!100,thick] (k) at (1,1) [label=right:$k$] {};
    \node[circle,draw, fill=black!100,thick] (l) at (0,2) [label=left:$l$] {};
    \draw[-] (j) -- (k);
    \draw[-] (k) -- (l);
\end{tikzpicture} & \begin{tikzpicture}[scale=.5]
      [meagre/.style={circle,draw, fill=black!100,thick},
      fat/.style={circle,draw,thick}]
      \node[circle,draw, fill=black!100,thick] (j) at (0,0) [label=right:$j$] {};
      \node[circle,draw, thick] (k) at (1,1) [label=right:$k$] {};
      \node[circle,draw, fill=black!100,thick] (l) at (0,2) [label=left:$l$] {};
      \draw[-] (j) -- (k);
      \draw[-] (k) -- (l);
  \end{tikzpicture} & \begin{tikzpicture}[scale=.5]
      [meagre/.style={circle,draw, fill=black!100,thick},
      fat/.style={circle,draw,thick}]
      \node[circle,draw,thick] (j) at (0,0) [label=right:$j$] {};
      \node[circle,draw, fill=black!100, thick] (k) at (1,1) [label=right:$k$] {};
      \node[circle,draw, fill=black!100,thick] (l) at (0,2) [label=left:$l$] {};
      \draw[-] (j) -- (k);
      \draw[-] (k) -- (l);
  \end{tikzpicture} & \begin{tikzpicture}[scale=.5]
      [meagre/.style={circle,draw, fill=black!100,thick},
      fat/.style={circle,draw,thick}]
      \node[circle,draw, thick] (j) at (0,0) [label=right:$j$] {};
      \node[circle,draw, thick] (k) at (1,1) [label=right:$k$] {};
      \node[circle,draw, fill=black!100,thick] (l) at (0,2) [label=left:$l$] {};
      \draw[-] (j) -- (k);
      \draw[-] (k) -- (l);
  \end{tikzpicture} \\
\hline
{W-tree name} & {$\tau^{W}_{5}$} & {$\tau^{W}_{6}$} & {$\tau^{W}_{7}$} & {$\tau^{W}_{8}$} \\
\hline
{K-tree name} & {$\tau^{K}_{4}$} & {--} & {--} & {--} \\
\hline
$F(\tau)$ & $f^J_Kf^K_Lf^L$ & $f^J_K\A_{KL}f^L$ & $\A_{JK}f^K_Lf^L$ & $\A_{JK}\A_{KL}f^L$  \\
\hline
$\mathsf{a}(\tau)$ & $x_5$ & $x_6$ & $x_7$ & $x_8$  \\
\hline
$B^\#\left(hf(B(\mathsf{a},y))\right)$ & $x_2$ & $x_3$ & $0$ & $0$  \\
\hline
$B^\#\left(h\A B(\mathsf{a},y)\right)$  & $0$ & $0$ & $x_2$ & $x_3$  \\
\hline
$B^\#\left(\varphi_j(h\, \A)B(\mathsf{a},y)\right)$  & $c_0x_5$ & $c_0x_6$ & $c_0x_7 + c_1x_2$ & $c_0x_8 + c_1x_3 + c_2x_1$  \\
\hline
\hline 
$\tau$  & \begin{tikzpicture}[scale=.5]
    [meagre/.style={circle,draw, fill=black!100,thick},
    fat/.style={circle,draw,thick}]
    \node[circle,draw, fill=black!100,thick] (j) at (0,0) [label=right:$j$] {};
    \node[circle,draw, fill=black!100,thick] (k) at (-1,1) [label=left:$k$] {};
    \node[circle,draw, fill=black!100,thick] (l) at (0,1) [label=above:$l$] {};
    \node[circle,draw, fill=black!100,thick] (m) at (1,1) [label=right:$m$] {};
    \draw[-] (j) -- (k);
    \draw[-] (j) -- (l);
    \draw[-] (j) -- (m);
\end{tikzpicture} & \begin{tikzpicture}[scale=.5]
    [meagre/.style={circle,draw, fill=black!100,thick},
    fat/.style={circle,draw,thick}]
    \node[circle,draw, fill=black!100,thick] (j) at (0,0) [label=right:$j$] {};
    \node[circle,draw, fill=black!100,thick] (k) at (-1,1) [label=left:$k$] {};
    \node[circle,draw, fill=black!100,thick] (l) at (1,1) [label=right:$l$] {};
    \node[circle,draw, fill=black!100,thick] (m) at (1,2) [label=right:$m$] {};
    \draw[-] (j) -- (k);
    \draw[-] (j) -- (l);
    \draw[-] (l) -- (m);
\end{tikzpicture} &  \begin{tikzpicture}[scale=.5]
      [meagre/.style={circle,draw, fill=black!100,thick},
      fat/.style={circle,draw,thick}]
      \node[circle,draw, fill=black!100,thick] (j) at (0,0) [label=right:$j$] {};
      \node[circle,draw, fill=black!100,thick] (k) at (-1,1) [label=left:$k$] {};
      \node[circle,draw, thick] (l) at (1,1) [label=right:$l$] {};
      \node[circle,draw, fill=black!100,thick] (m) at (1,2) [label=right:$m$] {};
      \draw[-] (j) -- (k);
      \draw[-] (j) -- (l);
      \draw[-] (l) -- (m);
  \end{tikzpicture} & \begin{tikzpicture}[scale=.5]
      [meagre/.style={circle,draw, fill=black!100,thick},
      fat/.style={circle,draw,thick}]
      \node[circle,draw, fill=black!100,thick] (j) at (0,0) [label=right:$j$] {};
      \node[circle,draw, fill=black!100,thick] (k) at (1,1) [label=right:$k$] {};
      \node[circle,draw, fill=black!100,thick] (l) at (2,2) [label=right:$l$] {};
      \node[circle,draw, fill=black!100,thick] (m) at (0,2) [label=left:$m$] {};
      \draw[-] (j) -- (k);
      \draw[-] (k) -- (l);
      \draw[-] (k) -- (m);
  \end{tikzpicture}  \\
\hline
{W-tree name} & {$\tau^{W}_{9}$} & {$\tau^{W}_{10}$}  & {$\tau^{W}_{11}$} & {$\tau^{W}_{12}$} \\
\hline
{K-tree name} & {$\tau^{K}_{5}$} & {$\tau^{K}_{6}$}  & {--} & {$\tau^{K}_{7}$} \\
\hline
$F(\tau)$ & $f^J_{KLM}f^Kf^Lf^M$ & $f^J_{KL}f^L_Mf^Mf^K$ & $f^J_{KL}\A_{LM}f^Mf^K$ & $f^J_Kf^K_{LM}f^Mf^L$ \\
\hline
$\mathsf{a}(\tau)$ & $x_9$ & $x_{10}$ & $x_{11}$ & $x_{12}$  \\
\hline
$B^\#\left(hf(B(\mathsf{a},y))\right)$ & $x_1^3$ & $x_1x_2$ & $x_1x_3$ & $x_4$  \\
\hline
$B^\#\left(h\A B(\mathsf{a},y)\right)$ & $0$ & $0$ & $0$ & $0$ \\
\hline
$B^\#\left(\varphi_j(h\, \A)B(\mathsf{a},y)\right)$ & $c_0x_9$ & $c_0x_{10}$ & $c_0x_{11}$ & $c_0x_{12}$  \\
\hline 
\end{tabular}
\caption{TW-trees up to order four (part one of two).}
\label{Table:TWtrees1}
\end{table}

\begin{table}[!htb]
\centering
\renewcommand{\arraystretch}{1.5}
\footnotesize
\begin{tabular}{|>{\centering\arraybackslash}m{1.2in}|>{\centering\arraybackslash}m{1in}|>{\centering\arraybackslash}m{1in}|>{\centering\arraybackslash}m{1in}|}
\hline
$\tau$ & \begin{tikzpicture}[scale=.5]
        [meagre/.style={circle,draw, fill=black!100,thick},
        fat/.style={circle,draw,thick}]
        \node[circle,draw,thick] (j) at (0,0) [label=right:$j$] {};
        \node[circle,draw, fill=black!100,thick] (k) at (1,1) [label=right:$k$] {};
        \node[circle,draw, fill=black!100,thick] (l) at (2,2) [label=right:$l$] {};
        \node[circle,draw, fill=black!100,thick] (m) at (0,2) [label=left:$m$] {};
        \draw[-] (j) -- (k);
        \draw[-] (k) -- (l);
        \draw[-] (k) -- (m);
    \end{tikzpicture} & \begin{tikzpicture}[scale=.5]
        [meagre/.style={circle,draw, fill=black!100,thick},
        fat/.style={circle,draw,thick}]
        \node[circle,draw, fill=black!100,thick] (j) at (0,0) [label=left:$j$] {};
        \node[circle,draw, fill=black!100,thick] (k) at (1,1) [label=right:$k$] {};
        \node[circle,draw, fill=black!100,thick] (l) at (0,2) [label=left:$l$] {};
        \node[circle,draw, fill=black!100,thick] (m) at (1,3) [label=right:$m$] {};
        \draw[-] (j) -- (k);
        \draw[-] (k) -- (l);
        \draw[-] (l) -- (m);
    \end{tikzpicture} & \begin{tikzpicture}[scale=.5]
        [meagre/.style={circle,draw, fill=black!100,thick},
        fat/.style={circle,draw,thick}]
        \node[circle,draw,fill=black!100,thick] (j) at (0,0) [label=left:$j$] {};
        \node[circle,draw,fill=black!100,thick] (k) at (1,1) [label=right:$k$] {};
        \node[circle,draw, thick] (l) at (0,2) [label=left:$l$] {};
        \node[circle,draw, fill=black!100,thick] (m) at (1,3) [label=right:$m$] {};
        \draw[-] (j) -- (k);
        \draw[-] (k) -- (l);
        \draw[-] (l) -- (m);
    \end{tikzpicture} \\
\hline
{W-tree name} & {$\tau^{W}_{13}$} & {$\tau^{W}_{14}$} & {$\tau^{W}_{15}$}  \\
\hline
{K-tree name} & {$\tau^{K}_{8}$} & {$\tau^{K}_{9}$} & {--}  \\
\hline
$F(\tau)$ & $\A_{JK}f^K_{LM}f^Lf^M$ & $f^J_Kf^K_Lf^L_Mf^M$ & $ f^J_Kf^K_L\A_{LM}f^M $  \\
\hline
$\mathsf{a}(\tau)$ & $x_{13}$ & $x_{14}$ & $x_{15}$\\
\hline
$B^\#\left(hf(B(\mathsf{a},y))\right)$ & $0$ & $x_5$ & $x_6$ \\
\hline
$B^\#\left(h\A B(\mathsf{a},y)\right)$  & $x_4$ & $0$ & $0$ \\
\hline
$B^\#\left(\varphi_j(h\, \A)B(\mathsf{a},y)\right)$ & $c_0x_{13} + c_1 x_4$ & $c_0x_{14}$ & $c_0x_{15}$ \\
\hline
\hline
$\tau$ & \begin{tikzpicture}[scale=.5]
      [meagre/.style={circle,draw, fill=black!100,thick},
      fat/.style={circle,draw,thick}]
      \node[circle,draw, fill=black!100,thick] (j) at (0,0) [label=left:$j$] {};
      \node[circle,draw, thick] (k) at (1,1) [label=right:$k$] {};
      \node[circle,draw, fill=black!100,thick] (l) at (0,2) [label=left:$l$] {};
      \node[circle,draw, fill=black!100,thick] (m) at (1,3) [label=right:$m$] {};
      \draw[-] (j) -- (k);
      \draw[-] (k) -- (l);
      \draw[-] (l) -- (m);
  \end{tikzpicture} & \begin{tikzpicture}[scale=.5]
      [meagre/.style={circle,draw, fill=black!100,thick},
      fat/.style={circle,draw,thick}]
      \node[circle,draw,fill=black!100, thick] (j) at (0,0) [label=left:$j$] {};
      \node[circle,draw, thick] (k) at (1,1) [label=right:$k$] {};
      \node[circle,draw, thick] (l) at (0,2) [label=left:$l$] {};
      \node[circle,draw, fill=black!100,thick] (m) at (1,3) [label=right:$m$] {};
      \draw[-] (j) -- (k);
      \draw[-] (k) -- (l);
      \draw[-] (l) -- (m);
  \end{tikzpicture} & \begin{tikzpicture}[scale=.5]
      [meagre/.style={circle,draw, fill=black!100,thick},
      fat/.style={circle,draw,thick}]
      \node[circle,draw, thick] (j) at (0,0) [label=left:$j$] {};
      \node[circle,draw, fill=black!100, thick] (k) at (1,1) [label=right:$k$] {};
      \node[circle,draw, fill=black!100,thick] (l) at (0,2) [label=left:$l$] {};
      \node[circle,draw, fill=black!100,thick] (m) at (1,3) [label=right:$m$] {};
      \draw[-] (j) -- (k);
      \draw[-] (k) -- (l);
      \draw[-] (l) -- (m);
  \end{tikzpicture} \\
\hline
{W-tree name} & {$\tau^{W}_{16}$} & {$\tau^{W}_{17}$} & {$\tau^{W}_{18}$}  \\
\hline
{K-tree name} & {--} & {--} & {--}  \\
\hline
$F(\tau)$ & $f^J_K\A_{KL}f^L_Mf^M$ & $f^J_K\A_{KL}\A_{LM}f^M$ & $\A_{JK}f^K_Lf^L_Mf^M$ \\
\hline
$\mathsf{a}(\tau)$ & $x_{16}$ & $x_{17}$ & $x_{18}$ \\
\hline
$B^\#\left(hf(B(\mathsf{a},y))\right)$ & $x_7$ & $x_8$ & $0$ \\
\hline
$B^\#\left(h\A B(\mathsf{a},y)\right)$ & $0$ & $0$ & $x_5$ \\
\hline
$B^\#\left(\varphi_j(h\, \A)B(\mathsf{a},y)\right)$ & $c_0x_{16}$ & $c_0x_{17}$ & $c_0x_{18} + c_1x_5$ \\
\hline
\hline
$\tau$ &  \begin{tikzpicture}[scale=.5]
      [meagre/.style={circle,draw, fill=black!100,thick},
      fat/.style={circle,draw,thick}]
      \node[circle,draw, thick] (j) at (0,0) [label=left:$j$] {};
      \node[circle,draw, fill=black!100,thick] (k) at (1,1) [label=right:$k$] {};
      \node[circle,draw, thick] (l) at (0,2) [label=left:$l$] {};
      \node[circle,draw, fill=black!100,thick] (m) at (1,3) [label=right:$m$] {};
      \draw[-] (j) -- (k);
      \draw[-] (k) -- (l);
      \draw[-] (l) -- (m);
  \end{tikzpicture} & \begin{tikzpicture}[scale=.5]
      [meagre/.style={circle,draw, fill=black!100,thick},
      fat/.style={circle,draw,thick}]
      \node[circle,draw, thick] (j) at (0,0) [label=left:$j$] {};
      \node[circle,draw, thick] (k) at (1,1) [label=right:$k$] {};
      \node[circle,draw, fill=black!100,thick] (l) at (0,2) [label=left:$l$] {};
      \node[circle,draw, fill=black!100,thick] (m) at (1,3) [label=right:$m$] {};
      \draw[-] (j) -- (k);
      \draw[-] (k) -- (l);
      \draw[-] (l) -- (m);
  \end{tikzpicture} & \begin{tikzpicture}[scale=.5]
      [meagre/.style={circle,draw, fill=black!100,thick},
      fat/.style={circle,draw,thick}]
      \node[circle,draw, thick] (j) at (0,0) [label=left:$j$] {};
      \node[circle,draw, thick] (k) at (1,1) [label=right:$k$] {};
      \node[circle,draw, thick] (l) at (0,2) [label=left:$l$] {};
      \node[circle,draw, fill=black!100,thick] (m) at (1,3) [label=right:$m$] {};
      \draw[-] (j) -- (k);
      \draw[-] (k) -- (l);
      \draw[-] (l) -- (m);
  \end{tikzpicture} \\
\hline
{W-tree name} & {$\tau^{W}_{19}$} & {$\tau^{W}_{20}$} & {$\tau^{W}_{2,1}$}  \\
\hline
{K-tree name} & {--} & {--} & {--}  \\
\hline
$F(\tau)$ & $ \A_{JK}f^K_L\A_{LM}f^M$ & $ \A_{JK}\A_{KL}f^L_Mf^M$ & $\A_{JK}\A_{KL}\A_{LM}f^M$ \\
\hline
$\mathsf{a}(\tau)$ &  $x_{19}$ & $x_{20}$ & $x_{21}$ \\
\hline
$B^\#\left(hf(B(\mathsf{a},y))\right)$ &$0$ & $0$ & $0$ \\
\hline
$B^\#\left(h\A B(\mathsf{a},y)\right)$ & $x_6$ & $x_7$ & $x_8$ \\
\hline
$B^\#\left(\varphi_j(h\, \A)B(\mathsf{a},y)\right)$ & $c_0x_{19} + c_1x_6$ & $c_0x_{20} + c_1x_7 + c_2x_2$ & $c_0x_{21} + c_1x_8 + c_2x_3 + c_3x_1$ \\
\hline
\end{tabular}
\caption{TW-trees up to order four (part two of two).}
\label{Table:TWtrees2}
\end{table}

In \textit{TW}-trees that correspond to elementary differentials of the Taylor expansion of numerical solution, the meagre nodes represents the appearance of $f$ and its derivatives. The fat nodes represent the appearance of the inexact Jacobian $\mathbf{A}_{n}$. It is useful to note that trees that contain both meagre and fat nodes do not appear in the trees corresponding to the Taylor expansion of the exact solution, as such an expansion does not contain the inexact Jacobian.

In order to construct the order conditions for the $W$-method we rely on B-series theory.  A B-series \cite[Section~II.12]{Hairer_book_I} is an operator that maps a set of of trees onto the set of real coefficients corresponding to the coefficients of the elementary differentials in a Taylor expansion. If $\mathsf{a}: TW \cup \emptyset \rightarrow \mathbb{R}$ is a mapping from the set of TW trees and the empty tree to real numbers the corresponding B-series is given by:
\begin{eqnarray*}
B(\mathsf{a}, y) &=& \mathsf{a}(\emptyset)\, y + \displaystyle\sum_{\tau \in TW} \mathsf{a}(\tau)\, \dfrac{h^{\abs{\tau}}}{\sigma(\tau)}\, F(\tau)(y) {.}
\end{eqnarray*}
Here $\tau$ is a \textit{TW}-tree and $\abs{\tau}$ is the order of the tree (the number of nodes of the \textit{TW}-tree), and $\sigma(\tau)$ is the order of the symmetry group of the tree  \cite{Butcher_2001_book,Tokman_2011_EPIRK}. A simple algorithm for evaluating $\sigma(\tau)$ is given in \cite{Tokman_2011_EPIRK}. Lastly, $F(\tau)(y)$ maps a tree to the corresponding elementary differential.

We also make use of the operator $B^{\#}(g)$, first defined in \cite{Tranquilli_2014_ExpK}, which takes a B-series $g$ and returns its coefficients. Therefore:
\begin{eqnarray*}
B^{\#}\big(B(\mathsf{a}, y)\big) = \mathsf{a}.
\end{eqnarray*}
Tables \ref{Table:TWtrees1} and \ref{Table:TWtrees2} list both TW- and TK-trees up to order four and the corresponding tree names. TK-trees which are a subset of TW-trees will be central to the discussion of $K$-methods in the next section. The subsequent rows of the two tables show the following:
\begin{itemize}
\item The elementary differentials corresponding to each TW-trees up to order four.
\item A set of coefficients $x_i$ for an arbitrary B-series $\mathsf{a}(\tau)$.
\item The coefficients of the B-series resulting from composing the function $f$ with the B-series $\mathsf{a}(\tau)$. The rules for composing B-series are discussed in \cite{Tokman_2011_EPIRK,Butcher_2011_exp-order}.
\item The coefficients of the B-series resulting from left-multiplying the B-series $\mathsf{a}(\tau)$ by an inexact Jacobian. The multiplication rule for B-series is explained in \cite{Butcher_2011_exp-order}.
\item The coefficients of the B-series resulting from left-multiplying the B-series $\mathsf{a}(\tau)$ by $\varphi_j(h\, \mathbf{A}_{n})$. The rules for multiplying a B-series  by $\psi(h\,\mathbf{J}_{n})$, where $\mathbf{J}_{n}$ is the exact Jacobian at $y_{n}\,$, are given in \cite{Tokman_2011_EPIRK}. The rules for multiplying a B-series  by $\psi(h\,\mathbf{A}_{n})$, for an arbitrary matrix $\mathbf{A}_{n}$, are given in \cite{Tranquilli_2014_ExpK}.
\end{itemize}
In addition, we note that the $B$ and $B^{\#}$ operators are linear. Consequently, the B-series of the sum of two arbitrary functions is the sum of the B-series of each and likewise, the operator $B^{\#}$ returns the sum of the coefficients of each lined up against corresponding elementary differentials.

Using B-series operations, as described above, one can systematically construct the B-series coefficients of the EPIRK-$W$ numerical  solution. This is achieved in Algorithm \ref{alg:EPIRK_w_b_series}. The algorithm starts by initializing the set of B-series coefficients to those of the solution at the current time step, $y_{n}\,$. Next, the algorithm operates on the sequence of coefficients by applying the B-series operations that correspond to the mathematical calculations performed during each stage of the EPIRK-$W$ method. All B-series operations are done on coefficient sets truncated to the desired order (order four herein). We repeat this process for each internal stage, and finally for the  last stage, to obtain the coefficients of the B-series corresponding to the numerical solution of one EPIRK-$W$ step.

\begin{algorithm}[h]
\caption{Compute the B-series coefficient of numerical solution}
\label{alg:EPIRK_w_b_series}
\begin{algorithmic}[1]
\State  {\bf Input:} $\mathsf{y_{n}}$\Comment{B-series coefficient of the current numerical solution.}
\For{$i=1:s - 1$}\Comment{Do for each internal stage of the method}
\State $ \mathsf{u} = B^{\#}(h\,f(B(\mathsf{y_{n}},y)))$\Comment{Composition of $f$ with B-series of the current solution.}
\State $ \mathsf{u} = B^{\#}(\psi_{i,1}(h\,g_{i,1}\,\,\mathbf{A}_{n})\,.\,B(\mathsf{u}, y))$\Comment{Multiplication by $\psi$ function}
\State $ \mathsf{u} = a_{i,1} * \mathsf{u}$\Comment{Scaling by a constant}
\For{$j=2:i$}
\State $ \mathsf{v} = B^{\#}(h\, \Delta^{(j-1)}r(\mathsf{y_{n}}))$\Comment{Recursive forward-difference}
\State $ \mathsf{v} = B^{\#}(\psi_{i,j}(h\,g_{i,j}\,\mathbf{A}_{n})\,.\,B(\mathsf{v}, y))$
\State $ \mathsf{u} = \mathsf{u} + a_{i,j}\, * \mathsf{v}$
\EndFor
\State $ \mathsf{Y_i} = \mathsf{y_{n}} + \mathsf{u}$\Comment{Addition of two B-series}
\EndFor
\State $ \mathsf{u} = B^{\#}(h\,f(B(\mathsf{y_{n}},y)))$
\State $ \mathsf{u} = B^{\#}(\psi_{s,1}(h\,g_{s,1}\,\mathbf{A}_{n})\,.\,B(\mathsf{u}, y))$
\State $ \mathsf{u} = b_{1}\, * \mathsf{u}$
\For{$j=2:s$}
\State $ \mathsf{v} = B^{\#}(h\, \Delta^{(j-1)}r(\mathsf{y_{n}}))$
\State $ \mathsf{v} = B^{\#}(\psi_{s,j}(h\,g_{s,j}\,\mathbf{A}_{n})\,.\,B(\mathsf{v}, y))$
\State $ \mathsf{u} = \mathsf{u} + b_{j}\, * \mathsf{v}$
\EndFor
\State $ \mathsf{y_{n+1}} = \mathsf{y_{n}} + \mathsf{u}$
\State  {\bf Output:} $\mathsf{y_{n+1}}$\Comment{B-series coefficient of the next step numerical solution.}
\end{algorithmic}
\end{algorithm}

We recall the following definition:
\begin{definition}[Density of a tree \cite{Butcher_2001_book, Hairer_book_I}]
The density  $\gamma(\tau)$ of a tree $\tau$ is the product over all vertices of the orders of the sub-trees rooted at those vertices. 
\end{definition}

We have the following theorem.
\begin{theorem}[B-series of the exact solution]
\label{Theorem:TW-trees-coefficients-of-exact-solution}
The B-series expansion of the exact solution at the next step $y(t_{n}\,+h)$, performed over \textit{TW}-trees, has the following coefficients: 
\[
\mathsf{a}(\tau) = 
\begin{cases}
0 & \tau \in  TW \smallsetminus T, \\
\gamma(\tau) & \tau \in T.
\end{cases}
\]
\end{theorem}
\begin{proof}
First part of the proof follows from the observation that the elementary differentials in the B-series expansion of the exact solution cannot have the approximate Jacobian appearing anywhere in their expressions. 
Second part of the proof follows from \cite[Theorem 2.6, 2.11]{Hairer_book_I} and \cite[Subsection 302, 380]{Butcher_2001_book}.
\end{proof}

The EPIRK-$W$ order conditions are obtained by imposing that the B-series coefficients of the numerical solution match the B-series coefficients of the exact solution up to the desired order of accuracy. The order conditions for the three stages EPIRK-$W$ method  \eqref{eqn:EPIRK-W-Formulation-3-Stage} are given in Table \ref{Table:3-stage-W-order-condition}.

\begin{longtable}{|>{\centering\arraybackslash}m{1.25cm}|>{\centering\arraybackslash}m{1.2cm}|>{\centering\arraybackslash}m{11cm}|}
\caption{Order conditions for the three stage EPIRK-$W$ method}
\label{Table:3-stage-W-order-condition}\\
\toprule
\centering \textbf{Tree} &
\centering \textbf{Order} &
\centering \textbf{Order Condition: $B^{\#}(y_{n+1}) - B^{\#}\big(y(t_{n}\,+h)\big) = 0$} \tabularnewline
\midrule
\endfirsthead
\toprule
\centering \textbf{Tree} &
\centering \textbf{Order} &
\centering \textbf{Order Condition: $B^{\#}(y_{n+1}) - B^{\#}\big(y(t_{n}\,+h)\big) = 0$} \tabularnewline
\midrule
\endhead
${\tau^{W}_1}$ &1&\begin{dmath*}b_{1}\, p_{1,1}\,-1= 0 \label{cond:1} \end{dmath*}\\\hline
${\tau^{W}_2}$ & 2&\begin{dmath*}[breakdepth={3}]\frac{1}{6} (6 a_{1,1}\, b_{2}\, p_{1,1}\, p_{2,1}\,+3 a_{1,1}\, b_{2}\, p_{1,1}\, p_{2,2}\,-12 a_{1,1}\, b_{3}\, p_{1,1}\, p_{3,1}\,+6 a_{2,1}\,
   b_{3}\, p_{1,1}\, p_{3,1}\,-6 a_{1,1}\, b_{3}\, p_{1,1}\, p_{3,2}\,+3 a_{2,1}\, b_{3}\, p_{1,1}\, p_{3,2}\,-2 a_{1,1}\, b_{3}\, p_{1,1}\, p_{3,3}\,+a_{2,1}\, b_{3}\,
   p_{1,1}\, p_{3,3}\,-3)= 0\end{dmath*}\\\hline
${\tau^{W}_3}$ & 2&\begin{dmath*}[breakdepth={3}]\frac{1}{6} p_{1,1}\, (3 b_{1}\, g_{3,1}\,-6 a_{1,1}\, b_{2}\, p_{2,1}\,-3 a_{1,1}\, b_{2}\, p_{2,2}\,+12 a_{1,1}\, b_{3}\, p_{3,1}\,-6 a_{2,1}\,
   b_{3}\, p_{3,1}\,+6 a_{1,1}\, b_{3}\, p_{3,2}\,-3 a_{2,1}\, b_{3}\, p_{3,2}\,+2 a_{1,1}\, b_{3}\, p_{3,3}\,-a_{2,1}\, b_{3}\, p_{3,3})= 0\end{dmath*}\\\hline
${\tau^{W}_4}$ & 3&\begin{dmath*}[breakdepth={3}]\frac{1}{6} \left(6 a_{1,1}^2\,
   b_{2}\, p_{2,1}\, p_{1,1}^2\,+3 a_{1,1}^2\, b_{2}\, p_{2,2}\, p_{1,1}^2\,-12 a_{1,1}^2\, b_{3}\, p_{3,1}\, p_{1,1}^2\,+6 a_{2,1}^2\, b_{3}\, p_{3,1}\, p_{1,1}^2\,-6
   a_{1,1}^2\, b_{3}\, p_{3,2}\, p_{1,1}^2\,+3 a_{2,1}^2\, b_{3}\, p_{3,2}\, p_{1,1}^2\,-2 a_{1,1}^2\, b_{3}\, p_{3,3}\, p_{1,1}^2\,+a_{2,1}^2\, b_{3}\, p_{3,3}\,
   p_{1,1}^2\,-2\right)= 0\end{dmath*}\\\hline
${\tau^{W}_5}$&3&\begin{dmath*}[breakdepth={3}]\frac{1}{12} (12 a_{1,1}\, a_{2,2}\, b_{3}\, p_{1,1}\, p_{2,1}\, p_{3,1}\,+6 a_{1,1}\, a_{2,2}\, b_{3}\, p_{1,1}\, p_{2,2}\, p_{3,1}\,+6 a_{1,1}\,
   a_{2,2}\, b_{3}\, p_{1,1}\, p_{2,1}\, p_{3,2}\,+3 a_{1,1}\, a_{2,2}\, b_{3}\, p_{1,1}\, p_{2,2}\, p_{3,2}\,+2 a_{1,1}\, a_{2,2}\, b_{3}\, p_{1,1}\, p_{2,1}\,
   p_{3,3}\,+a_{1,1}\, a_{2,2}\, b_{3}\, p_{1,1}\, p_{2,2}\, p_{3,3}\,-2)= 0\end{dmath*}\\\hline
${\tau^{W}_6}$&3&\begin{dmath*}[breakdepth={3}]\frac{1}{12} p_{1,1}\, (6 a_{1,1}\, b_{2}\, g_{1,1}\, p_{2,1}\,-12 a_{1,1}\, a_{2,2}\, b_{3}\,
   p_{3,1}\, p_{2,1}\,-6 a_{1,1}\, a_{2,2}\, b_{3}\, p_{3,2}\, p_{2,1}\,-2 a_{1,1}\, a_{2,2}\, b_{3}\, p_{3,3}\, p_{2,1}\,+3 a_{1,1}\, b_{2}\, g_{1,1}\, p_{2,2}\,-12
   a_{1,1}\, b_{3}\, g_{1,1}\, p_{3,1}\,+6 a_{2,1}\, b_{3}\, g_{2,1}\, p_{3,1}\,-6 a_{1,1}\, a_{2,2}\, b_{3}\, p_{2,2}\, p_{3,1}\,-6 a_{1,1}\, b_{3}\, g_{1,1}\,
   p_{3,2}\,+3 a_{2,1}\, b_{3}\, g_{2,1}\, p_{3,2}\,-3 a_{1,1}\, a_{2,2}\, b_{3}\, p_{2,2}\, p_{3,2}\,-2 a_{1,1}\, b_{3}\, g_{1,1}\, p_{3,3}\,+a_{2,1}\, b_{3}\, g_{2,1}\,
   p_{3,3}\,-a_{1,1}\, a_{2,2}\, b_{3}\, p_{2,2}\, p_{3,3})= 0\end{dmath*}\\\hline
${\tau^{W}_7}$&3&\begin{dmath*}[breakdepth={3}]\frac{1}{24} p_{1,1}\, (12 a_{1,1}\, b_{2}\, g_{3,2}\, p_{2,1}\,-24 a_{1,1}\, a_{2,2}\, b_{3}\, p_{3,1}\,
   p_{2,1}\,-12 a_{1,1}\, a_{2,2}\, b_{3}\, p_{3,2}\, p_{2,1}\,-4 a_{1,1}\, a_{2,2}\, b_{3}\, p_{3,3}\, p_{2,1}\,+4 a_{1,1}\, b_{2}\, g_{3,2}\, p_{2,2}\,-24 a_{1,1}\,
   b_{3}\, g_{3,3}\, p_{3,1}\,+12 a_{2,1}\, b_{3}\, g_{3,3}\, p_{3,1}\,-12 a_{1,1}\, a_{2,2}\, b_{3}\, p_{2,2}\, p_{3,1}\,-8 a_{1,1}\, b_{3}\, g_{3,3}\, p_{3,2}\,+4
   a_{2,1}\, b_{3}\, g_{3,3}\, p_{3,2}\,-6 a_{1,1}\, a_{2,2}\, b_{3}\, p_{2,2}\, p_{3,2}\,-2 a_{1,1}\, b_{3}\, g_{3,3}\, p_{3,3}\,+a_{2,1}\, b_{3}\, g_{3,3}\, p_{3,3}\,-2
   a_{1,1}\, a_{2,2}\, b_{3}\, p_{2,2}\, p_{3,3})= 0\end{dmath*}\\\hline
${\tau^{W}_8}$&3&\begin{dmath*}[breakdepth={3}]\frac{1}{24} p_{1,1}\, \left(4 b_{1}\, g_{3,1}^2\,-12 a_{1,1}\, b_{2}\, g_{1,1}\, p_{2,1}\,-12 a_{1,1}\, b_{2}\,
   g_{3,2}\, p_{2,1}\,-6 a_{1,1}\, b_{2}\, g_{1,1}\, p_{2,2}\,-4 a_{1,1}\, b_{2}\, g_{3,2}\, p_{2,2}\,+24 a_{1,1}\, b_{3}\, g_{1,1}\, p_{3,1}\,-12 a_{2,1}\, b_{3}\,
   g_{2,1}\, p_{3,1}\,+24 a_{1,1}\, b_{3}\, g_{3,3}\, p_{3,1}\,-12 a_{2,1}\, b_{3}\, g_{3,3}\, p_{3,1}\,+24 a_{1,1}\, a_{2,2}\, b_{3}\, p_{2,1}\, p_{3,1}\,+12 a_{1,1}\,
   a_{2,2}\, b_{3}\, p_{2,2}\, p_{3,1}\,+12 a_{1,1}\, b_{3}\, g_{1,1}\, p_{3,2}\,-6 a_{2,1}\, b_{3}\, g_{2,1}\, p_{3,2}\,+8 a_{1,1}\, b_{3}\, g_{3,3}\, p_{3,2}\,-4
   a_{2,1}\, b_{3}\, g_{3,3}\, p_{3,2}\,+12 a_{1,1}\, a_{2,2}\, b_{3}\, p_{2,1}\, p_{3,2}\,+6 a_{1,1}\, a_{2,2}\, b_{3}\, p_{2,2}\, p_{3,2}\,+4 a_{1,1}\, b_{3}\, g_{1,1}\,
   p_{3,3}\,-2 a_{2,1}\, b_{3}\, g_{2,1}\, p_{3,3}\,+2 a_{1,1}\, b_{3}\, g_{3,3}\, p_{3,3}\,-a_{2,1}\, b_{3}\, g_{3,3}\, p_{3,3}\,+4 a_{1,1}\, a_{2,2}\, b_{3}\, p_{2,1}\,
   p_{3,3}\,+2 a_{1,1}\, a_{2,2}\, b_{3}\, p_{2,2}\, p_{3,3}\,\right)= 0\end{dmath*}\\\hline
${\tau^{W}_{9}}$&4&\begin{dmath*}[breakdepth={3}]\frac{1}{12} \left(12 a_{1,1}^3\, b_{2}\, p_{2,1}\, p_{1,1}^3\,+6 a_{1,1}^3\, b_{2}\, p_{2,2}\,
   p_{1,1}^3\,-24 a_{1,1}^3\, b_{3}\, p_{3,1}\, p_{1,1}^3\,+12 a_{2,1}^3\, b_{3}\, p_{3,1}\, p_{1,1}^3\,-12 a_{1,1}^3\, b_{3}\, p_{3,2}\, p_{1,1}^3\,+6 a_{2,1}^3\, b_{3}\,
   p_{3,2}\, p_{1,1}^3\,-4 a_{1,1}^3\, b_{3}\, p_{3,3}\, p_{1,1}^3\,+2 a_{2,1}^3\, b_{3}\, p_{3,3}\, p_{1,1}^3\,-3\right)= 0\end{dmath*}\\\hline
${\tau^{W}_{10}}$&4&\begin{dmath*}[breakdepth={3}]\frac{1}{24} \left(24 a_{1,1}\, a_{2,1}\, a_{2,2}\,
   b_{3}\, p_{2,1}\, p_{3,1}\, p_{1,1}^2\,+12 a_{1,1}\, a_{2,1}\, a_{2,2}\, b_{3}\, p_{2,2}\, p_{3,1}\, p_{1,1}^2\,+12 a_{1,1}\, a_{2,1}\, a_{2,2}\, b_{3}\, p_{2,1}\,
   p_{3,2}\, p_{1,1}^2\,+6 a_{1,1}\, a_{2,1}\, a_{2,2}\, b_{3}\, p_{2,2}\, p_{3,2}\, p_{1,1}^2\,+4 a_{1,1}\, a_{2,1}\, a_{2,2}\, b_{3}\, p_{2,1}\, p_{3,3}\, p_{1,1}^2\,+2
   a_{1,1}\, a_{2,1}\, a_{2,2}\, b_{3}\, p_{2,2}\, p_{3,3}\, p_{1,1}^2\,-3\right)= 0\end{dmath*}\\\hline
${\tau^{W}_{11}}$&4&\begin{dmath*}[breakdepth={3}]\frac{1}{12} p_{1,1}^2\, \left(6 b_{2}\, g_{1,1}\, p_{2,1}\, a_{1,1}^2\,+3 b_{2}\, g_{1,1}\,
   p_{2,2}\, a_{1,1}^2\,-12 b_{3}\, g_{1,1}\, p_{3,1}\, a_{1,1}^2\,-6 b_{3}\, g_{1,1}\, p_{3,2}\, a_{1,1}^2\,-2 b_{3}\, g_{1,1}\, p_{3,3}\, a_{1,1}^2\,-12 a_{2,1}\,
   a_{2,2}\, b_{3}\, p_{2,1}\, p_{3,1}\, a_{1,1}\,-6 a_{2,1}\, a_{2,2}\, b_{3}\, p_{2,2}\, p_{3,1}\, a_{1,1}\,-6 a_{2,1}\, a_{2,2}\, b_{3}\, p_{2,1}\, p_{3,2}\, a_{1,1}\,-3
   a_{2,1}\, a_{2,2}\, b_{3}\, p_{2,2}\, p_{3,2}\, a_{1,1}\,-2 a_{2,1}\, a_{2,2}\, b_{3}\, p_{2,1}\, p_{3,3}\, a_{1,1}\,-a_{2,1}\, a_{2,2}\, b_{3}\, p_{2,2}\, p_{3,3}\,
   a_{1,1}\,+6 a_{2,1}^2\, b_{3}\, g_{2,1}\, p_{3,1}\,+3 a_{2,1}^2\, b_{3}\, g_{2,1}\, p_{3,2}\,+a_{2,1}^2\, b_{3}\, g_{2,1}\, p_{3,3}\,\right)= 0\end{dmath*}\\\hline
${\tau^{W}_{12}}$&4&\begin{dmath*}[breakdepth={3}]\frac{1}{12} \left(12
   a_{1,1}^2\, a_{2,2}\, b_{3}\, p_{2,1}\, p_{3,1}\, p_{1,1}^2\,+6 a_{1,1}^2\, a_{2,2}\, b_{3}\, p_{2,2}\, p_{3,1}\, p_{1,1}^2\,+6 a_{1,1}^2\, a_{2,2}\, b_{3}\, p_{2,1}\,
   p_{3,2}\, p_{1,1}^2\,+3 a_{1,1}^2\, a_{2,2}\, b_{3}\, p_{2,2}\, p_{3,2}\, p_{1,1}^2\,+2 a_{1,1}^2\, a_{2,2}\, b_{3}\, p_{2,1}\, p_{3,3}\, p_{1,1}^2\,+a_{1,1}^2\, a_{2,2}\,
   b_{3}\, p_{2,2}\, p_{3,3}\, p_{1,1}^2\,-1\right)= 0\end{dmath*}\\\hline
${\tau^{W}_{1,3}}$&4&\begin{dmath*}[breakdepth={3}]\frac{1}{24} p_{1,1}^2\, \left(12 b_{2}\, g_{3,2}\, p_{2,1}\, a_{1,1}^2\,+4 b_{2}\, g_{3,2}\, p_{2,2}\, a_{1,1}^2\,-24
   b_{3}\, g_{3,3}\, p_{3,1}\, a_{1,1}^2\,-24 a_{2,2}\, b_{3}\, p_{2,1}\, p_{3,1}\, a_{1,1}^2\,-12 a_{2,2}\, b_{3}\, p_{2,2}\, p_{3,1}\, a_{1,1}^2\,-8 b_{3}\, g_{3,3}\,
   p_{3,2}\, a_{1,1}^2\,-12 a_{2,2}\, b_{3}\, p_{2,1}\, p_{3,2}\, a_{1,1}^2\,-6 a_{2,2}\, b_{3}\, p_{2,2}\, p_{3,2}\, a_{1,1}^2\,-2 b_{3}\, g_{3,3}\, p_{3,3}\, a_{1,1}^2\,-4
   a_{2,2}\, b_{3}\, p_{2,1}\, p_{3,3}\, a_{1,1}^2\,-2 a_{2,2}\, b_{3}\, p_{2,2}\, p_{3,3}\, a_{1,1}^2\,+12 a_{2,1}^2\, b_{3}\, g_{3,3}\, p_{3,1}\,+4 a_{2,1}^2\, b_{3}\,
   g_{3,3}\, p_{3,2}\,+a_{2,1}^2\, b_{3}\, g_{3,3}\, p_{3,3}\,\right)= 0\end{dmath*}\\\hline
${\tau^{W}_{14}}$&4&\begin{dmath*}[breakdepth={3}]-\frac{1}{24}\end{dmath*}\\\hline
${\tau^{W}_{15}}$&4&\begin{dmath*}[breakdepth={3}]\frac{1}{24} a_{1,1}\, a_{2,2}\, b_{3}\, g_{1,1}\, p_{1,1}\, (2 p_{2,1}\,+p_{2,2}) (6
   p_{3,1}\,+3 p_{3,2}\,+p_{3,3})= 0\end{dmath*}\\\hline
${\tau^{W}_{16}}$&4&\begin{dmath*}[breakdepth={3}]\frac{1}{36} a_{1,1}\, a_{2,2}\, b_{3}\, g_{2,2}\, p_{1,1}\, (3 p_{2,1}\,+p_{2,2}) (6 p_{3,1}\,+3 p_{3,2}\,+p_{3,3})= 0\end{dmath*}\\\hline
${\tau^{W}_{17}}$&4&\begin{dmath*}[breakdepth={3}]\frac{1}{72}
   p_{1,1}\, \left(12 a_{1,1}\, b_{2}\, p_{2,1}\, g_{1,1}^2\,+6 a_{1,1}\, b_{2}\, p_{2,2}\, g_{1,1}^2\,-24 a_{1,1}\, b_{3}\, p_{3,1}\, g_{1,1}^2\,-12 a_{1,1}\, b_{3}\,
   p_{3,2}\, g_{1,1}^2\,-4 a_{1,1}\, b_{3}\, p_{3,3}\, g_{1,1}^2\,-36 a_{1,1}\, a_{2,2}\, b_{3}\, p_{2,1}\, p_{3,1}\, g_{1,1}\,-18 a_{1,1}\, a_{2,2}\, b_{3}\, p_{2,2}\,
   p_{3,1}\, g_{1,1}\,-18 a_{1,1}\, a_{2,2}\, b_{3}\, p_{2,1}\, p_{3,2}\, g_{1,1}\,-9 a_{1,1}\, a_{2,2}\, b_{3}\, p_{2,2}\, p_{3,2}\, g_{1,1}\,-6 a_{1,1}\, a_{2,2}\,
   b_{3}\, p_{2,1}\, p_{3,3}\, g_{1,1}\,-3 a_{1,1}\, a_{2,2}\, b_{3}\, p_{2,2}\, p_{3,3}\, g_{1,1}\,+12 a_{2,1}\, b_{3}\, g_{2,1}^2\, p_{3,1}\,-36 a_{1,1}\, a_{2,2}\,
   b_{3}\, g_{2,2}\, p_{2,1}\, p_{3,1}\,-12 a_{1,1}\, a_{2,2}\, b_{3}\, g_{2,2}\, p_{2,2}\, p_{3,1}\,+6 a_{2,1}\, b_{3}\, g_{2,1}^2\, p_{3,2}\,-18 a_{1,1}\, a_{2,2}\,
   b_{3}\, g_{2,2}\, p_{2,1}\, p_{3,2}\,-6 a_{1,1}\, a_{2,2}\, b_{3}\, g_{2,2}\, p_{2,2}\, p_{3,2}\,+2 a_{2,1}\, b_{3}\, g_{2,1}^2\, p_{3,3}\,-6 a_{1,1}\, a_{2,2}\,
   b_{3}\, g_{2,2}\, p_{2,1}\, p_{3,3}\,-2 a_{1,1}\, a_{2,2}\, b_{3}\, g_{2,2}\, p_{2,2}\, p_{3,3}\,\right)= 0\end{dmath*}\\\hline
${\tau^{W}_{18}}$&4&\begin{dmath*}[breakdepth={3}]\frac{1}{48} a_{1,1}\, a_{2,2}\, b_{3}\, g_{3,3}\, p_{1,1}\, (2
   p_{2,1}\,+p_{2,2}) (12 p_{3,1}\,+4 p_{3,2}\,+p_{3,3})= 0\end{dmath*}\\\hline
${\tau^{W}_{19}}$&4&\begin{dmath*}[breakdepth={3}]\frac{1}{48} p_{1,1}\, (12 a_{1,1}\, b_{2}\, g_{1,1}\, g_{3,2}\, p_{2,1}\,-24 a_{1,1}\, a_{2,2}\, b_{3}\,
   g_{1,1}\, p_{3,1}\, p_{2,1}\,-24 a_{1,1}\, a_{2,2}\, b_{3}\, g_{3,3}\, p_{3,1}\, p_{2,1}\,-12 a_{1,1}\, a_{2,2}\, b_{3}\, g_{1,1}\, p_{3,2}\, p_{2,1}\,-8 a_{1,1}\,
   a_{2,2}\, b_{3}\, g_{3,3}\, p_{3,2}\, p_{2,1}\,-4 a_{1,1}\, a_{2,2}\, b_{3}\, g_{1,1}\, p_{3,3}\, p_{2,1}\,-2 a_{1,1}\, a_{2,2}\, b_{3}\, g_{3,3}\, p_{3,3}\, p_{2,1}\,+4
   a_{1,1}\, b_{2}\, g_{1,1}\, g_{3,2}\, p_{2,2}\,-24 a_{1,1}\, b_{3}\, g_{1,1}\, g_{3,3}\, p_{3,1}\,+12 a_{2,1}\, b_{3}\, g_{2,1}\, g_{3,3}\, p_{3,1}\,-12 a_{1,1}\,
   a_{2,2}\, b_{3}\, g_{1,1}\, p_{2,2}\, p_{3,1}\,-12 a_{1,1}\, a_{2,2}\, b_{3}\, g_{3,3}\, p_{2,2}\, p_{3,1}\,-8 a_{1,1}\, b_{3}\, g_{1,1}\, g_{3,3}\, p_{3,2}\,+4
   a_{2,1}\, b_{3}\, g_{2,1}\, g_{3,3}\, p_{3,2}\,-6 a_{1,1}\, a_{2,2}\, b_{3}\, g_{1,1}\, p_{2,2}\, p_{3,2}\,-4 a_{1,1}\, a_{2,2}\, b_{3}\, g_{3,3}\, p_{2,2}\, p_{3,2}\,-2
   a_{1,1}\, b_{3}\, g_{1,1}\, g_{3,3}\, p_{3,3}\,+a_{2,1}\, b_{3}\, g_{2,1}\, g_{3,3}\, p_{3,3}\,-2 a_{1,1}\, a_{2,2}\, b_{3}\, g_{1,1}\, p_{2,2}\, p_{3,3}\,-a_{1,1}\,
   a_{2,2}\, b_{3}\, g_{3,3}\, p_{2,2}\, p_{3,3})= 0\end{dmath*}\\\hline
${\tau^{W}_{20}}$&4&\begin{dmath*}[breakdepth={3}]\frac{1}{720} p_{1,1}\, \left(120 a_{1,1}\, b_{2}\, p_{2,1}\, g_{3,2}^2\,+30 a_{1,1}\, b_{2}\, p_{2,2}\, g_{3,2}^2\,-240
   a_{1,1}\, b_{3}\, g_{3,3}^2\, p_{3,1}\,+120 a_{2,1}\, b_{3}\, g_{3,3}^2\, p_{3,1}\,-360 a_{1,1}\, a_{2,2}\, b_{3}\, g_{2,2}\, p_{2,1}\, p_{3,1}\,-360 a_{1,1}\, a_{2,2}\,
   b_{3}\, g_{3,3}\, p_{2,1}\, p_{3,1}\,-120 a_{1,1}\, a_{2,2}\, b_{3}\, g_{2,2}\, p_{2,2}\, p_{3,1}\,-180 a_{1,1}\, a_{2,2}\, b_{3}\, g_{3,3}\, p_{2,2}\, p_{3,1}\,-60
   a_{1,1}\, b_{3}\, g_{3,3}^2\, p_{3,2}\,+30 a_{2,1}\, b_{3}\, g_{3,3}^2\, p_{3,2}\,-180 a_{1,1}\, a_{2,2}\, b_{3}\, g_{2,2}\, p_{2,1}\, p_{3,2}\,-120 a_{1,1}\, a_{2,2}\,
   b_{3}\, g_{3,3}\, p_{2,1}\, p_{3,2}\,-60 a_{1,1}\, a_{2,2}\, b_{3}\, g_{2,2}\, p_{2,2}\, p_{3,2}\,-60 a_{1,1}\, a_{2,2}\, b_{3}\, g_{3,3}\, p_{2,2}\, p_{3,2}\,-12
   a_{1,1}\, b_{3}\, g_{3,3}^2\, p_{3,3}\,+6 a_{2,1}\, b_{3}\, g_{3,3}^2\, p_{3,3}\,-60 a_{1,1}\, a_{2,2}\, b_{3}\, g_{2,2}\, p_{2,1}\, p_{3,3}\,-30 a_{1,1}\, a_{2,2}\,
   b_{3}\, g_{3,3}\, p_{2,1}\, p_{3,3}\,-20 a_{1,1}\, a_{2,2}\, b_{3}\, g_{2,2}\, p_{2,2}\, p_{3,3}\,-15 a_{1,1}\, a_{2,2}\, b_{3}\, g_{3,3}\, p_{2,2}\,
   p_{3,3}\,\right)= 0\end{dmath*}\\\hline
${\tau^{W}_{21}}$&4&\begin{dmath*}[breakdepth={3}]\frac{1}{720} p_{1,1}\, \left(30 b_{1}\, g_{3,1}^3\,-120 a_{1,1}\, b_{2}\, g_{1,1}^2\, p_{2,1}\,-120 a_{1,1}\, b_{2}\, g_{3,2}^2\, p_{2,1}\,-180 a_{1,1}\,
   b_{2}\, g_{1,1}\, g_{3,2}\, p_{2,1}\,-60 a_{1,1}\, b_{2}\, g_{1,1}^2\, p_{2,2}\,-30 a_{1,1}\, b_{2}\, g_{3,2}^2\, p_{2,2}\,-60 a_{1,1}\, b_{2}\, g_{1,1}\, g_{3,2}\,
   p_{2,2}\,+240 a_{1,1}\, b_{3}\, g_{1,1}^2\, p_{3,1}\,-120 a_{2,1}\, b_{3}\, g_{2,1}^2\, p_{3,1}\,+240 a_{1,1}\, b_{3}\, g_{3,3}^2\, p_{3,1}\,-120 a_{2,1}\, b_{3}\,
   g_{3,3}^2\, p_{3,1}\,+360 a_{1,1}\, b_{3}\, g_{1,1}\, g_{3,3}\, p_{3,1}\,-180 a_{2,1}\, b_{3}\, g_{2,1}\, g_{3,3}\, p_{3,1}\,+360 a_{1,1}\, a_{2,2}\, b_{3}\, g_{1,1}\,
   p_{2,1}\, p_{3,1}\,+360 a_{1,1}\, a_{2,2}\, b_{3}\, g_{2,2}\, p_{2,1}\, p_{3,1}\,+360 a_{1,1}\, a_{2,2}\, b_{3}\, g_{3,3}\, p_{2,1}\, p_{3,1}\,+180 a_{1,1}\, a_{2,2}\,
   b_{3}\, g_{1,1}\, p_{2,2}\, p_{3,1}\,+120 a_{1,1}\, a_{2,2}\, b_{3}\, g_{2,2}\, p_{2,2}\, p_{3,1}\,+180 a_{1,1}\, a_{2,2}\, b_{3}\, g_{3,3}\, p_{2,2}\, p_{3,1}\,+120
   a_{1,1}\, b_{3}\, g_{1,1}^2\, p_{3,2}\,-60 a_{2,1}\, b_{3}\, g_{2,1}^2\, p_{3,2}\,+60 a_{1,1}\, b_{3}\, g_{3,3}^2\, p_{3,2}\,-30 a_{2,1}\, b_{3}\, g_{3,3}^2\,
   p_{3,2}\,+120 a_{1,1}\, b_{3}\, g_{1,1}\, g_{3,3}\, p_{3,2}\,-60 a_{2,1}\, b_{3}\, g_{2,1}\, g_{3,3}\, p_{3,2}\,+180 a_{1,1}\, a_{2,2}\, b_{3}\, g_{1,1}\, p_{2,1}\,
   p_{3,2}\,+180 a_{1,1}\, a_{2,2}\, b_{3}\, g_{2,2}\, p_{2,1}\, p_{3,2}\,+120 a_{1,1}\, a_{2,2}\, b_{3}\, g_{3,3}\, p_{2,1}\, p_{3,2}\,+90 a_{1,1}\, a_{2,2}\, b_{3}\,
   g_{1,1}\, p_{2,2}\, p_{3,2}\,+60 a_{1,1}\, a_{2,2}\, b_{3}\, g_{2,2}\, p_{2,2}\, p_{3,2}\,+60 a_{1,1}\, a_{2,2}\, b_{3}\, g_{3,3}\, p_{2,2}\, p_{3,2}\,+40 a_{1,1}\,
   b_{3}\, g_{1,1}^2\, p_{3,3}\,-20 a_{2,1}\, b_{3}\, g_{2,1}^2\, p_{3,3}\,+12 a_{1,1}\, b_{3}\, g_{3,3}^2\, p_{3,3}\,-6 a_{2,1}\, b_{3}\, g_{3,3}^2\, p_{3,3}\,+30
   a_{1,1}\, b_{3}\, g_{1,1}\, g_{3,3}\, p_{3,3}\,-15 a_{2,1}\, b_{3}\, g_{2,1}\, g_{3,3}\, p_{3,3}\,+60 a_{1,1}\, a_{2,2}\, b_{3}\, g_{1,1}\, p_{2,1}\, p_{3,3}\,+60
   a_{1,1}\, a_{2,2}\, b_{3}\, g_{2,2}\, p_{2,1}\, p_{3,3}\,+30 a_{1,1}\, a_{2,2}\, b_{3}\, g_{3,3}\, p_{2,1}\, p_{3,3}\,+30 a_{1,1}\, a_{2,2}\, b_{3}\, g_{1,1}\, p_{2,2}\,
   p_{3,3}\,+20 a_{1,1}\, a_{2,2}\, b_{3}\, g_{2,2}\, p_{2,2}\, p_{3,3}\,+15 a_{1,1}\, a_{2,2}\, b_{3}\, g_{3,3}\, p_{2,2}\, p_{3,3}\,\right) = 0\end{dmath*}\\ \hline
\end{longtable}

\subsection{Construction of practical EPIRK-$W$ integrators}

We will focus on constructing EPIRK-$W$ methods \eqref{eqn:EPIRK-W-Formulation} with three stages. Such a method using an approximate Jacobian reads:
\begin{equation}
\label{eqn:EPIRK-W-Formulation-3-Stage}
\begin{split}
Y_1 &= y_{n}\, + a_{1,1}\, \psi_{1,1}(g_{1,1}\,h\,\mathbf{A}_{n})\,hf(y_{n}), \\
Y_2 &= y_{n}\, + a_{2,1}\, \psi_{2,1}(g_{2,1}\,h\,\mathbf{A}_{n})\,hf(y_{n}) +  a_{2,2}\, \psi_{2,2}(g_{2,2}\,h\,\mathbf{A}_{n})\,h\Delta^{(1)}r(y_{n}), \\
y_{n+1} &= y_{n}\, + b_{1}\, \psi_{3,1}(g_{3,1}\,h\,\mathbf{A}_{n})\,hf(y_{n}) + b_{2}\, \psi_{3,2}(g_{3,2}\,h\,\mathbf{A}_{n})\,h\Delta^{(1)}r(y_{n}) \\
& \qquad~ +  b_{3}\, \psi_{3,3}(g_{3,3}\,h\,\mathbf{A}_{n})\, h\Delta^{(2)}r(y_{n}).
\end{split}
\end{equation}

An important observation from Table \ref{Table:3-stage-W-order-condition} is that the difference between the B-series coefficients of the exact solution and of the numerical solution corresponding to $\tau^{W}_{14}$ is equal to $-1/24$ and cannot be zeroed out. Since $\tau^{W}_{14}$ has only meagre nodes  it  appears in the B-series expansions of both the exact solution and the numerical solution. The conclusion is that three stage EPIRK-$W$ methods with the simplified choice of $\psi_{i,j}(z)$ given in equation \eqref{eqn:simplified_psi_function_defn} cannot achieve order four. Consequently , we will limit our solution procedure to constructing third order EPIRK-$W$ methods by zeroing out the first eight order conditions corresponding to \textit{TW}-trees of up to order three.  We use \Mathematica to solve the order conditions using two different approaches, as discussed next.

\paragraph{First approach to solving the order conditions}
In the first approach we seek to make the terms of each order condition as similar to another as possible. We start the solution process by noticing that the first order condition trivially reduces to the substitution $b_1 \rightarrow {1}/{p_{1,1}}$. The substitution $p_{3,2} \rightarrow -2p_{3,1}\,$ will zero out the following four terms: $-12 a_{1,1}\, b_{3}\, p_{1,1}\, p_{3,1}\,+6 a_{2,1}\,
   b_{3}\, p_{1,1}\, p_{3,1}\,-6 a_{1,1}\, b_{3}\, p_{1,1}\, p_{3,2}\,+3 a_{2,1}\, b_{3}\, p_{1,1}\, p_{3,2}\,$, in order conditions $\tau^{W}_{2}$ and $\tau^{W}_{3}$ and hence reducing the complexity of the two conditions. It is immediately observed that $g_{3,1} \rightarrow 1$ as all the other terms in order conditions $\tau^{W}_{2}$ and $\tau^{W}_{3}$ are the same. Additionally, we make a choice that $g_{3,3} \rightarrow 0$. After making the substitutions, we solve order condition $\tau^{W}_{2}$ or $\tau^{W}_{3}$ to get,
\begin{equation*}p_{3,3}\, \to \frac{3 (2 a_{1,1}\, b_{2}\, p_{1,1}\, p_{2,1}\,+a_{1,1}\, b_{2}\, p_{1,1}\, p_{2,2}\,-1)}{b_{3}\, p_{1,1}\, (2 a_{1,1}\,-a_{2,1})}.
\end{equation*}
This substitution results in a number of terms in multiple order conditions having the expression $(2a_{1,1}\,-a_{2,1})$ in the denominator. So we arbitrarily choose $a_{2,1} \rightarrow 2 a_{1,1}\, - 1$. The order conditions are far simpler now with several terms having the factor $(-1 + 2a_{1,1})$ and its powers in their expression. We choose $a_{1,1} \rightarrow 1/2$. Following this substitution, we can solve from order condition $\tau^{W}_{4}$ that $p_{1,1} \rightarrow 4/3$. After plugging in the value for $p_{1,1}\,$, order conditions $\tau^{W}_{5}$ and $\tau^{W}_{6}$ suggest that $g_{1,1} \rightarrow 2/3$. Now order conditions $\tau^{W}_{5}$, $\tau^{W}_{6}$, $\tau^{W}_{7}$ and $\tau^{W}_{8}$ have the following coefficients as part of their expressions: $a_{2,2}\,$, $p_{2,1}\,$, $p_{2,2}\,$, $g_{3,2}\,$ and $b_{2}\,$. We arbitrarily set $p_{2,2} \rightarrow 2p_{2,1}\,$ and solve for the remaining coefficients to obtain
\[
b_{2}\,\to \frac{3 a_{2,2}\,p_{2,1}\,+1}{8 a_{2,2}\,p_{2,1}^2} \quad \textnormal{and} \quad g_{3,2}\,\to \frac{12\, a_{2,2}\,p_{2,1}}{5 (3 \,a_{2,2}\,p_{2,1}\,+1)}.
\]
We now select arbitrary values for some of the remaining coefficients. We choose $p_{2,1}\, \to 1$, $a_{2,2}\, \to 1$, $b_3 \to 1$,  $a_{1,2}\, \to 0$, $a_{1,3}\, \to 0$, $a_{2,3}\, \to 0$, $p_{3,1}\, \to 0$, $p_{2,3}\, \to 0$, $p_{1,2}\, \to 0$, $p_{1,3}\, \to 0$, $g_{1,2}\, \to 0$, $g_{1,3}\, \to 0$, $g_{2,1}\, \to 0$, $g_{2,2}\, \to 0$, and $g_{2,3}\, \to 0$.

In order to solve for a second order embedded method we rewrite the final stage of the method with new coefficients $\widehat{b}_i$. Plugging in the substitutions that have been arrived at while solving the order conditions for the third order method results in a set of new conditions in only the $\widehat{b}_i$. It is again straightforward to observe that $\widehat{b}_{1} \to {1}/{p_{1,1}}$. We next solve the conditions for \textit{TW}-trees up to order = 2. Solving order conditions $\tau^{W}_{2}$ and $\tau^{W}_{3}$ we get $\widehat{b}_3 \to -3 + 8\, \widehat{b}_2$.

The resulting coefficients for a third-order EPIRK-$W$ method ({\sc epirkw3a}) with an embedded second-order method are given in Figure \ref{fig:EPIRK-w-coefficients-1}.
\begin{figure}[tbh]
\caption{Coefficients for {\sc epirkw3a}\label{fig:EPIRK-w-coefficients-1}}
\begin{equation*}
\renewcommand{\arraystretch}{1.5}
a = \begin{bmatrix}
\frac{1}{2} & 0 & 0\\
0 & 1& 0
\end{bmatrix}, \quad
\begin{bmatrix} b \\ \widehat{b} \end{bmatrix} = \begin{bmatrix}
\frac{3}{4} & \frac{1}{2} & 1\\[3pt]
\frac{3}{4} & \frac{3}{4} & \frac{6}{5}
\end{bmatrix},  \quad
g = \begin{bmatrix}
\frac{2}{3} & 0 & 0\\
0 & 0 & 0\\
1 & \frac{3}{5} & 0
\end{bmatrix},   \quad
p = \begin{bmatrix}
\frac{4}{3} & 0 & 0\\
1 & 2 & 0\\
0 & 0 & \frac{3}{4}
\end{bmatrix}.
\end{equation*}
\end{figure}
The choices that we have made for the coefficients $a$'s, $b$'s,  $g$'s and  $p$'s result in the sum of coefficients on trees $\tau^{W}_{5}$, $\tau^{W}_{6}$, $\tau^{W}_{7}$ and $\tau^{W}_{8}$ to sum to zero in the embedded method no matter what value is chosen for $\widehat{b}_2$. And when we choose $\mathbf{A}_{n} = \mathbf{J}_{n}$, the exact Jacobian, trees $\tau^{W}_{5}$, $\tau^{W}_{6}$, $\tau^{W}_{7}$ and $\tau^{W}_{8}$ are the same tree and it turns out that we incidentally solve the embedded method up to third order as well. To work around this restriction, we need to have the $a$'s,  $g$'s and  $p$'s to be independent of the $b$'s and to solve separately for the $\widehat{b}$'s to ensure that we get a second-order embedded method.

\paragraph{Second approach to solving the order conditions}

In this approach we impose a horizontal structure on the $g$ coefficients akin to the horizontally adaptive method described in \cite{Rainwater_Tokman_2016}. In order to impose a horizontal structure on $g$ we make the following substitutions: $g_{1,2}\, \to g_{1,1}\,$, $g_{1,3}\, \to g_{1,1}\,$, $g_{2,2}\, \to g_{2,1}\,$, $g_{2,3}\, \to g_{2,1}\,$, $g_{3,2}\, \to g_{3,1}\,$, and $g_{3,3}\, \to g_{3,1}\,$. We also note that the first order condition reduces to $b_1 \to {1}/{p_{1,1}}$ again. Order-conditions $\tau^{W}_{2}$ and $\tau^{W}_{3}$ imply $g_{3,1}\, \to 1$. We solve order-conditions $\tau^{W}_{2}$ and $\tau^{W}_{3}$ and get two possible solutions for a subset of the variables. We plugin the first solution and solve order conditions $\tau^{W}_{4}$, $\tau^{W}_{5}$ and $\tau^{W}_{6}$. We again get multiple solutions and we use the first solution to substitute into the order conditions and proceed to solve order conditions $\tau^{W}_{7}$ and $\tau^{W}_{8}$. We again get multiple solutions for the remaining coefficients. One of the solutions returned by \Mathematica gives a non-singular system, and using it leads to the following family of order three method coefficients:
\begin{equation}
\begin{split}
g_{1,2}\,\to g_{1,1}\,,
g_{1,3}\,\to g_{1,1}\,,
g_{2,2}\,\to g_{2,1}\,,
g_{2,3}\,\to g_{2,1}\,,
g_{3,2}\,\to g_{3,1}\,,
g_{3,3}\,\to g_{3,1}\,,\\
b_{1}\, \to \dfrac{1}{p_{1,1}}, b_{2}\,\to -\dfrac{a_{2,2}\, (3 a_{2,1}\, p_{1,1}\,-4)}{2 a_{2,1}^2\,
   p_{1,1}^2},
b_{3}\,\to -\dfrac{3 a_{2,1}\, p_{1,1}\,-4}{a_{2,1}^2\, p_{1,1}^2\, (6 p_{3,1}\,+3 p_{3,2}\,+p_{3,3})},\\
p_{2,1}\,\to 0, a_{1,1}\,\to \dfrac{a_{2,1}}{2}, g_{3,1}\,\to 1, g_{2,1}\,\to \dfrac{2 (3
   a_{2,1}\, g_{1,1}\, p_{1,1}\,-a_{2,1}\, p_{1,1}\,-2 g_{1,1})}{3 a_{2,1}\, p_{1,1}\,-4},\\
   p_{2,2}\,\to -\dfrac{2 \left(9 a_{1,1}\,
   a_{2,1}\, a_{2,2}\, p_{1,1}\, p_{2,1}\,-6 a_{1,1}\, a_{2,1}\, p_{1,1}\,-12 a_{1,1}\, a_{2,2}\, p_{2,1}\,+8 a_{1,1}\,+6 a_{2,1}^2\, p_{1,1}\,-4 a_{2,1}\,\right)}{3 a_{1,1}\,
   a_{2,2}\, (3 a_{2,1}\, p_{1,1}\,-4)}.
\end{split}
\end{equation}
After making the above substitutions in the embedded method, we obtain the following solutions for the embedded coefficients:
\[
\widehat{b}_{1} \to \frac{1}{p_{1,1}}, \quad \widehat{b}_{2}\to -\frac{a_{2,2}\, (3 a_{2,1}\, p_{1,1}\,-4)}{2 a_{2,1}^2\, p_{1,1}^2},
\]
with $\widehat{b}_{3}$ a free parameter that can be chosen to obtain a second order embedded method.

A solution to the above family leads to the third order EPIRK-$W$ method ({\sc epirkw3b}) in Figure \ref{fig:epirk-w-coefficients-2}, with a second order embedded  method that overcomes the limitation of the embedded method of {\sc eprikw3a}.
\begin{figure}[tbh]
\caption{Coefficients for {\sc epirkw3b}\label{fig:epirk-w-coefficients-2}}
\begin{equation*}
\begin{split}
a &= \begin{bmatrix}
0.22824182961171620396 & 0 & 0\\
0.45648365922343240794 & 0.33161664063356950085 & 0
\end{bmatrix},\\
b &= \begin{bmatrix}
1 \\ 2.0931591383832578214 \\ 1.2623969257900804404
\end{bmatrix}^{T},\\
\widehat{b} &= \begin{bmatrix}
1 \\ 2.0931591383832578214 \\ 1
\end{bmatrix}^{T},\\
g &= \begin{bmatrix}
0 & 0 & 0\\
0.34706341174296320958 & 0.34706341174296320958 & 0.34706341174296320958\\
1 & 1 & 1
\end{bmatrix},\\
p &= \begin{bmatrix}
1 & 0 & 0\\
0 & 2.0931604100438501004 & 0\\
1 & 1 & 1
\end{bmatrix}.
\end{split} 
\end{equation*}
\end{figure}
%

\section{EPIRK-$K$ methods}\label{sec:EPIRK-K}

$W$-methods have the advantage that any approximation of the Jacobian that ensures stability can be used, therefore they have the potential for attaining excellent computational efficiency in comparison to methods that require exact Jacobians. The drawback of the $W$-methods, however, is the very large number of order conditions that need to be solved for constructing the  method  \cite[Section IV.7]{Hairer_book_II}.

In order to reduce the number of order conditions Krylov-based methods ($K$-methods) were developed for Rosenbrock integrators in \cite{Tranquilli_2014_ROK} and for exponential integrators in \cite{Tranquilli_2014_ExpK}.  $K$-methods are built in the framework of $W$-methods and use a very specific approximation to the Jacobian constructed in the Krylov space. This approximation allows  \textit{TW}-trees with linear sub-trees having fat root to be re-colored (as meagre), leading to the new class of \textit{TK}-trees. The recoloring results in substantially fewer trees, and therefore fewer order conditions for the $K$-methods \cite[Lemma 3.2, 3.3]{Tranquilli_2014_ROK}.

In this section we derive $K$-methods in the framework of EPIRK integrators.

\subsection{Krylov-subspace approximation of Jacobian}

$K$-methods build the Krylov-subspace \textit{M}-dimensional Krylov-subspace $\mathcal{K}_M$, $M \ll N$, based on the exact Jacobian $\mathbf{J}_n$ and the ODE function value $f_n$ at the current time step $t_n$:
\begin{eqnarray}
\mathcal{K}_M = \text{span}\{f_n, \mathbf{J}_{n}f_n, \mathbf{J}_{n}^2f_n, \hdots, \mathbf{J}_{n}^{M-1}f_n\}.
\end{eqnarray}
The modified Arnoldi iteration \cite{van_der_Vorst_2003_Krylov} computes an orthonormal basis $\mathbf{V}$ and upper-Hessenberg matrix $\mathbf{H}$ such that:
\begin{eqnarray}
\mathcal{K}_M = \text{span}\{v_1, \hdots, v_M\}, \quad
\mathbf{V} = [v_1,  \hdots, v_M] \in \mathbb{R}^{N \times M}, \quad
\mathbf{V}^{T}\mathbf{V} = \mathbf{I}_M, \quad
\mathbf{H} = \mathbf{V}^T \mathbf{J}_{n} \mathbf{V} \in \mathbb{R}^{M \times M}.
\end{eqnarray}
Using $\mathbf{H}$ and $\mathbf{V}$ the approximation of the Jacobian matrix $\mathbf{J}_{n}$ in the Krylov sub-space is defined  as
\begin{eqnarray}
\label{eqn:approximate_Jacobian_in_Krylov}
\mathbf{A}_{n} = \mathbf{V}\,\mathbf{H}\,\mathbf{V}^T = \mathbf{V} \mathbf{V}^T \mathbf{J}_{n} \mathbf{V} \mathbf{V}^T \in \mathbb{R}^{N \times N}.
\end{eqnarray}
This approximation of the Jacobian is used by $K$-methods. The important properties of this approximation are given by the following Lemmas from \cite{Tranquilli_2014_ExpK}.

\begin{lemma}[Powers of $\mathbf{A}_{n}$ \cite{Tranquilli_2014_ExpK}]
\label{Lemma:powers_of_a0}
Powers of $\mathbf{A}_{n}$: $\mathbf{A}_{n}^k = \mathbf{V} \mathbf{H}^k \mathbf{V}^T$ for any $k \ge 1$.
\end{lemma}

\begin{lemma}[Evaluation of $\varphi$ functions \eqref{eqn:phi_k} on the approximate Jacobian \eqref{eqn:approximate_Jacobian_in_Krylov} \cite{Tranquilli_2014_ExpK}]
\label{Lemma:phik_reduced}
We have that:
\[
\varphi_k(h\,\gamma\,\mathbf{A}_{n}) = \frac{1}{k!}\left(\mathbf{I}_N - \mathbf{V}\mathbf{V}^T\right) + \mathbf{V}\, \varphi_k(h\,\gamma\,\mathbf{H})\, \mathbf{V}^T, \quad k = 1, 2, \hdots .
\]
\end{lemma}

In order to derive the Krylov formulation of the EPIRK methods \eqref{eqn:EPIRK-W-Formulation-3-Stage} we need to extend Lemma \ref{Lemma:phik_reduced} to the evaluation of $\psi$ functions \eqref{eqn:simplified_psi_function_defn}. We have the following result.

\begin{lemma}[Evaluation of $\psi$ functions \eqref{eqn:simplified_psi_function_defn} on the approximate Jacobian \eqref{eqn:approximate_Jacobian_in_Krylov}]
\label{Lemma:psik_reduced}
\[
\psi_j(h\,\gamma\,\mathbf{A}_{n}) = \widetilde{p}_j \left(\mathbf{I}_N - \mathbf{V}\mathbf{V}^{T}\right) +  \mathbf{V}\, \psi_{j}\left(h\gamma \mathbf{H}\right)\, \mathbf{V}^{T}, \quad j = 1, 2, \hdots \quad
\textnormal{where}\quad \widetilde{p}_j = \displaystyle\sum_{k=1}^{j}\frac{p_{j,k}}{k!}.
\]
\end{lemma}
\begin{proof}
From the definition \eqref{eqn:simplified_psi_function_defn} we have:
\[
\begin{split}
\psi_{j}(h\gamma \mathbf{A}_{n}) &= \displaystyle\sum_{k=1}^{j} p_{j,k}\, \varphi_k(h\,\gamma\,\mathbf{A}_{n}) \\
&= \displaystyle\sum_{k=1}^{j} p_{j,k}\,  \bigg[ \frac{1}{k!} (\mathbf{I}_N - \mathbf{V}\mathbf{V}^{T}) + \mathbf{V}\,  \varphi_k(h\,\gamma\,\mathbf{H})\,  \mathbf{V}^{T} \bigg] \\
 &=  (\mathbf{I}_N - \mathbf{V}\mathbf{V}^{T}) \, \bigg[\displaystyle\sum_{k=1}^{j} \frac{p_{j,k}}{k!} \bigg] + \mathbf{V}\,  \psi_{j}(h\gamma \mathbf{H})\,  \mathbf{V}^{T} \\
 &= \widetilde{p}_j (\mathbf{I}_N - \mathbf{V}\mathbf{V}^{T}) +  \mathbf{V}\, \psi_{j}(h\gamma \mathbf{H})\, \mathbf{V}^{T}.
\end{split}
\]
\end{proof}

\subsection{Formulation of EPIRK-$K$ methods}
\label{sec:EPIRK-K-formulation}

We now derive the Krylov-subspace formulation of the EPIRK methods, which will be called  EPIRK-$K$ methods. For this we begin with the  EPIRK-$W$ formulation \eqref{eqn:EPIRK-W-Formulation} and use the Jacobian approximation \eqref{eqn:approximate_Jacobian_in_Krylov}.
The first step in the method derivation is to split all vectors appearing in the method formulation \eqref{eqn:EPIRK-W-Formulation} into components within the Krylov-subspace and components orthogonal to it, as follows:
\begin{enumerate}
\item Splitting the internal stage vectors leads to:
\begin{equation}
\label{eqn:Stage_vector_In_Krylov_subspace}
Y_i = \mathbf{V}\lambda_i + Y_i^{\bot} \qquad \textnormal{where} \quad
\mathbf{V}^{T} Y_i = \lambda_i, \quad
\left(\mathbf{I}_N - \mathbf{V}\mathbf{V}^{T}\right)\,Y_i = Y_i^{\bot},
\end{equation}
where $Y_0 \equiv y_n$.
\item  Splitting the right-hand side function evaluated at the internal stage vectors gives:
\begin{equation}
\label{eqn:F_Function_In_Krylov_subspace}
f_i :=  f(Y_i) = \mathbf{V}\eta_i + f_i^{\bot} \qquad \textnormal{where} \quad
\mathbf{V}^{T} f_i = \eta_i, \quad
\left(\mathbf{I}_N - \mathbf{V}\mathbf{V}^{T}\right) f_i = f_i^{\bot},
\end{equation}
where $f_0 \equiv f(y_n)$.
\item  Splitting the non-linear Taylor remainder terms of the right-hand side functions yields:
\begin{equation} \label{eqn:R_Function_In_Krylov_subspace}
\begin{split}
r(Y_i) &= f(Y_i) - f(y_{n}) - \mathbf{A}_{n}\, (Y_i - y_{n})  = f_i - f_0 - \mathbf{V}\, \mathbf{H}\, \mathbf{V}^{T}\, (Y_i - y_{n}), \\
\textnormal{where} \quad \mathbf{V}^{T}\, r(Y_i)  &= \eta_i - \eta_0 - \mathbf{H} \, (\lambda_i - \lambda_0), \\
\left(\mathbf{I}_N - \mathbf{V}\mathbf{V}^{T}\right)\, r(Y_i)  &= f_i^{\bot} - f_0^{\bot}.
\end{split}
\end{equation}
\item Splitting the forward differences of the non-linear remainder terms leads to:
\begin{equation}
\label{eqn:Delta_r_Function_In_Krylov_subspace}
\begin{split}
\widetilde{r}_{(j-1)} &:=\Delta^{(j-1)}r(y_{n}) = \mathbf{V}\, d_{(j-1)} + \widetilde{r}_{(j-1)}^{\bot}, \\
\textnormal{where} &\quad \mathbf{V}^{T}\,\widetilde{r}_{(j-1)}   = d_{(j-1)}, \quad
\left(\mathbf{I}_N - \mathbf{V}\mathbf{V}^{T}\right)\, \widetilde{r}_{(j-1)}  = \widetilde{r}_{(j-1)}^{\bot}.
\end{split}
\end{equation}
\end{enumerate}
In the above equations, $\mathbf{V}\lambda_i$, $\mathbf{V}\eta_i$ and $\mathbf{V} d_{(j-1)}$ are components of $Y_i$, $f_i$ and  $\Delta^{(j-1)}r(y_{n})$ in the Krylov-subspace, whereas $Y_i^{\bot}$, $f_i^{\bot}$ and $\widetilde{r}_{(j-1)}^{\bot}$ lie in the orthogonal subspace.

Using the above equations and Lemma \ref{Lemma:psik_reduced}, the intermediate stage equations of the method \eqref{eqn:EPIRK-W-Formulation}
\begin{equation}
\begin{split}
Y_i &= y_{n}\, + a_{i,1}\,  \psi_{1}\left(g_{i,1}\,h\,\mathbf{V}\,\mathbf{H}\,\mathbf{V}^T\right)\, hf(y_{n}) + \displaystyle \sum_{j = 2}^{i} a_{i,j}\, \psi_{j}\left(g_{i,j}\,h\,\mathbf{V}\,\mathbf{H}\,\mathbf{V}^T\right)\, h\Delta^{(j-1)}r(y_{n}),
\end{split}
\end{equation}
become:
\begin{eqnarray}
Y_i = \mathbf{V}\lambda_i + Y_i^{\bot}
&=& y_{n}\, +  h\, a_{i,1} \bigg(\widetilde{p_1} \, (f_0 - \mathbf{V}\eta_0) +  \mathbf{V}\,\psi_{1}(h\,g_{i,1}\, \mathbf{H}) \eta_0\bigg)  \nonumber \\
&& + \displaystyle\sum_{j = 2}^{i} h\, a_{i,j}\, \bigg(\widetilde{p}_j \, \widetilde{r}_{(j-1)}^{\bot} +  \mathbf{V}\,\psi_{j}(h\,g_{i,j}\, \mathbf{H}) \, d_{(j-1)}\bigg), \label{eqn:full_stage_recovery}
\end{eqnarray}
where  $\widetilde{r}_{(j-1)}^{\bot}$ and $d_{(j-1)}$ can be proved to be, as is done in appendix \ref{sec:AppendixB},
\begin{subequations}
\begin{eqnarray}
\label{eqn:reduced_stage_vector_d}
d_{(j-1)} &=& \displaystyle\sum_{k=0}^{j-1} \bigg((-1)^k {j-1 \choose k} \eta_{j-1-k} - \mathbf{H} \bigg((-1)^k {j-1 \choose k} \lambda_{j-1 -k}\bigg)\bigg), \\
\label{eqn:reduced_stage_vector_r}
\widetilde{r}_{(j-1)}^{\bot} &=&
\displaystyle\sum_{k=0}^{j-1} \bigg((-1)^k {j-1 \choose k} (f_{j-1-k} - \mathbf{V} \eta_{j-1-k})\bigg).
\end{eqnarray}
\end{subequations}
The reduced stage vector for the $K$-method is obtained by multiplying  the full stage vector in equation \eqref{eqn:full_stage_recovery} by $\mathbf{V}^T$ from the left,
\begin{eqnarray}
\lambda_i
&=& \lambda_0 + h \, a_{i,1} \psi_{1}(h\, g_{i,1}\, \mathbf{H})\, \eta_0 + \displaystyle\sum_{j = 2}^{i} h\, a_{i,j}\, \psi_{j}(h\, g_{i,j}\, \mathbf{H}) \, d_{(j-1)},
\end{eqnarray}
and the component of the full stage vector, when multiplied by $(I-\mathbf{V}\mathbf{V}^T)$, in the orthogonal subspace is,
\begin{eqnarray}
Y_i^{\bot}
&=&(y_{n}\, - \mathbf{V} \lambda_0) + h \, a_{i,1} \, \widetilde{p_1} \, (f_0 - \mathbf{V}\eta_0)+ \displaystyle\sum_{j = 2}^{i} h\, a_{i,j}\, \widetilde{p}_j \, \widetilde{r}_{(j-1)}^{\bot}.
\end{eqnarray}
The full-stage vector can be recovered by first projecting the reduced stage vector $\lambda_i$ back into full space and adding the piece that is orthogonal to it, $Y_i^{\bot}$, as done in equation \eqref{eqn:full_stage_recovery}.

Similarly, the computation of the next solution in the method \eqref{eqn:EPIRK-W-Formulation}
\begin{equation}
\begin{split}
y_{n+1} &= y_{n}\, + b_{1}\, \psi_{s,1}(g_{s,1}\,h\,\mathbf{A}_{n})\, hf(y_{n}) + \displaystyle\sum_{j = 2}^{s} b_{j}\, \psi_{s,j}(g_{s,j}\,h\,\mathbf{A}_{n})\, h\Delta^{(j-1)}r(y_{n}),
\end{split}
\end{equation}
becomes:
\begin{equation}
y_{n+1} = V \lambda_s + y_{n+1}^{\bot},
\end{equation}
where 
\begin{equation}
\lambda_{s}\, = \lambda_0 + h \, b_{1}\, \psi_{1}(h\,g_{s,1}\, \mathbf{H}) \eta_0 + \displaystyle\sum_{j = 2}^{s} h\, b_{j}\, \psi_{j}(h\,g_{s,j}\, \mathbf{H}) \, d_{(j-1)},
\end{equation}
and 
\begin{equation}
y_{n+1}^{\bot} = (y_{n} - \mathbf{V} \lambda_0) + h \, b_{1}\, \widetilde{p_1} \, (f_0 - \mathbf{V}\eta_0)+ \displaystyle\sum_{j = 2}^{s} h\, b_{j}\, \widetilde{p}_j \, \widetilde{r}_{(j-1)}^{\bot}.
\end{equation}
One step of the resulting EPIRK-$K$ method (for an autonomous system) is summarized in Algorithm \ref{alg:EPIRK_algorithm}.
\begin{algorithm}
\caption{EPIRK-$K$}\label{alg:EPIRK_algorithm}
\begin{algorithmic}[1]
\State $f_0 := f(y_{n})$\Comment{Repeat for every timestep in the timespan}
\State $\mathbf{J}_{n} := \mathbf{J}(y_{n})$
\State $[\mathbf{H}, \mathbf{V}] = \textnormal{Arnoldi}(\mathbf{J}_{n}, f_0)$
\State $\lambda_0 = \mathbf{V}^{T}y_{n}\,$
\State $\eta_0 = \mathbf{V}^{T}f_0$
\For{$i=1:s - 1$}\Comment{Do for each stage of the method}
\State $d_{(i-1)} = \displaystyle\begin{cases} 0 & \text{, if } i = 1\\ \displaystyle\sum_{k=0}^{i-1} \bigg((-1)^k\, {i-1 \choose k}\, \eta_{i-1-k} - \mathbf{H}\, \bigg((-1)^k\, {i-1 \choose k}\, \lambda_{i-1 -k}\bigg)\bigg) & \text{, if }  i \geq 2 \end{cases}$
\State $\widetilde{r}_{(i-1)}^{\bot} = \displaystyle\begin{cases}
0 & \text{, if } i = 1 \\
\displaystyle\sum_{k=0}^{i-1} \bigg((-1)^k\, {i-1 \choose k}\, (f_{i-1-k} - \mathbf{V} \,\eta_{i-1-k})\bigg) & \text{, if } i\geq 2 \end{cases}$
\State $\lambda_i = \lambda_0 + h \, a_{i,1}\, \psi_{1}(h\, g_{i,1}\, \mathbf{H})\, \eta_0 + \displaystyle\sum_{j = 2}^{i} h\, a_{i,j}\, \psi_{j}(h\, g_{i,j}\, \mathbf{H}) \, d_{(j-1)}$
\State $Y_i =  \mathbf{V}\,\lambda_i + (y_{n}\, - \mathbf{V}\,\lambda_0) + h \,a_{i,1} \, \widetilde{p_{1}} \, (f_0 - \mathbf{V}\,\eta_0) + \displaystyle\sum_{j = 2}^{i} h\, a_{i,j}\, \widetilde{p}_j \,\widetilde{r}_{(j-1)}^{\bot}$
\State $f_i = f(Y_i)$
\State $\eta_i = \mathbf{V}^{T}f_i$
\EndFor
\State $d_{(s-1)} = \displaystyle\sum_{k=0}^{s-1} \bigg((-1)^k\, {s-1 \choose k}\, \eta_{s-1-k} - \mathbf{H}\, \bigg((-1)^k\, {s-1 \choose k}\, \lambda_{s-1 -k}\bigg)\bigg)$
\State $\widetilde{r}_{(s-1)}^{\bot} =
\displaystyle\sum_{k=0}^{s-1} \bigg((-1)^k\, {s-1 \choose k}\, (f_{s-1-k} - \mathbf{V}\, \eta_{s-1-k})\bigg)$
\State $\lambda_s = \lambda_0 + h \, b_{1}\, \psi_{1}(h\, g_{s,1}\, \mathbf{H})\, \eta_0 + \displaystyle\sum_{j = 2}^{s} h\, b_{j}\, \psi_{j}(h \,g_{s,j}\,\mathbf{H}) \, d_{(j-1)}$
\State $y_{n+1} =  \mathbf{V}\, \lambda_s + (y_{n}\, - \mathbf{V}\,\lambda_0) + h \,b_{1}\, \widetilde{p_{1}}\, (f_0 - \mathbf{V}\,\eta_0) + \displaystyle\sum_{j = 2}^{s} h\, b_{j}\, \widetilde{p}_j \,\widetilde{r}_{(j-1)}^{\bot}$ 
\end{algorithmic}
\end{algorithm}

\subsection{Order conditions theory for EPIRK-$K$ methods}
\label{sec:EPIRK-K-order}

$K$-methods construct a single Krylov-subspace per timestep, and use it to approximate the Jacobian.  All stage vectors are also computed in this reduced space, before being projected back into the full space between successive stages. The order condition theory accounts for this computational procedure  \cite[Theorem 3.6]{Tranquilli_2014_ROK}. 

Before we discuss the order conditions for the $K$-methods, we define the trees that arise in the expansion of their solutions. Recalling that a linear tree is one where each node has only one child, consider the following definition:
\begin{definition}[TK-Trees \cite{Tranquilli_2014_ROK}]\leavevmode
\begin{enumerate}
\item[] $TK = \bigg\{\text{TW-trees: no linear sub-tree has a fat (empty) root}\bigg\}.$
\item[] $TK(k) = \bigg\{\text{TW-trees: no linear sub-tree of order less than or equal to k has a fat (empty) root}\bigg
\}.$
\end{enumerate}
\end{definition}

\begin{theorem}
Trees corresponding to series expansion of the numerical solution of $K$-method with Krylov-subspace dimension $M$ are $TK(M)$.
\end{theorem}
\begin{proof}
Using lemmas \ref{Lemma:powers_of_a0}, \ref{Lemma:phik_reduced}, \ref{Lemma:psik_reduced}, and \cite[Lemmas 3, 4]{Tranquilli_2014_ExpK} we arrive at the desired result.
\end{proof}

The order conditions are derived from matching the coefficients of the B-series expansion of the numerical solution to those of the B-series expansion of the exact solution. If the expansion is made in the elementary differentials corresponding to the \textit{TK} trees, then for a fourth order method there is a single additional tree ($\tau^{K}_{8}$) in the numerical solution having both meagre and fat nodes in comparison to the \textit{T}-trees up to order four for a classical EPIRK method. In contrast with the $W$-method that had twenty-one order conditions, the $K$-method has just nine for a fourth order method, which is almost the same as a classical EPIRK having eight.

As mentioned earlier, \textit{TK}-trees come from re-coloring all the linear sub-trees with a fat root as meagre in the \textit{TW}-trees. This significantly reduces the number of order conditions for the $K$-method since groups of \textit{TW}-trees become isomorphic to one another after recoloring. \textit{TK}-trees up to order four are given in \cite{Tranquilli_2014_ExpK}. This also indicates that the coefficients in front of these \textit{TK}-trees in the B-series expansion can be obtained by summing together the coefficients of \textit{TW}-trees,  which become isomorphic to one another, from the corresponding expansion.

The order four conditions for EPIRK-$K$ methods with three stages  \eqref{eqn:EPIRK-W-Formulation-3-Stage} are discussed next and are summarized in Table \ref{tab:Order-EPIRK-K-3stages}.

\begin{longtable}{|>{\centering\arraybackslash}m{1.25cm}|>{\centering\arraybackslash}m{1.2cm}|>{\centering\arraybackslash}m{11cm}|}
\caption{Order conditions for the three stage EPIRK-$K$ method}
\label{tab:Order-EPIRK-K-3stages}\\
\toprule
\centering \textbf{Tree \#} &
\centering \textbf{Order} &
\centering \textbf{Order Condition: $B^{\#}(y_{n+1}) - B^{\#}\big(y(t_{n}\,+h)\big) = 0$}\tabularnewline
\midrule
\endfirsthead
\toprule
\centering \textbf{Tree \#} &
\centering \textbf{Order} &
\centering \textbf{Order Condition: $B^{\#}(y_{n+1}) - B^{\#}\big(y(t_{n}\,+h)\big) = 0$}\tabularnewline
\midrule \endhead
$\tau_{1}^{K}$&1&\begin{dmath*}[breakdepth={3}]b_{1}\, p_{1,1}\,-1= 0\end{dmath*}\\\hline
$\tau_{2}^{K}$&2&\begin{dmath*}[breakdepth={3}]\frac{1}{2} (b_{1}\, g_{3,1}\, p_{1,1}\,-1)= 0\end{dmath*}\\\hline
$\tau_{3}^{K}$&3&\begin{dmath*}[breakdepth={3}]\frac{1}{6} \left(6 a_{1,1}^2\, b_{2}\, p_{1,1}^2\, p_{2,1}\,+3 a_{1,1}^2\, b_{2}\, p_{1,1}^2\,
   p_{2,2}\,-12 a_{1,1}^2\, b_{3}\, p_{1,1}^2\, p_{3,1}\,-6 a_{1,1}^2\, b_{3}\, p_{1,1}^2\, p_{3,2}\,-2 a_{1,1}^2\, b_{3}\, p_{1,1}^2\, p_{3,3}\,+6 a_{2,1}^2\, b_{3}\,
   p_{1,1}^2\, p_{3,1}\,+3 a_{2,1}^2\, b_{3}\, p_{1,1}^2\, p_{3,2}\,+a_{2,1}^2\, b_{3}\, p_{1,1}^2\, p_{3,3}\,-2\right)= 0\end{dmath*}\\\hline
$\tau_{4}^{K}$&3&\begin{dmath*}[breakdepth={3}]\frac{1}{6} \left(b_{1}\, g_{3,1}^2\,
   p_{1,1}\,-1\right)= 0\end{dmath*}\\\hline
$\tau_{5}^{K}$&4&\begin{dmath*}[breakdepth={3}]\frac{1}{12} \left(12 a_{1,1}^3\, b_{2}\, p_{1,1}^3\, p_{2,1}\,+6 a_{1,1}^3\, b_{2}\, p_{1,1}^3\, p_{2,2}\,-24 a_{1,1}^3\, b_{3}\, p_{1,1}^3\,
   p_{3,1}\,-12 a_{1,1}^3\, b_{3}\, p_{1,1}^3\, p_{3,2}\,-4 a_{1,1}^3\, b_{3}\, p_{1,1}^3\, p_{3,3}\,+12 a_{2,1}^3\, b_{3}\, p_{1,1}^3\, p_{3,1}\,+6 a_{2,1}^3\, b_{3}\,
   p_{1,1}^3\, p_{3,2}\,+2 a_{2,1}^3\, b_{3}\, p_{1,1}^3\, p_{3,3}\,-3\right)= 0\end{dmath*}\\\hline
$\tau_{6}^{K}$&4&\begin{dmath*}[breakdepth={3}]\frac{1}{24} \left(12 a_{1,1}^2\, b_{2}\, g_{1,1}\, p_{1,1}^2\, p_{2,1}\,+6 a_{1,1}^2\, b_{2}\,
   g_{1,1}\, p_{1,1}^2\, p_{2,2}\,-24 a_{1,1}^2\, b_{3}\, g_{1,1}\, p_{1,1}^2\, p_{3,1}\,-12 a_{1,1}^2\, b_{3}\, g_{1,1}\, p_{1,1}^2\, p_{3,2}\,-4 a_{1,1}^2\, b_{3}\,
   g_{1,1}\, p_{1,1}^2\, p_{3,3}\,+12 a_{2,1}^2\, b_{3}\, g_{2,1}\, p_{1,1}^2\, p_{3,1}\,+6 a_{2,1}^2\, b_{3}\, g_{2,1}\, p_{1,1}^2\, p_{3,2}\,+2 a_{2,1}^2\, b_{3}\,
   g_{2,1}\, p_{1,1}^2\, p_{3,3}\,-3\right)= 0\end{dmath*}\\\hline
$\tau_{7}^{K}$&4&\begin{dmath*}[breakdepth={3}]\frac{1}{12} \left(12 a_{1,1}^2\, a_{2,2}\, b_{3}\, p_{1,1}^2\, p_{2,1}\, p_{3,1}\,+6 a_{1,1}^2\, a_{2,2}\, b_{3}\, p_{1,1}^2\,
   p_{2,1}\, p_{3,2}\,+2 a_{1,1}^2\, a_{2,2}\, b_{3}\, p_{1,1}^2\, p_{2,1}\, p_{3,3}\,+6 a_{1,1}^2\, a_{2,2}\, b_{3}\, p_{1,1}^2\, p_{2,2}\, p_{3,1}\,+3 a_{1,1}^2\, a_{2,2}\,
   b_{3}\, p_{1,1}^2\, p_{2,2}\, p_{3,2}\,+a_{1,1}^2\, a_{2,2}\, b_{3}\, p_{1,1}^2\, p_{2,2}\, p_{3,3}\,-1\right)= 0\end{dmath*}\\\hline
$\tau_{8}^{K}$&4&\begin{dmath*}[breakdepth={3}]\frac{1}{24} p_{1,1}^2\, \left(-24 a_{1,1}^2\, a_{2,2}\,
   b_{3}\, p_{2,1}\, p_{3,1}\,-12 a_{1,1}^2\, a_{2,2}\, b_{3}\, p_{2,1}\, p_{3,2}\,-4 a_{1,1}^2\, a_{2,2}\, b_{3}\, p_{2,1}\, p_{3,3}\,-12 a_{1,1}^2\, a_{2,2}\, b_{3}\,
   p_{2,2}\, p_{3,1}\,-6 a_{1,1}^2\, a_{2,2}\, b_{3}\, p_{2,2}\, p_{3,2}\,-2 a_{1,1}^2\, a_{2,2}\, b_{3}\, p_{2,2}\, p_{3,3}\,+12 a_{1,1}^2\, b_{2}\, g_{3,2}\, p_{2,1}\,+4
   a_{1,1}^2\, b_{2}\, g_{3,2}\, p_{2,2}\,-24 a_{1,1}^2\, b_{3}\, g_{3,3}\, p_{3,1}\,-8 a_{1,1}^2\, b_{3}\, g_{3,3}\, p_{3,2}\,-2 a_{1,1}^2\, b_{3}\, g_{3,3}\, p_{3,3}\,+12
   a_{2,1}^2\, b_{3}\, g_{3,3}\, p_{3,1}\,+4 a_{2,1}^2\, b_{3}\, g_{3,3}\, p_{3,2}\,+a_{2,1}^2\, b_{3}\, g_{3,3}\, p_{3,3}\,\right)= 0\end{dmath*}\\\hline
$\tau_{9}^{K}$&4&\begin{dmath*}[breakdepth={3}]\frac{1}{24} \left(b_{1}\, g_{3,1}^3\,
   p_{1,1}\,-1\right) = 0\end{dmath*}\\\hline
\end{longtable}

\begin{corollary}\label{Cor:K-method-implemented-as-classical}
Any EPIRK-$K$-method of order $p$ gives rise to a classical EPIRK method of order $p$ (with the same coefficients).
\end{corollary}
\begin{proof}
The proof follows directly from comparing the set of \textit{T}-trees to \textit{TK}-trees.
\end{proof}
\begin{corollary}\label{Cor:Classical-method-implemented-as-K}
Any classical EPIRK method of order $p \geq 3$ gives rise to an equivalent $K$-method of order at least three  (with the same coefficients).
\end{corollary}
\begin{proof}
The proof follows directly from comparing the set of \textit{T}-trees to \textit{TK}-trees.
\end{proof}

\subsection{Construction of practical EPIRK-$K$ integrators}

Note that order condition $\tau_{8}^{K}$ is the additional order condition that was mentioned earlier whose \textit{TW}-tree could not be re-colored according to \cite[Lemma 3.2, 3.3]{Tranquilli_2014_ROK}. This is the only tree that is not present in \textit{T}-trees up to order 4, and therefore we impose that the associated coefficient is equal to zero. To solve the order conditions we use \Mathematica and the solution procedure described for the first variant of EPIRK-$W$ method where we make terms of each order condition similar to one another by making suitable substitutions. We arrive at the EPIRK-$K$ method ({\sc epirkk4}) of order four in Figure \ref{fig:EPIRK-k-coefficients-1}, with an embedded method of order three.
\begin{figure}[tbh]
\caption{Coefficients for {\sc epirkk4}\label{fig:EPIRK-k-coefficients-1}}
\begin{equation*}
\renewcommand{\arraystretch}{1.5}
\begin{split}
a = \begin{bmatrix}
\frac{692665874901013}{799821658665135} & 0 & 0\\
\frac{692665874901013}{799821658665135} & \frac{3}{4} & 0
\end{bmatrix}, &\hspace{1cm}
\begin{bmatrix} b \\ \widehat{b} \end{bmatrix}  = \begin{bmatrix}
\frac{799821658665135}{692665874901013} & \frac{352}{729} & \frac{64}{729}\\
\frac{799821658665135}{692665874901013} & \frac{32}{81} & 0
\end{bmatrix},\\
g = \begin{bmatrix}
\frac{3}{4} & 0 & 0\\
\frac{3}{4} & 0 & 0\\
1 & \frac{9}{16} & \frac{9}{16}
\end{bmatrix},
&\hspace{1cm}
p = \begin{bmatrix}
\frac{692665874901013}{799821658665135} & 0 & 0\\
1 & 1 & 0\\
1 & 1 & 0
\end{bmatrix}.
\end{split}
\end{equation*}
\end{figure}

The theory of $K$-methods gives a lower bound on the Krylov-subspace size that guarantees the order of convergence \cite[Theorem 3.6]{Tranquilli_2014_ROK}. This bound dependents only on the order of convergence of the method and not on the dimension of the ODE system.

Here we have constructed a fourth order EPIRK-$K$ method in Figure \ref{fig:EPIRK-k-coefficients-1}, which requires a Krylov-subspace of dimension four \cite[Theorem 3.6]{Tranquilli_2014_ROK}. All expensive operations such as computing products of $\psi$ function of matrices with vectors are performed in this reduced space. Significant computational advantages are obtained if the Krylov-subspace captures all the stiff eigenmodes of the system. If not all stiff eigenmodes are captured, stability requirements will force the integrator to take smaller timesteps, which will increase the overall cost of integration. In such cases we typically observe that adding more vectors to the Krylov-subspace can improve the performance of the $K$-type integrator.

\section{Implementation strategies for exponential integrators\label{sec:Implementations}}

An important part of the implementation of exponential integrators is the computation of products of $\varphi$ or $\psi$ functions of matrices with vectors. The choice of approximation used depends on the properties of the matrix and the type of exponential method under consideration. In this section, we briefly review the strategies used to evaluate these products in the context of the methods discussed in this paper. Table \ref{Table:fixed-convergence-order} lists these methods and the implementation framework used in each case. It includes the $W$-methods presented in Section \ref{sec:EPIRK-W}, a fourth order $K$-method presented in Section \ref{sec:EPIRK-K}, and several classical methods that have been implemented to perform comparative studies. In the following discussion we will use the terms implementation and integrator interchangeably, and the term method should be understood as the formulation along with the set of coefficients.

We first discuss the computation of these products for classical exponential integrators listed in Table \ref{Table:fixed-convergence-order}. For {\sc epirkk4-classical} and {\sc epirk5} \cite[Table 4, Equation 28]{Tokman_2011_EPIRK} the evaluation of products of $\psi$ functions with vectors proceeds by first approximating the individual $\varphi$ function products in the Krylov-subspace as illustrated in \cite[sec 3.1]{Tokman_2006_EPI}, i.e.  $\varphi_k(h\,\gamma\, \mathbf{A}_{n})b \approx \norm{b} \mathbf{V}\varphi_k(h\,\gamma\, \mathbf{H})e_1$. The $\varphi$ function products associated with the approximation, $\varphi_k(h\,\gamma\,\mathbf{H})e_1$, are computed by constructing an augmented matrix and exponentiating it as described in \cite[Theorem 1]{Sidje_1998}. Finally taking a linear combination of columns of the resulting matrix gives the $\psi$ function product. In the augmented matrix approach, an Arnoldi iteration is needed for each new vector in the formulation given in equation \eqref{eqn:EPIRK-Formulation}. Each of the classical integrators listed above has three stages that work with three different vectors $b$, requiring three Arnoldi iterations per timestep.  {\sc epirk5p1bvar} \cite{Loffeld_Tranquilli_Tokman_2012}   and {\sc erow4} \cite{Hochbruck_2009_exp,Tokman_2010_Efficient}, which we include in our numerical experiments, also perform three Arnoldi projections per timestep. While {\sc erow4} uses the augmented matrix approach to compute the $\varphi$ function products, {\sc epirk5p1bvar} takes a completely different approach to evaluating the $\psi$ function products. It combines the Arnoldi process and the evaluation of $\psi$ function products by sub-stepping, to make it adaptive, into a single computation process as described in \cite{Niesen_2012_Alg919}.

Next we will discuss how we compute the $\psi$ function products in the context of $W$-methods as they are closely related to classical methods in the computation of these products.  Recall that $W$-methods admit arbitrary Jacobian approximations while maintaining full convergence order. In order to demonstrate this, we have alternate implementations that use different approximations to the Jacobian, namely, the Jacobian itself;  the identity matrix; the zero matrix, which reduces the {\sc epirkw} scheme \eqref{eqn:EPIRK-W-Formulation-3-Stage} to an explicit Runge-Kutta scheme; and lastly, the diagonal of the Jacobian. We will refer to these methods as {\sc epirkw3} in the paper and highlight the approximation in the context. Figure \ref{fig:fixed-convergence-plot} in the following section shows that these alternative implementations retain full order when we use an approximation to $\mathbf{J}_{n}$, as expected. It is to be noted, however, that in the case of $W$-methods, $\A_n=\J_n$ might be needed to assure stability, and some discussion in this regard when $\mathbf{A}_{n} = \mathbf{J}_{n}$ or when $\Vert \mathbf{A}_{n} - \mathbf{J}_{n} \Vert$ is small, is done in \cite[Sec IV.7, IV.11]{Hairer_book_II} and \cite{Steihaug_1979}.

The different implementations of the $W$-method vary in the way the $\psi$ function products are computed. Among the many possible combinations, we will restrict our discussion of the evaluation of these products to the following cases: 

\begin {enumerate}
\item $\mathbf{A}_{n} = \text{diag}(\mathbf{J}_{n})$.
We compute the individual $\varphi_k$ functions with successively higher subscripts using the recursive definition given in \eqref{eqn:phi_k_recursive_formulation_and_phi_k_0}, where $\varphi_0$ is computed as point-wise exponential of the entries along the diagonal. Next, we take the linear combination of the $\varphi$ functions to evaluate the $\psi$ function as defined in equation \eqref{eqn:simplified_psi_function_defn}. Finally, the product of $\psi$ function of matrix times vector is evaluated as a matrix-vector product. A similar procedure is adapted for other approximations where $\mathbf{A}_{n}$ is either zero or identity.

\item $\mathbf{A}_{n} = \mathbf{J}_{n}$. 
The evaluation of products of $\psi$ function of Jacobian with vectors is similar to the classical integrators. Three Arnoldi iterations are required for each new vector in the three stage formulation given in equation \eqref{eqn:EPIRK-W-Formulation-3-Stage}. An implementation involving a non-trivial approximation of the Jacobian may use a similar strategy to compute the $\psi$ function products.
\end {enumerate}

In addition to classical integrators and $W$-type integrators, we have implemented {\sc epirkk4}, a fourth-order $K$-type integrator, and {\sc epirk5-k} \cite[Table 4]{Tokman_2011_EPIRK}, a classical fifth-order method implemented in the $K$-framework. These integrators use the reduced space formulation for stage values as shown in Algorithm \ref{alg:EPIRK_algorithm}. The reduced space is constructed using a single Arnoldi iteration per timestep. \cite{Tranquilli_2014_ROK} proves that $K$-methods  require only as many basis vectors in the reduced space as the order of the method to guarantee same order of convergence. This is in direct contrast to other methods considered in this paper that perform Arnoldi iteration where the basis size is variable and is dependent on the residual being below a chosen tolerance. {\sc epirkk4}, being a fourth order $K$-method, theoretically requires only four vectors in the Krylov-subspace. Since the stage values of {\sc epirkk4} are computed in this reduced space, the size of the matrices used in the computation of $\psi$ function products is usually smaller than the size of the ODE system under consideration. As a result, in our implementation of {\sc epirkk4}, quantities $\varphi_k(h\,\gamma\,\mathbf{H})$ are directly computed using the recursive definition given in \eqref{eqn:phi_k_recursive_formulation_and_phi_k_0}, they are stored independently and linearly combined to compute the $\psi$ function. The product of a $\psi$ function of a matrix with a vector is computed as a matrix-vector product. Despite using a similar procedure with five basis vectors in the Krylov-subspace, {\sc epirk5-k} does not attain its theoretical order of convergence; the reason for this will become clear in the following section.

\section{Numerical results}
\label{sec:Results}

The integrators discussed in Section \ref{sec:Implementations} were evaluated by running fixed-step convergence experiments on the Lorenz-96 system and variable time-stepping experiments on Lorenz-96, Shallow Water Equations, and Allen-Cahn system. In order to run variable time-stepping experiments, we used the \Matlode framework \cite{Sandu_2015_MATLODE}, a \Matlab -based ODE solver suite developed by the co-authors.

\subsection{Lorenz-96 system\label{sec:L96_section}}

Lorenz-96 model \cite{Lorenz_1996_predictability} is described by the system of ODEs:
\begin{equation}
\label{eqn:Lorenz}
\frac{dy_j}{dt} = -y_{j-1}(y_{j-2}-y_{j+1})-y_j+F,\quad j = 1\dots N, \quad
y_{1} = y_{N}.
\end{equation}
Here $N$ denotes the number of states and $F$ is the forcing term. For our experiments we let $N=40$, and $F=8$. The initial condition for the model was obtained by integrating over $[0, 0.3]$  time units using \texttt{MATLAB}'s \texttt{ODE45} integrator. Note also that the boundary condition is periodic. 

\begin{figure}[htb]
  \centering
  \includegraphics[scale=0.25]{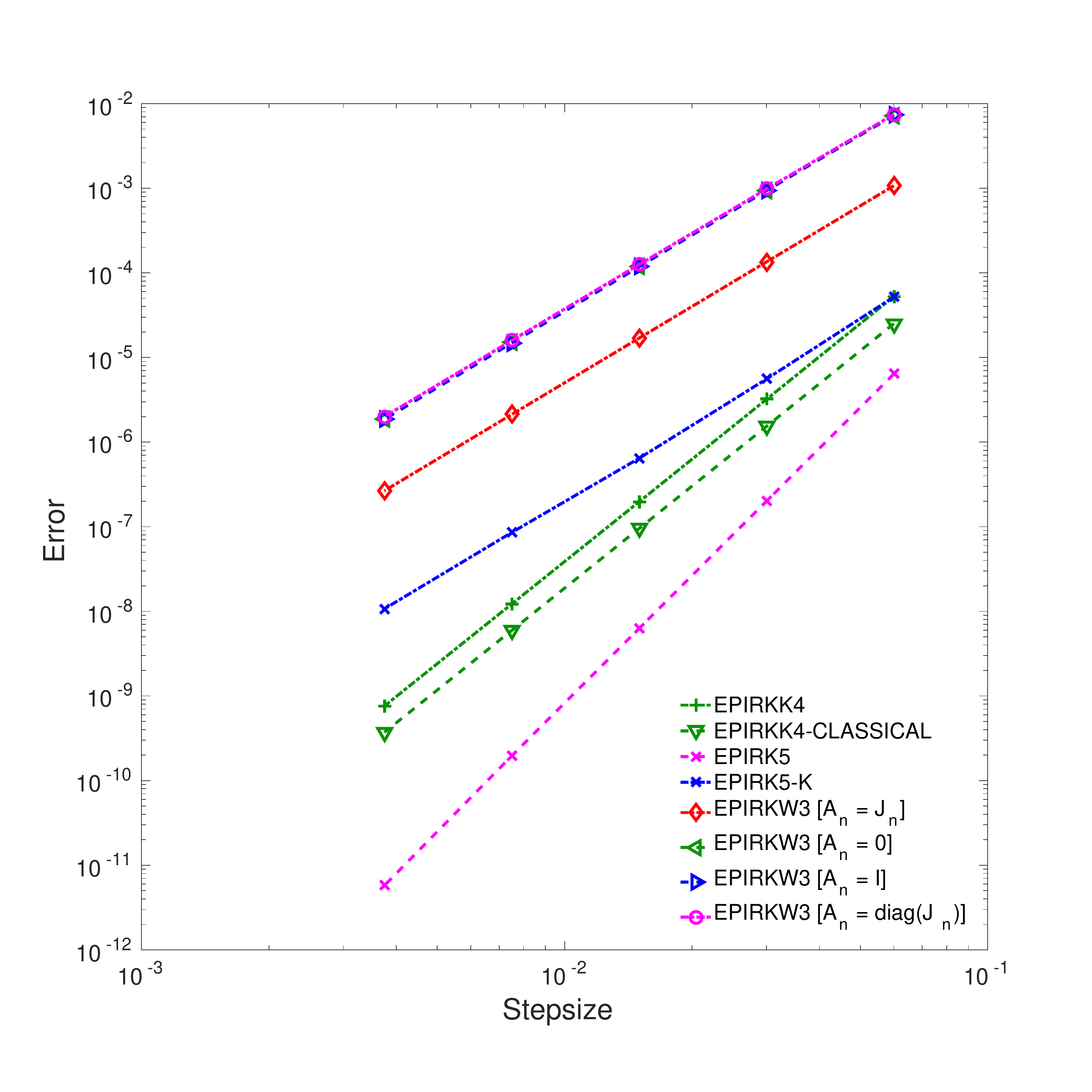}
  \caption{Convergence plot for different methods applied to Lorenz-96 model \eqref{eqn:Lorenz}.}
    \label{fig:fixed-convergence-plot}
\end{figure}

The Lorenz-96 system was used to perform fixed-step convergence experiments on a subset of integrators listed in Table \ref{Table:fixed-convergence-order}. The reference solution was computed using \texttt{ODE45} with absolute and relative tolerances set to \num{1e-12}. 
The convergence plots for fixed-step experiments are shown in Figure \ref{fig:fixed-convergence-plot}. The results   for each implementation and coefficient combination are summarized in Table \ref{Table:fixed-convergence-order}.
These results support the following conclusions:
\begin{enumerate}[parsep=1pt,listparindent=\parindent]
\item The numerical experiments verify the order of the methods derived in the paper.  Both {\sc epirkk4} and {\sc epirkw3} show the theoretical order of convergence.
\item The theory of $W$-methods is validated through testing different approximations to the Jacobian. Indeed,
{\sc epirkw3} consistently shows third order convergence irrespective of the approximation used for the Jacobian.
\item The results demonstrate that every $K$-method can be implemented as a classical method of the same order.
Here we test {\sc epirkk4-classical}, the implementation of the method {\sc epirkk4} in the classical framework as discussed in Section \ref{sec:Implementations}. It has the same order of convergence as {\sc epirkk4} since the coefficients derived for {\sc epirkk4} method satisfies all the order conditions of a fourth order classical method. This is particularly true for every $K$-method.
\item In general, classical methods implemented as a $K$-methods do not preserve the order and may suffer from order reduction.  
Here {\sc epirk5} and {\sc epirk5-k} were implemented for this specific purpose where {\sc epirk5-k} is a $K$-type implementation of the fifth order classical method {\sc epirk5} derived in \cite{Tokman_2011_EPIRK}. Clearly, from Table \ref{Table:fixed-convergence-order}, we see that the $K$-type implementation shows order reduction as the coefficient in front of the elementary differential corresponding to $\tau^{K}_{8}$ is non-zero in the $K$-type method. 
In general, solving the order conditions for a $K$-type method is more restrictive than solving for a classical method due to the additional order conditions that arise from approximating the Jacobian in the Krylov-subspace. As a result, not all classical methods lead to $K$-type integrators of the same order. 
\end{enumerate}
The last two observations are in line with Corollaries \ref{Cor:K-method-implemented-as-classical} and \ref{Cor:Classical-method-implemented-as-K}. 

\begin{table}[htb]
\centering
\begin{tabular}{|>{\arraybackslash}m{4.5cm}|>{\centering\arraybackslash}m{2.7cm}|>{\centering\arraybackslash}m{2.2cm}|>{\centering\arraybackslash}m{1.3cm}|>{\centering\arraybackslash}m{2.25cm}|}
\hline
\textbf{Integrator} & \textbf{Implementation Framework} & \textbf{Coefficients} & \textbf{Derived Order} &\textbf{Fixed Step Convergence Order} \\\hline
{\sc epirkk4}                                    & $K$-Type & Figure \ref{fig:EPIRK-k-coefficients-1}        & 4 & 4.018722  \\ \hline
{\sc epirkk4-classical}                          & Classical       & Figure \ref{fig:EPIRK-k-coefficients-1}        & 4 & 4.009777  \\ \hline
${\ddag}$ {\sc epirk5}                                              & Classical       & \cite[~Table 4]{Tokman_2011_EPIRK}             & 5 & 5.016092  \\ \hline
${\ddag}$ {\sc epirk5-k}                                            & $K$-Type & \cite[~Table 4]{Tokman_2011_EPIRK}             & 5 & 3.051511  \\ \hline
{\sc epirkw3} $[\A_n = \J_n]$                    & $W$-Type & Figure \ref{fig:epirk-w-coefficients-2}        & 3 & 2.994241  \\ \hline
{\sc epirkw3} $[\A_n = \textnormal{diag}(\J_n)]$ & $W$-Type & Figure \ref{fig:epirk-w-coefficients-2}        & 3 & 2.967430  \\ \hline
${\ddag}$ {\sc epirkw3} $[\A_n = \I]$                               & $W$-Type & Figure \ref{fig:epirk-w-coefficients-2}        & 3 & 2.987911  \\ \hline
${\ddag}$ {\sc epirkw3} $[\A_n = \mathbf{0}]$                       & $W$-Type & Figure \ref{fig:epirk-w-coefficients-2}        & 3 & 2.977000  \\ \hline
{\sc erow4}                                     & Classical       & \cite[Section 5]{Hochbruck_2009_exp}           & 4 & --        \\ \hline
{\sc epirk5p1bvar}                               & Classical       & \cite[Table 3]{Loffeld_Tranquilli_Tokman_2012} & 5 & --        \\ \hline
\end{tabular}
\caption{Comparative study of various exponential integrators. The convergence orders shown are obtained from fixed step size experiments with the Lorenz-96 model \eqref{eqn:Lorenz}. Integrators implemented only for studying fixed step convergence behavior are excluded from variable timestep experiments and these are prefixed with $\ddag$ in column 1.} 
\label{Table:fixed-convergence-order}
\end{table}

\begin{figure}[!htb]
    \centering
    \begin{subfigure}{.45\textwidth}
        \centering
	    \includegraphics[scale=0.25]{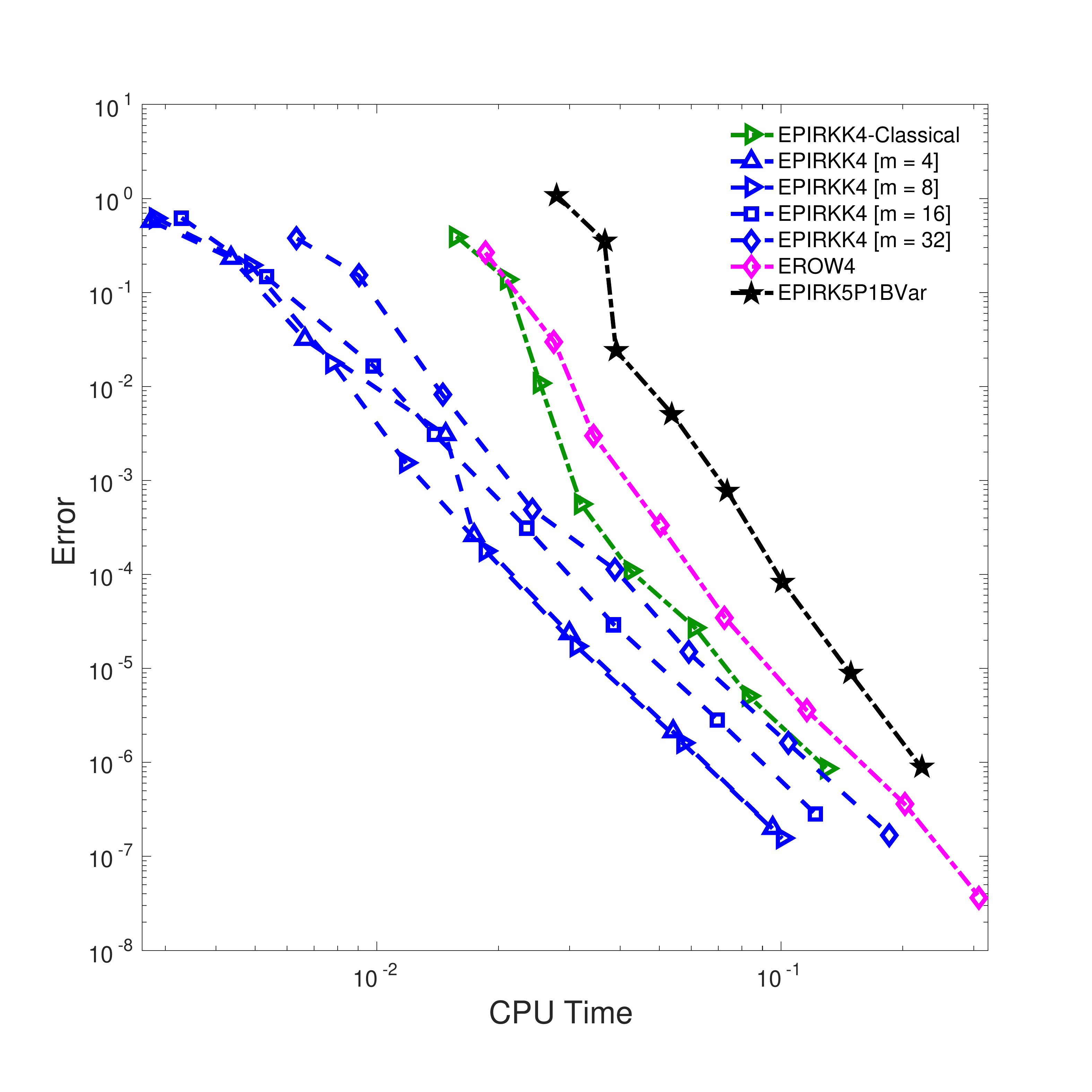}
   	    \caption{K-type versus classical methods}
		\label{fig:L96-work-prec-diag-kc}
    \end{subfigure}%
    \begin{subfigure}{0.45\textwidth}
        \centering
	    \includegraphics[scale=0.25]{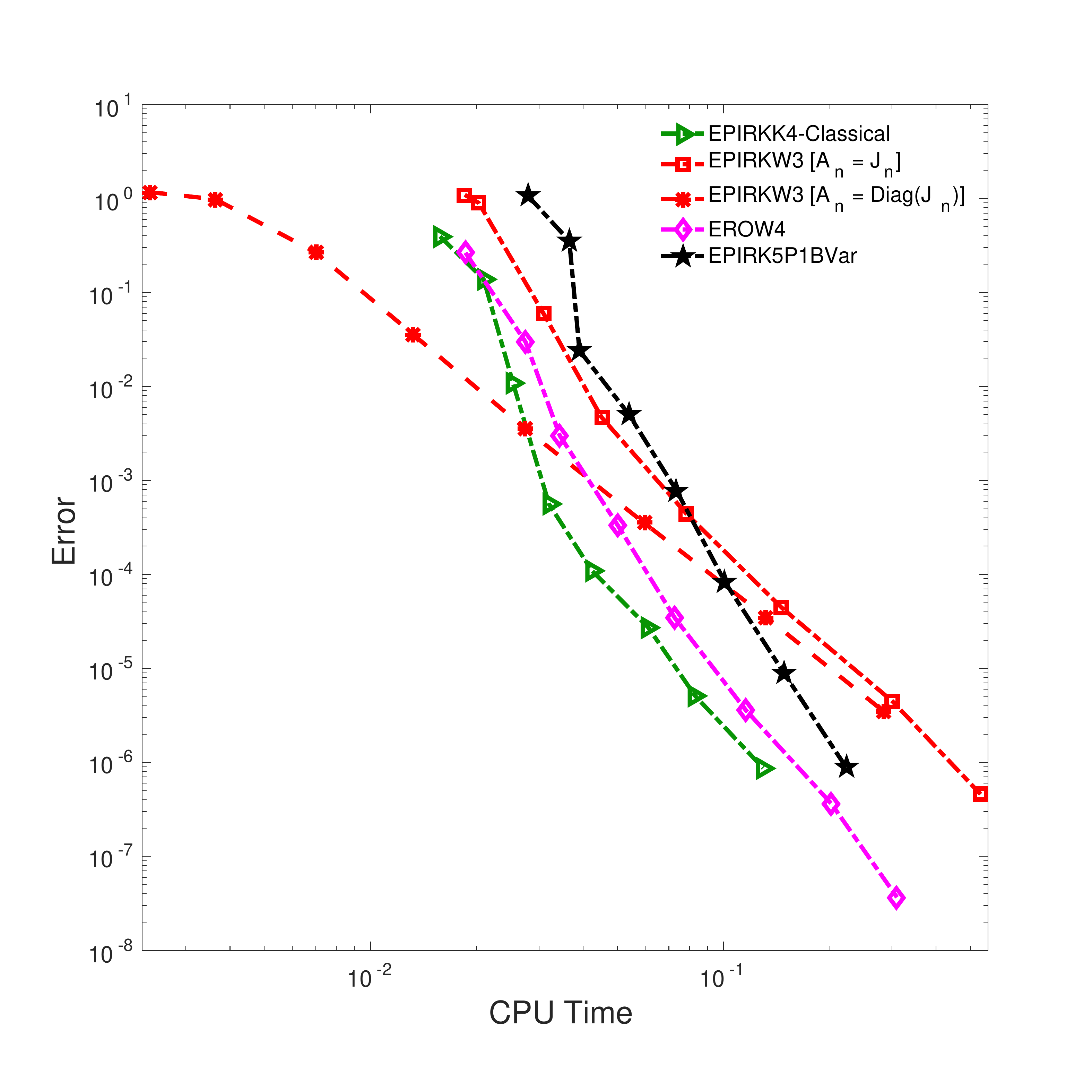}
   	    \caption{W-type versus classical methods}
		\label{fig:L96-work-prec-diag-wc}
    \end{subfigure}%
    \caption{Work-precision diagrams for different methods  applied to Lorenz-96 problem \eqref{eqn:Lorenz}.}
\end{figure}

\begin{figure}[!htb]
    \centering
    \begin{subfigure}{.45\textwidth}
        \centering
	    \includegraphics[scale=0.25]{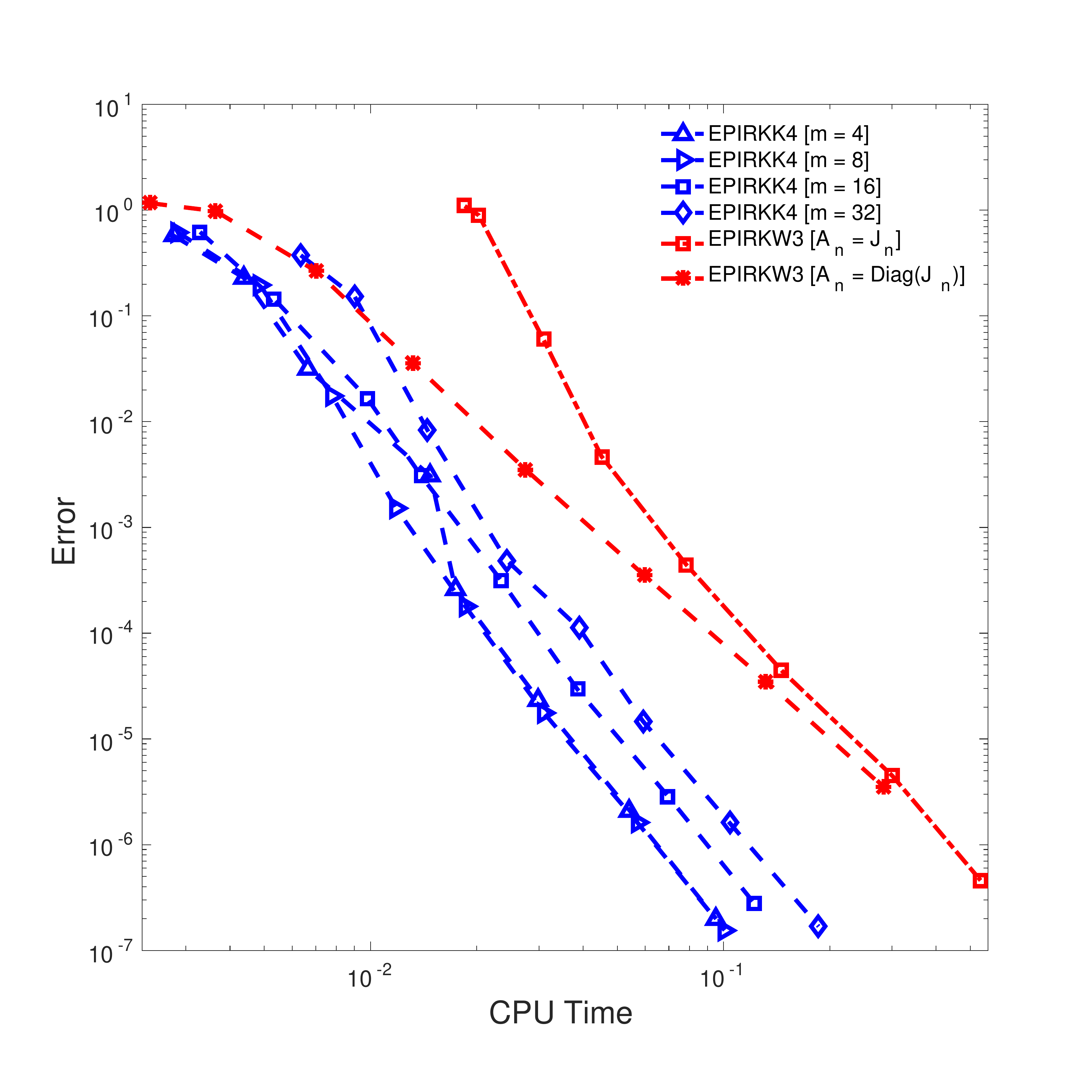}
   	    \caption{W-type versus K-type methods}
		\label{fig:L96-work-prec-diag-wk}
    \end{subfigure}%
    \begin{subfigure}{0.45\textwidth}
        \centering
		\includegraphics[scale=0.25]{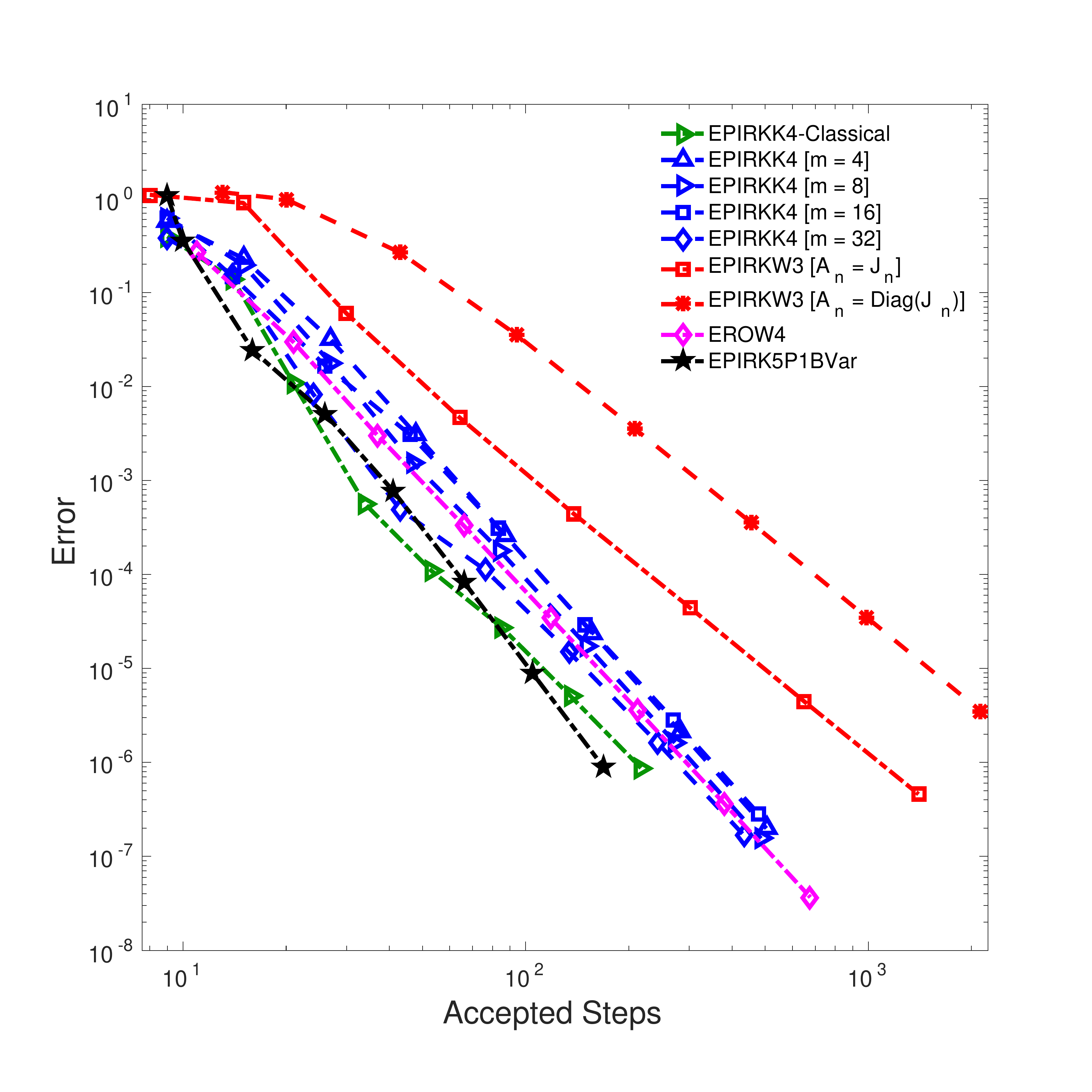}
		\caption{Convergence diagram}
		\label{fig:L96-error-vs-steps}
    \end{subfigure}%
    \caption{Work-precision and convergence diagrams for different methods applied to Lorenz-96 problem\eqref{eqn:Lorenz}.}
\end{figure}

To evaluate the relative computational efficiency of the integrators -- {\sc epirkk4}, {\sc epirkw3}, and {\sc epirkk4-classical}  -- we perform variable time-stepping experiments with the Lorenz-96 system, the Shallow Water Equations, and the Allen-Cahn system. We also include known exponential schemes such as {\sc epirk5p1bvar} \cite{Loffeld_Tranquilli_Tokman_2012} and {\sc erow4} \cite{Hochbruck_2009_exp,Tokman_2010_Efficient} as  benchmark for the performance of the methods derived in this paper. Experiments are run for different relative tolerances $[\num{1e-1}, \num{1e-2}, \hdots, \num{1e-8}]$; the error controller \cite[Section II.4]{Hairer_book_I}, which is the same across all methods, adjusts the stepsize to ensure that the solution meets the tolerance requirement. The relative tolerance for the Arnoldi algorithm was set to \num{1e-12} for {\sc epirkk4}, {\sc epirkk4-classical} and {\sc erow4}; \num{1e-9} for {\sc epirkw3} ($\A_n = \J_n$) and to experimental tolerance for {\sc epirk5p1bvar}. As in the case of fixed-step convergence experiments, the reference solution is computed using \texttt{ODE45} with absolute and relative tolerances equal to \num{1e-12}.
 
To aid the interpretation of the results, we organize the performance diagrams by subgroups of integrators.  
The work-precision diagrams for the Lorenz-96 system integrated over the time interval $[0, 1.8]$ [units], are shown in Figures \ref{fig:L96-work-prec-diag-kc}, \ref{fig:L96-work-prec-diag-wc}, and \ref{fig:L96-work-prec-diag-wk}. Figure \ref{fig:L96-error-vs-steps} shows the convergence diagram for varying solution tolerances and plots the global error with  respect to reference solution against the number of accepted steps. Table \ref{Table:l96-kdim} reports the (root mean square) number of Krylov vectors used by each integrator per projection at each tolerance level. Table \ref{Table:l96-rejn} gives the number of rejected time steps for each integrator at each tolerance level. Lastly, Table \ref{Table:l96-accuracy-time} interprets the global error at the end of the timespan and the timing results for each tolerance setting as the best CPU time in which certain level of accuracy is achieved by each integrator.

\begin{table}[]
\centering
\begin{tabular}{|>{\arraybackslash}m{4.2cm}|>{\centering\arraybackslash}m{0.75cm}|>{\centering\arraybackslash}m{0.75cm}|>{\centering\arraybackslash}m{0.75cm}|>{\centering\arraybackslash}m{0.75cm}|>{\centering\arraybackslash}m{0.75cm}|>{\centering\arraybackslash}m{0.75cm}|>{\centering\arraybackslash}m{0.75cm}|>{\centering\arraybackslash}m{0.75cm}|}
\hline
\backslashbox{\textbf{Integrator}}{\textbf{Tolerance}}  & $10^{-1}$ & $10^{-2}$ & $10^{-3}$ & $10^{-4}$ & $10^{-5}$ & $10^{-6}$ & $10^{-7}$ & $10^{-8}$ \\ \hline
{\sc epirkk4-classical}                          &  18  &  16  &  13  &  12  &  10  &   9  &   7  &   7    \\ \hline
{\sc epirkk4} [M = 4]                            &   4  &   4  &   4  &   4  &   4  &   4  &   4  &   4    \\ \hline
{\sc epirkk4} [M = 8]                            &   8  &   8  &   8  &   8  &   8  &   8  &   8  &   8    \\ \hline
{\sc epirkk4} [M = 16]                           &  16  &  16  &  16  &  16  &  16  &  16  &   1  &  16    \\ \hline
{\sc epirkk4} [M = 32]                           &  32  &  32  &  32  &  32  &  32  &  32  &  32  &  32    \\ \hline
{\sc epirkw3} [$\A_n = \J_n$]                    &  20  &  14  &  10  &   7  &   6  &   5  &   4  &   3    \\ \hline
{\sc epirkw3} [$\A_n = \textnormal{diag}(\J_n)$] &   0  &   0  &   0  &   0  &   0  &   0  &   0  &   0    \\ \hline
{\sc erow4}                                      &  20  &  15  &  12  &  10  &   8  &   7  &   7  &   5    \\ \hline
{\sc epirk5p1bvar}                               &  13  &   8  &   6  &   5  &   5  &   4  &   4  &   4    \\ \hline
\end{tabular}
\caption{Root mean square number of Krylov vectors per projection for each integrator applied applied to Lorenz-96 problem \eqref{eqn:Lorenz}. {\sc epirkk4} uses a fixed number of basis vectors in the Arnoldi process (indicated by the value $m$ in brackets). {\sc epirkw3} with $\A_n = \textnormal{diag}(\J_n)$ computes the $\psi$ function product  directly, without the need to perform Arnoldi. Both {\sc epirkw3} with $\A_n = \J_n$ and the classical methods use an adaptive Arnoldi process to approximate the $\psi$ function product.}
\label{Table:l96-kdim}
\end{table}

\begin{table}[]
\centering
\begin{tabular}{|>{\arraybackslash}m{4.2cm}|>{\centering\arraybackslash}m{0.75cm}|>{\centering\arraybackslash}m{0.75cm}|>{\centering\arraybackslash}m{0.75cm}|>{\centering\arraybackslash}m{0.75cm}|>{\centering\arraybackslash}m{0.75cm}|>
{\centering\arraybackslash}m{0.75cm}|>{\centering\arraybackslash}m{0.75cm}|>{\centering\arraybackslash}m{0.75cm}|}
\hline
\backslashbox{\textbf{Integrator}}{\textbf{Tolerance}}  & $10^{-1}$ & $10^{-2}$ & $10^{-3}$ & $10^{-4}$ & $10^{-5}$ & $10^{-6}$ & $10^{-7}$ & $10^{-8}$ \\ \hline
{\sc epirkk4-classical}                          &   5  &   6  &   7  &   5  &   4  &   1  &   0  &   0   \\ \hline         
{\sc epirkk4} [M = 4]                            &   5  &   7  &   6  &   5  &   3  &   1  &   0  &   0   \\ \hline
{\sc epirkk4} [M = 8]                            &   4  &   7  &   9  &   8  &   2  &   0  &   0  &   0   \\ \hline
{\sc epirkk4} [M = 16]                           &   3  &   5  &   9  &   6  &   5  &   0  &   0  &   0   \\ \hline
{\sc epirkk4} [M = 32]                           &   5  &   5  &   8  &  11  &  12  &   3  &   0  &   0   \\ \hline
{\sc epirkw3} [$\A_n = \J_n$]                    &   6  &   6  &   9  &   7  &   0  &   0  &   0  &   0   \\ \hline
{\sc epirkw3} [$\A_n = \textnormal{diag}(\J_n)$] &   5  &   6  &   9  &   2  &   0  &   0  &   0  &   0   \\ \hline
{\sc erow4}                                      &   4  &   7  &   6  &   6  &   1  &   0  &   0  &   0   \\ \hline
{\sc epirk5p1bvar}                               &   5  &   4  &   5  &   7  &   4  &   4  &   1  &   0   \\ \hline 
\end{tabular}
\caption{Total number of rejected timesteps for each integrator  applied to Lorenz-96 problem \eqref{eqn:Lorenz}. {\sc epirkk4} [M = 32] attempts to take large steps in the high tolerance region, but the stepsize controller rejects them and forces it to take about the same number of steps as other methods.}
\label{Table:l96-rejn}
\end{table}

\begin{table}[]
\centering
\small
\begin{tabular}{|>{\arraybackslash}m{3.8cm}|>{\centering\arraybackslash}m{1cm}|>{\centering\arraybackslash}m{1cm}|>{\centering\arraybackslash}m{1cm}|>{\centering\arraybackslash}m{1cm}|>{\centering\arraybackslash}m{1cm}|>{\centering\arraybackslash}m{1cm}|>{\centering\arraybackslash}m{1cm}|}
\hline
\backslashbox{\textbf{Integrator}}{\textbf{Accuracy}}  & $10^{-1}$ & $10^{-2}$ & $10^{-3}$ & $10^{-4}$ & $10^{-5}$ & $10^{-6}$ & $10^{-7}$ \\ \hline
{\sc epirkk4-classical}                              &   .025   &   ---      &     .032  &      .061  &      .083  &      .129  &    ---      \\ \hline
{\sc epirkk4} [M = 4]                                &   .007   &     .015   &     .017  &      .03   &      .054  &      .095  &    ---      \\ \hline
{\sc epirkk4} [M = 8]                                &   .008   &     .012   &     .019  &      .031  &      .057  &      .101  &    ---      \\ \hline
{\sc epirkk4} [M = 16]                               &   .01    &     .014   &     .023  &      .039  &      .07   &      .122  &    ---      \\ \hline
{\sc epirkk4} [M = 32]                               &    ---   &     .015   &     .024  &      .059  &      .104  &      .186  &    ---      \\ \hline
{\sc epirkw3} [$\A_n = \J_n$]                        &   .031   &     .045   &     .078  &      .146  &      .301  &      .535  &    ---      \\ \hline
{\sc epirkw3} [$\A_n = \textnormal{diag}(\J_n)$]     &   .013   &     .027   &     .06   &      .132  &      .283  &    ---     &    ---      \\ \hline
{\sc erow4}                                          &   .027   &     .034   &     .05   &      .073  &      .116  &      .202  &      .31    \\ \hline
{\sc epirk5p1bvar}                                   &   .039   &     .054   &     .074  &      .101  &      .149  &      .223  &    ---      \\ \hline
\end{tabular}
\caption{CPU time in which accuracy is achieved by different integrators applied to Lorenz-96 problem \eqref{eqn:Lorenz}.\label{Table:l96-accuracy-time}}
\end{table}

Figures \ref{fig:L96-work-prec-diag-kc} and \ref{fig:L96-work-prec-diag-wk} indicate that  {\sc epirkk4} is more efficient than both {\sc epirkw3} ($\A_n = \J_n$) and the classical integrators on the Lorenz-96 system \eqref{eqn:Lorenz}. This is because {\sc epirkk4} integrators use a single Krylov projection per timestep with a fixed number of basis vectors, not needing to compute any residuals along the way. On the other hand, each of the classical integrators and {\sc epirkw3} ($\A_n = \J_n$) need three Krylov projections per timestep and compute residuals to make Arnoldi adaptive. {\sc epirkk4-classical},  {\sc erow4} and {\sc epirkw3} ($\A_n = \J_n$)  build larger bases for high tolerance values, as can be seen from results reported in Table \ref{Table:l96-kdim}.

Between different fixed subspace sizes, {\sc epirkk4} with the basis size of four ($M = 4$) or eight ($M = 8$) performs better than those fixed at sixteen and thirty-two. The increased cost of performing Arnoldi for larger basis sizes does not give a proportionate advantage in terms of stability as all the {\sc epirkk4} runs end up taking approximately the same number of timesteps as indicated in Figure \ref{fig:L96-error-vs-steps}. However, runs with larger basis size do more work per timestep. 

The difference in performance between the classical integrators can be attributed to the way the computation of $\psi$ function products is carried out. {\sc epirkk4-classical} and {\sc erow4} use the augmented matrix approach described in Section \ref{sec:Implementations} and only compute the residual at specific indices, where the cost of computing the residual equals the total cost of computing the basis vectors up to that point \cite[Section 6.4]{Hochbruck_1998_exp}. On the other hand, {\sc epirk5p1bvar} computes the $\psi$ function product according to \cite{Niesen_2012_Alg919}, which involves sub-stepping and other complex arithmetic operations that do not pay off in the case of Lorenz-96 system.

As seen in Figure \ref{fig:L96-work-prec-diag-wc},  {\sc epirkw3} ($\A_n = \J_n$) behaves like a classical exponential method. The implementation uses three Krylov projections per timestep and does mathematical operations involving matrices of approximately the same size as classical exponential methods, except it only has third order convergence. As a consequence, its performance is similar to that of classical exponential methods in the low-to-medium accuracy regime. In the medium-to-high accuracy regime the performance is less competitive, and Figure \ref{fig:L96-error-vs-steps} shows that this is due to the many more steps taken by the integrator {\sc epirkw3} ($\A_n = \J_n$), as enforced by the stepsize controller. 

In contrast, {\sc epirkw3} ($\A_n = \textnormal{diag}(\J_n)$) does not need any Arnoldi iterations. Moreover, computations do not involve operations on large matrices. This makes {\sc epirkw3} ($\A_n = \textnormal{diag}(\J_n)$) very inexpensive per timestep. Even though the method appears to be efficient for low accuracy solutions, relative error in the region is $\mathcal{O}(1)$ as is for all other methods. For medium-to-high accuracy solutions, the method takes significantly more number of timesteps, as is shown in Figure \ref{fig:L96-error-vs-steps}, explaining the tilt towards higher CPU time. In the case of {\sc epirkw3} ($\A_n = \textnormal{diag}(\J_n)$), we note that the larger number of timesteps for high accuracy solutions is due to a combination of lower order of the $W$-method and insufficient stability as the approximation to the exact Jacobian is not fully representative of its structure. Good Jacobian approximations for the problem at hand are important in order to take full advantage of the capabilities of EPIRK-$W$ methods.

\subsection{Shallow water equations on the sphere}

The Shallow Water Equations (SWE) on the sphere \cite{Neta_1997_Analysis,Daescu2004,rao2015posteriori}  are given by:
\begin{eqnarray}
\frac{\partial u}{\partial t} + \frac{1}{a \,\cos \theta} \left[ u \frac{\partial u}{\partial \lambda} + v \,\cos \theta\, \frac{\partial u}{\partial \theta} \right] -  \left(f + \frac{u}{a}\, \tan \theta \right) v + \frac{g}{a \,\cos \theta} \frac{\partial h}{\partial \lambda} &=& 0,\nonumber \\
\label{eqn:SWE}
\frac{\partial v}{\partial t} + \frac{1}{a \,\cos \theta} \left[ u \frac{\partial u}{\partial \lambda} + v \,\cos \theta\, \frac{\partial v}{\partial \theta} \right] +  \left(f + \frac{u}{a} \,\tan \theta \right) u + \frac{g}{a} \frac{\partial h}{\partial \theta} &=& 0, \\
\frac{\partial h}{\partial t} + \frac{1}{a \,\cos \theta} \left[  \frac{\partial (hu)}{\partial \lambda} +  \frac{\partial (hv \,\cos \theta)}{\partial \theta} \right]  &=& 0, \nonumber
\end{eqnarray}
where $f$ is the Coriolis parameter given by $f = 2 \Omega \sin \theta$, with $\Omega$ being the angular speed of the rotation of the earth, $h$ is the height of the atmosphere, $u$ and $v$ are zonal and meridional wind components, $a$ is the radius of the earth and $g$ is the gravitational constant, $\theta$ and $\lambda$ are latitudinal and longitudinal directions respectively. Initial conditions for the shallow water equations are same as those that appear in \cite[Section 6]{Neta_1997_Analysis}. The model was discretized in space with 72 gridpoints along each latitude and 36 gridpoints along each longitude. The system has three variables, $u$, $v$ and $h$, that need to be computed at each grid point making the total dimension of the semi-discrete system $N=7776$. 

We ran variable time-stepping experiments on the shallow water equations model over the integration window $[0, 28800]$ time units, which roughly corresponds to one day of real time. The work precision diagrams appear in Figures \ref{fig:SWE-work-prec-diag-kc}, \ref{fig:SWE-work-prec-diag-wc}, and \ref{fig:SWE-work-prec-diag-wk}. Figure \ref{fig:SWE-error-vs-steps} is the convergence diagram for varying solution tolerances and plots the global error with  respect to reference solution against the number of accepted steps. Table \ref{Table:swe-kdim} details the (root mean square) number of Krylov vectors per projection used by each integrator at each tolerance level. Table \ref{Table:swe-rejn} specifies the number of rejected timesteps for each integrator at each tolerance level. 

\begin{figure}[!htb]
    \centering
    \begin{subfigure}{.45\textwidth}
        \centering
	    \includegraphics[scale=0.25]{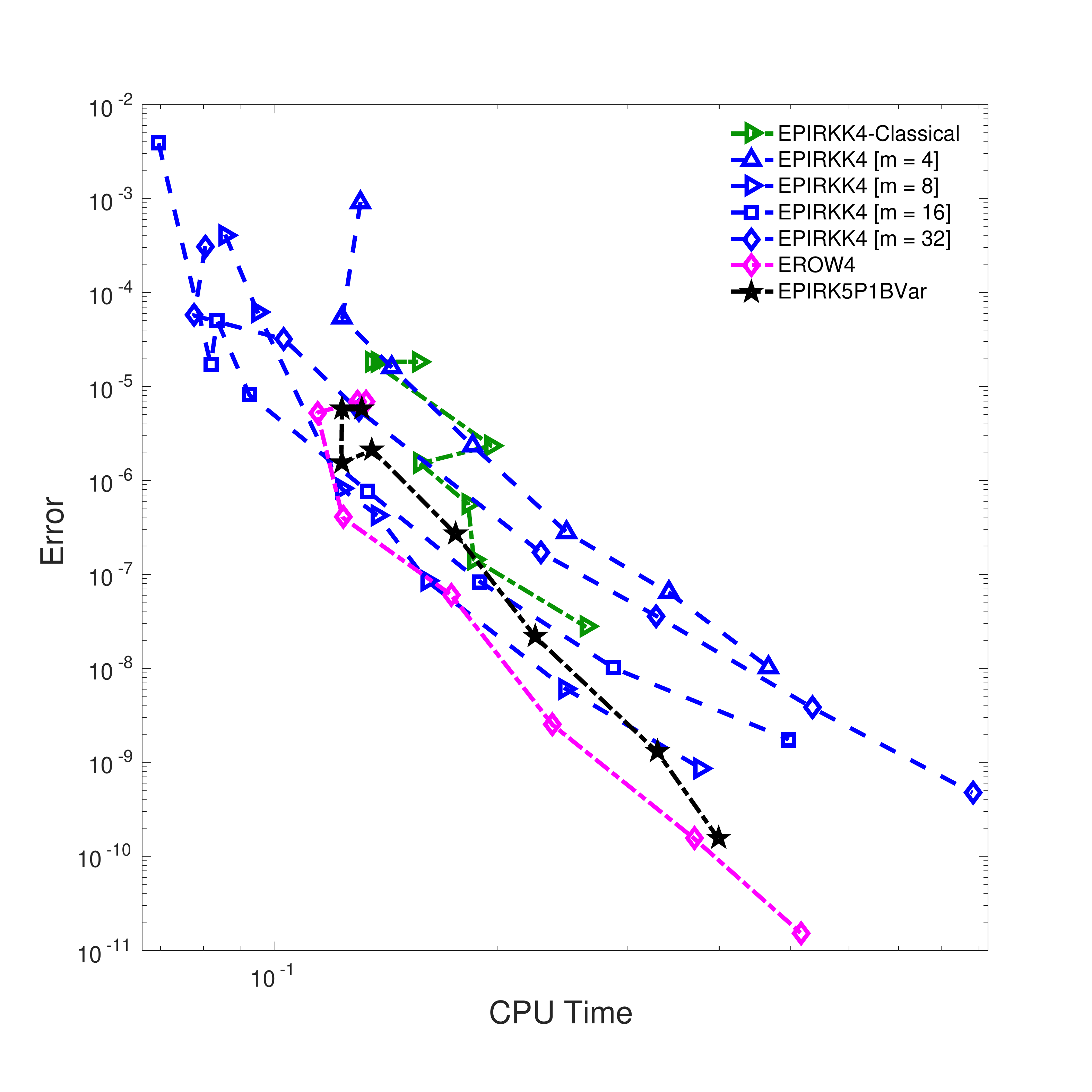}
   	    \caption{K-type versus classical methods}
		\label{fig:SWE-work-prec-diag-kc}
    \end{subfigure}%
    \begin{subfigure}{0.45\textwidth}
        \centering
	    \includegraphics[scale=0.25]{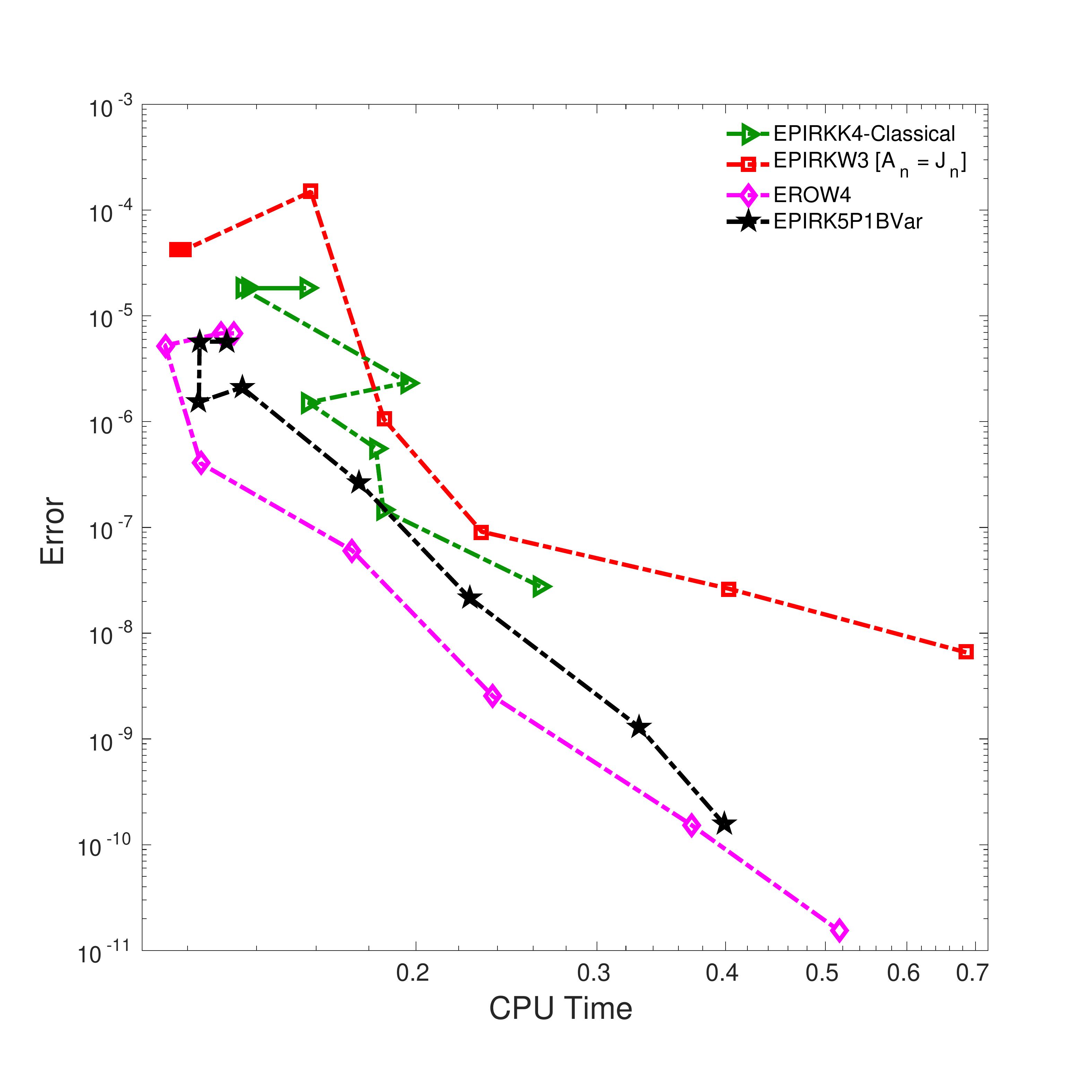}
   	    \caption{W-type versus classical methods}
		\label{fig:SWE-work-prec-diag-wc}
    \end{subfigure}%
    \caption{Work-precision diagrams for different methods applied to Shallow Water Equations problem \eqref{eqn:SWE}.}
\end{figure}

\begin{figure}[!htb]
    \centering
    \begin{subfigure}{.45\textwidth}
        \centering
	    \includegraphics[scale=0.25]{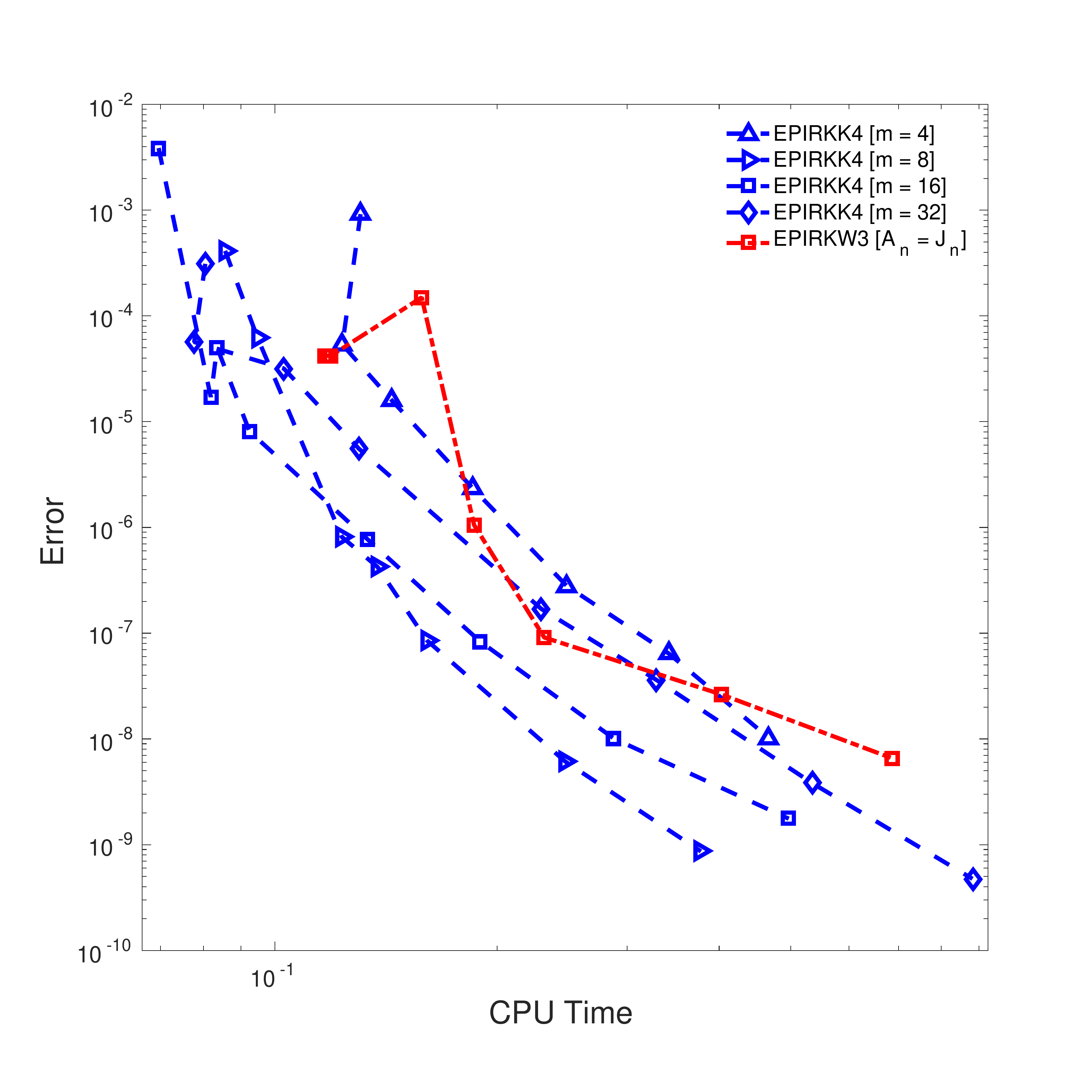}
   	    \caption{W-type versus K-type methods}
		\label{fig:SWE-work-prec-diag-wk}
    \end{subfigure}%
    \begin{subfigure}{0.45\textwidth}
        \centering
		\includegraphics[scale=0.25]{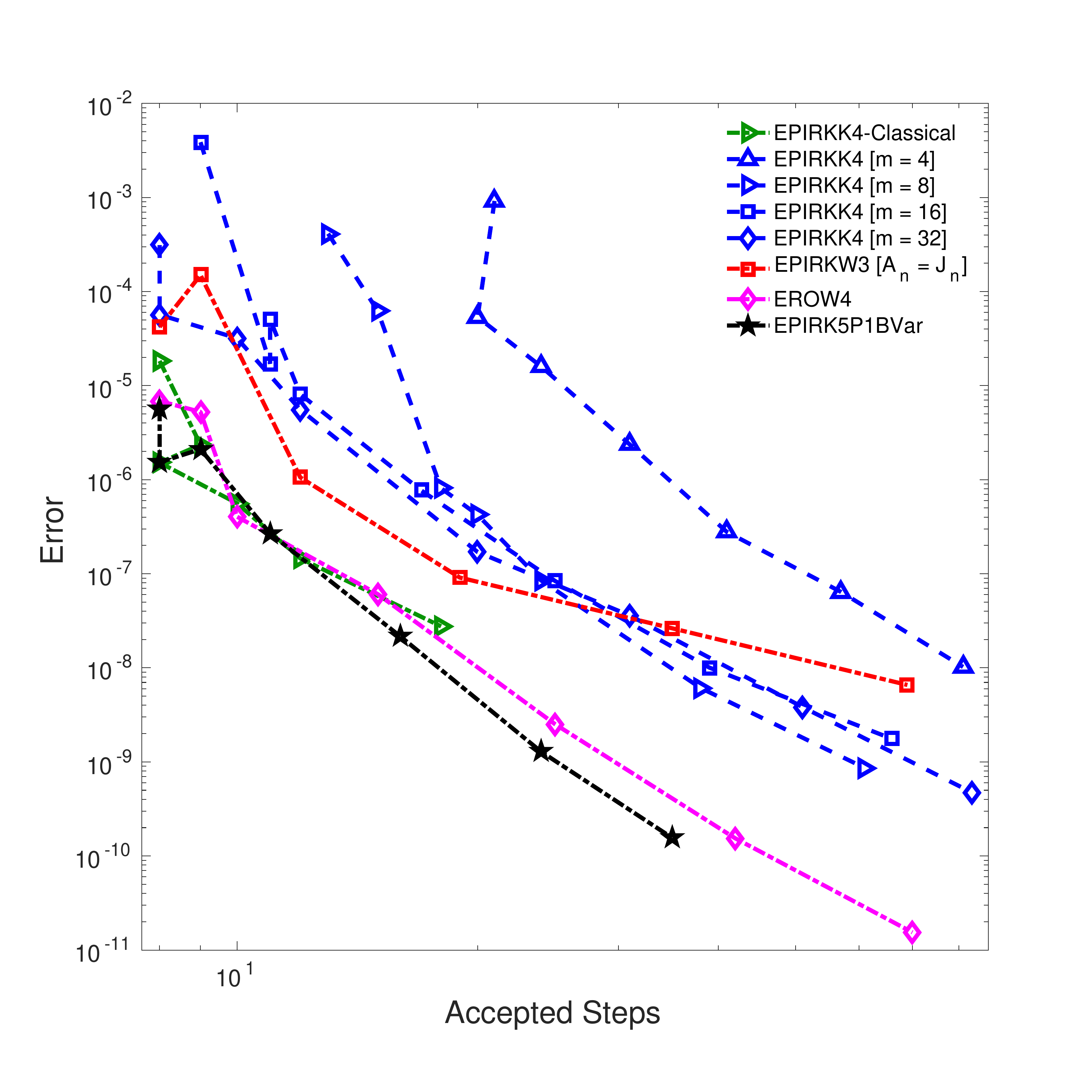}
		\caption{Convergence diagram}
		\label{fig:SWE-error-vs-steps}
    \end{subfigure}%
    \caption{Work-precision and convergence diagrams for different methods applied to Shallow Water Equations problem \eqref{eqn:SWE}.}
\end{figure}

\begin{table}[]
\centering
\begin{tabular}{|>{\arraybackslash}m{4cm}|>{\centering\arraybackslash}m{0.75cm}|>{\centering\arraybackslash}m{0.75cm}|>{\centering\arraybackslash}m{0.75cm}|>{\centering\arraybackslash}m{0.75cm}|>{\centering\arraybackslash}m{0.75cm}|>{\centering\arraybackslash}m{0.75cm}|>{\centering\arraybackslash}m{0.75cm}|>{\centering\arraybackslash}m{0.75cm}|}
\hline
\backslashbox{\textbf{Integrator}}{\textbf{Tolerance}}  & $10^{-1}$ & $10^{-2}$ & $10^{-3}$ & $10^{-4}$ & $10^{-5}$ & $10^{-6}$ & $10^{-7}$ & $10^{-8}$ \\ \hline
{\sc epirkk4-classical}       &  25  &  25  &  25  &  27  &  26  &  21  &  16  &  12 \\ \hline
{\sc epirkk4} [M = 4]         &   4  &   4  &   4  &   4  &   4  &   4  &   4  &   4 \\ \hline
{\sc epirkk4} [M = 8]         &   8  &   8  &   8  &   8  &   8  &   8  &   8  &   8 \\ \hline
{\sc epirkk4} [M = 16]        &  16  &  16  &  16  &  16  &  16  &  16  &  16  &  16 \\ \hline
{\sc epirkk4} [M = 32]        &  32  &  32  &  32  &  32  &  32  &  32  &  32  &  32 \\ \hline
{\sc epirkw3} [$\A_n = \J_n$] &  21  &  21  &  21  &  20  &  14  &   8  &   5  &   4 \\ \hline
{\sc erow4}                   &  30  &  30  &  22  &  22  &  16  &  12  &   8  &   7 \\ \hline
{\sc epirk5p1bvar}            &  57  &  53  &  38  &  30  &  23  &  19  &  16  &  14 \\ \hline   
\end{tabular}
\caption{Root mean square number of Krylov vectors per projection for each integrator applied to Shallow Water Equations problem \eqref{eqn:SWE}. {\sc epirkk4} uses fixed number of basis vectors in the Arnoldi process as indicated in brackets succeeding the name of the method. {\sc epirkw3} with $A_n = diag(J_n)$ directly computes the $\psi$ function product not needing to perform Arnoldi. All other methods, both {\sc epirkw3} with $A_n = J_n$ and the classical methods use an adaptive Arnoldi process to approximate the $\psi$ function product.
\label{Table:swe-kdim}}
\end{table}

\begin{table}[]
\centering
\begin{tabular}{|>{\arraybackslash}m{4cm}|>{\centering\arraybackslash}m{0.75cm}|>{\centering\arraybackslash}m{0.75cm}|>{\centering\arraybackslash}m{0.75cm}|>{\centering\arraybackslash}m{0.75cm}|>{\centering\arraybackslash}m{0.75cm}|>{\centering\arraybackslash}m{0.75cm}|>{\centering\arraybackslash}m{0.75cm}|>{\centering\arraybackslash}m{0.75cm}|}
\hline
\backslashbox{\textbf{Integrator}}{\textbf{Tolerance}}  & $10^{-1}$ & $10^{-2}$ & $10^{-3}$ & $10^{-4}$ & $10^{-5}$ & $10^{-6}$ & $10^{-7}$ & $10^{-8}$ \\ \hline
{\sc epirkk4-classical}       &   0  &   0  &   0  &   2  &   1  &   2  &   2  &   4  \\ \hline                                                                                                                                            
{\sc epirkk4} [M = 4]         &  12  &   6  &   7  &   5  &   8  &  11  &  13  &  10  \\ \hline                                                                                                                                            
{\sc epirkk4} [M = 8]         &  15  &   3  &   3  &   6  &   6  &   6  &   7  &   6  \\ \hline                                                                                                                                            
{\sc epirkk4} [M = 16]        &   3  &   2  &   2  &   2  &   4  &   6  &   7  &  18  \\ \hline                                                                                                                                            
{\sc epirkk4} [M = 32]        &   1  &   0  &   1  &   2  &   6  &   7  &   9  &  16  \\ \hline                                                                                                                                            
{\sc epirkw3} [$\A_n = \J_n$] &   0  &   0  &   0  &   1  &   3  &   3  &   7  &   2  \\ \hline                                                                                                                                            
{\sc erow4}                   &   0  &   0  &   1  &   0  &   3  &   6  &   9  &   8  \\ \hline                                                                                                                                            
{\sc epirk5p1bvar}            &   0  &   0  &   0  &   0  &   1  &   2  &   6  &   5  \\ \hline      
\end{tabular}
\caption{Total number of rejected timesteps for each integrator  applied to Shallow Water Equations problem \eqref{eqn:SWE}.\label{Table:swe-rejn}}
\end{table}

\begin{table}[]
\centering
\small
\begin{tabular}{|>{\arraybackslash}m{3.8cm}|>{\centering\arraybackslash}m{1cm}|>{\centering\arraybackslash}m{1cm}|>{\centering\arraybackslash}m{1cm}|>{\centering\arraybackslash}m{1cm}|>{\centering\arraybackslash}m{1cm}|>{\centering\arraybackslash}m{1cm}|>{\centering\arraybackslash}m{1cm}|>{\centering\arraybackslash}m{1cm}|}
\hline
\backslashbox{\textbf{Integrator}}{\textbf{Accuracy}}  & $10^{-2}$ & $10^{-3}$ & $10^{-4}$ & $10^{-5}$ & $10^{-6}$ & $10^{-7}$ & $10^{-8}$ & $10^{-9}$ \\ \hline
{\sc epirkk4-classical}          &     ---   &   ---      &     .135  &      .157  &      .183  &      .264  &    ---     &    ---    \\ \hline
{\sc epirkk4} [M = 4]            &     ---   &     .131   &     .123  &      .185  &      .249  &      .342  &    ---     &    ---    \\ \hline
{\sc epirkk4} [M = 8]            &     ---   &     .086   &     .095  &    ---     &      .123  &      .161  &      .247  &      .378 \\ \hline
{\sc epirkk4} [M = 16]           &    .07    &   ---      &     .082  &      .092  &      .134  &      .19   &      .497  &    ---    \\ \hline
{\sc epirkk4} [M = 32]           &     ---   &     .08    &     .078  &      .13   &      .23   &      .329  &      .534  &      .882 \\ \hline
{\sc epirkw3} [$\A_n = \J_n$]    &     ---   &     .158   &     .117  &      .186  &    ---     &      .232  &      .685  &    ---    \\ \hline
{\sc erow4}                      &     ---   &   ---      &   ---     &      .114  &      .124  &      .173  &      .238  &      .37  \\ \hline
{\sc epirk5p1bvar}               &     ---   &   ---      &   ---     &      .123  &      .176  &      .225  &      .33   &      .399 \\ \hline
\end{tabular}
\caption{CPU time in which accuracy is achieved by different integrators applied to Shallow Water Equations problem \eqref{eqn:SWE}. In our experiments, only {\sc erow4} achieves an accuracy of 1e-10 in   0.515 seconds.\label{Table:swe-accuracy-time}}
\end{table}

Figure \ref{fig:SWE-work-prec-diag-kc} suggests that for the Shallow Water Equations on sphere problem, {\sc epirkk4} with basis size set to eight (sixteen) is able to achieve solutions in the high tolerance region unlike the classical integrators, and in appropriately less amount of time. {\sc epirkk4} with sixteen basis vectors continues to be atleast 23\% more efficient until a solution accuracy of \num{1e-5} before becoming slightly more expensive than the classical methods. On the other hand, {\sc epirkk4} with basis size eight remains to be about as good as the classical methods for the entire tolerance spectrum. It is rather interesting that increasing the basis size to sixteen or thirty-two does not significantly improve the performance of {\sc epirkk4} integrators and in some cases makes them perform poorly. This can be explained based on Figure \ref{fig:SWE-error-vs-steps} where executions with eight, sixteen and thirty-two basis vectors end up taking approximately the same number of timesteps beyond a tolerance setting of \num{1e-6} making executions with larger basis size cost more per timestep. For these runs of {\sc epirkk4}, attempts to increase the stepsize are faced with step rejections as shown in Table \ref{Table:swe-rejn}.

Among different executions of {\sc epirkk4}, one with four basis vectors performs the poorest as seen in Figure \ref{fig:SWE-work-prec-diag-kc} and \ref{fig:SWE-work-prec-diag-wk}. This is because four basis vectors in the Krylov-subspace is insufficient to capture the dynamics of the ODE system under consideration forcing the integrator to take significantly more timesteps than other executions of {\sc epirkk4} (see Figure \ref{fig:SWE-error-vs-steps}).

{\sc epirkw3} method with $\A_n=\J_n$ performs like classical exponential integrators as shown in Figure \ref{fig:SWE-work-prec-diag-wc} and this is explained in Section \ref{sec:L96_section}. We notice that the error controller forces the method to take larger number of time steps as accuracy increases (see Figure \ref{fig:SWE-error-vs-steps}) explaining the devitation from the classical integrators for lower tolerances. 

For the Shallow Water Equation system, {\sc epirkw3}  with $\A_n = \textnormal{diag}(\J_n)$ did not converge for any tolerance setting as the method was highly unstable. This can be explained by the structure of the exact Jacobian versus that of the approximate Jacobian. The exact Jacobian is not diagonally dominant and has an incomplete banded structure. On the other hand, the approximation only includes the diagonal and is diagonally dominant. Therefore, we exclude {\sc epirkw3} with $\A_n = \textnormal{diag}(\J_n)$ from our current discussion.

Table \ref{Table:swe-accuracy-time} presents an alternative approach to interpreting the work precision diagrams. It shows the best CPU time by which each integrator could achieve a certain level of accuracy.

\subsection{The Allen-Cahn problem}

We consider the Allen-Cahn equation \cite{Allen_1979_ALCN}:
\begin{equation}
\label{eqn:AllenCahn}
\frac{\partial u}{\partial t} = \alpha\, \nabla^2 u + \gamma\, (u - u^3),\quad (x,y) \in [0, 1] \times [0, 1]~\textnormal{(space units)}, \quad t \in [0, 1.2]~\textnormal{(time units)}.
\end{equation}
For our experiments we set $\alpha = 0.01$ and $\gamma = 1.0$. The model has homogeneous Neumann boundary conditions and the initial conditions are $u(t = 0) = 0.4 + 0.1\,(x + y) + 0.1\, \sin(10x)\,\sin(20y)$. Discretization is done in the spatial domain using finite difference for two different grid sizes -- $64 \times 64$ and $256 \times 256$.
The work-precision diagrams from variable time-stepping are shown in Figures \ref{fig:ALCN-work-prec-diag-kc}, \ref{fig:ALCN-work-prec-diag-wc}, and \ref{fig:ALCN-work-prec-diag-wk} for grid of size $64 \times 64$. Figure \ref{fig:ALCN-error-vs-steps} shows the corresponding convergence diagram and plots the global error with respect to the reference against the number of accepted steps. Table \ref{Table:alcn-kdim} shows the (root mean square) number of Krylov vectors per projection used by each method at different tolerance levels. Table \ref{Table:alcn-accuracy-time} is a reinterpretation of the work precision diagrams in a tabular form where each entry represents the best CPU time in which the solution at the end of the timespan attains a certain level of accuracy. Figures \ref{fig:ALCN-work-prec-diag-kc-256}, \ref{fig:ALCN-work-prec-diag-wc-256}, \ref{fig:ALCN-work-prec-diag-wk-256} and \ref{fig:ALCN-error-vs-steps-256}, and Tables \ref{Table:alcn-kdim-256}, \ref{Table:alcn-rejn-256} and \ref{Table:alcn-accuracy-time-256} correspond to results from performing the same experiment on a  grid of size $256 \times 256$.

\begin{figure}[!htb]
    \centering
    \begin{subfigure}{.45\textwidth}
        \centering
	    \includegraphics[scale=0.25]{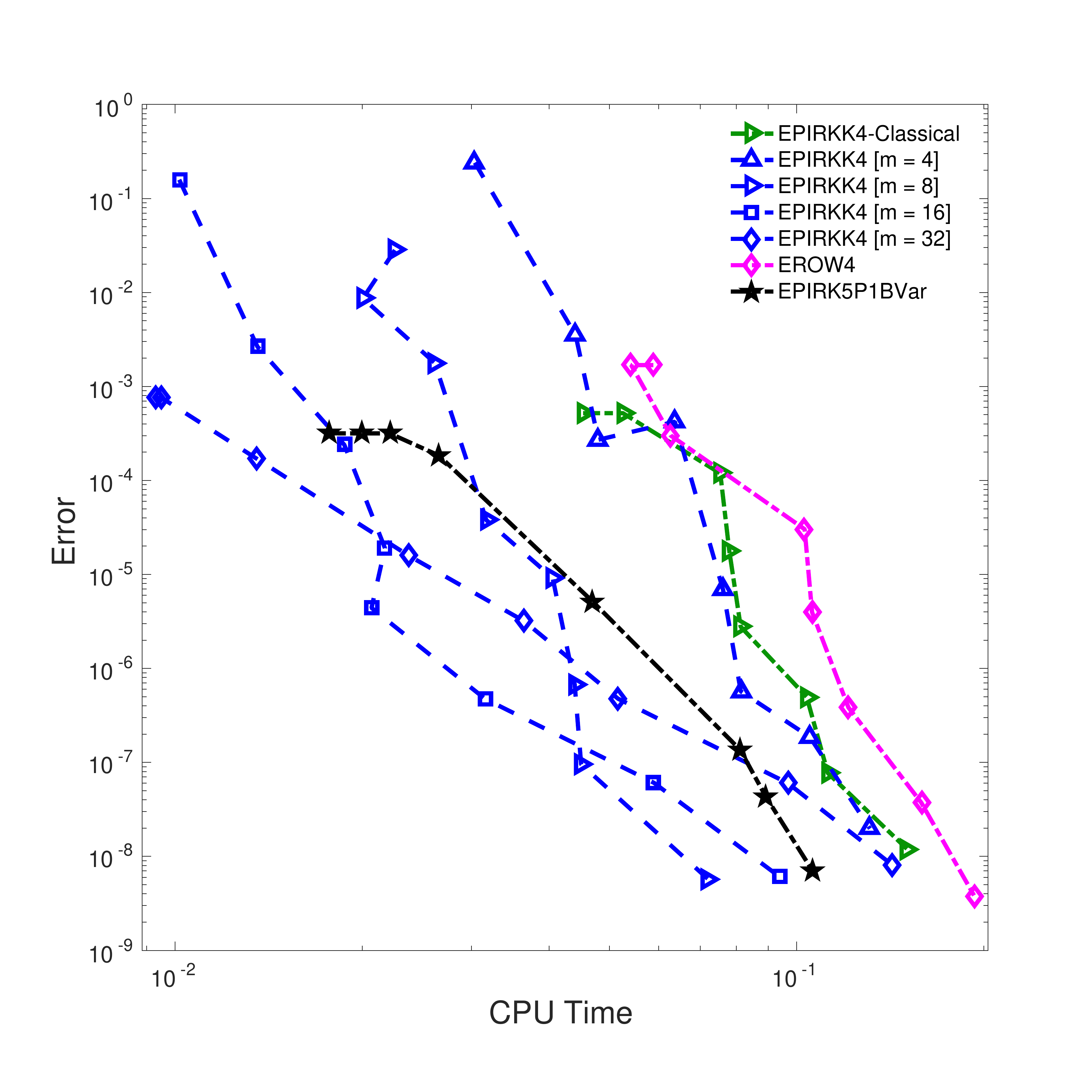}
	    \caption{K-type versus classical methods}
            \label{fig:ALCN-work-prec-diag-kc}
    \end{subfigure}
    ~~
    \begin{subfigure}{0.45\textwidth}
        \centering
	    \includegraphics[scale=0.25]{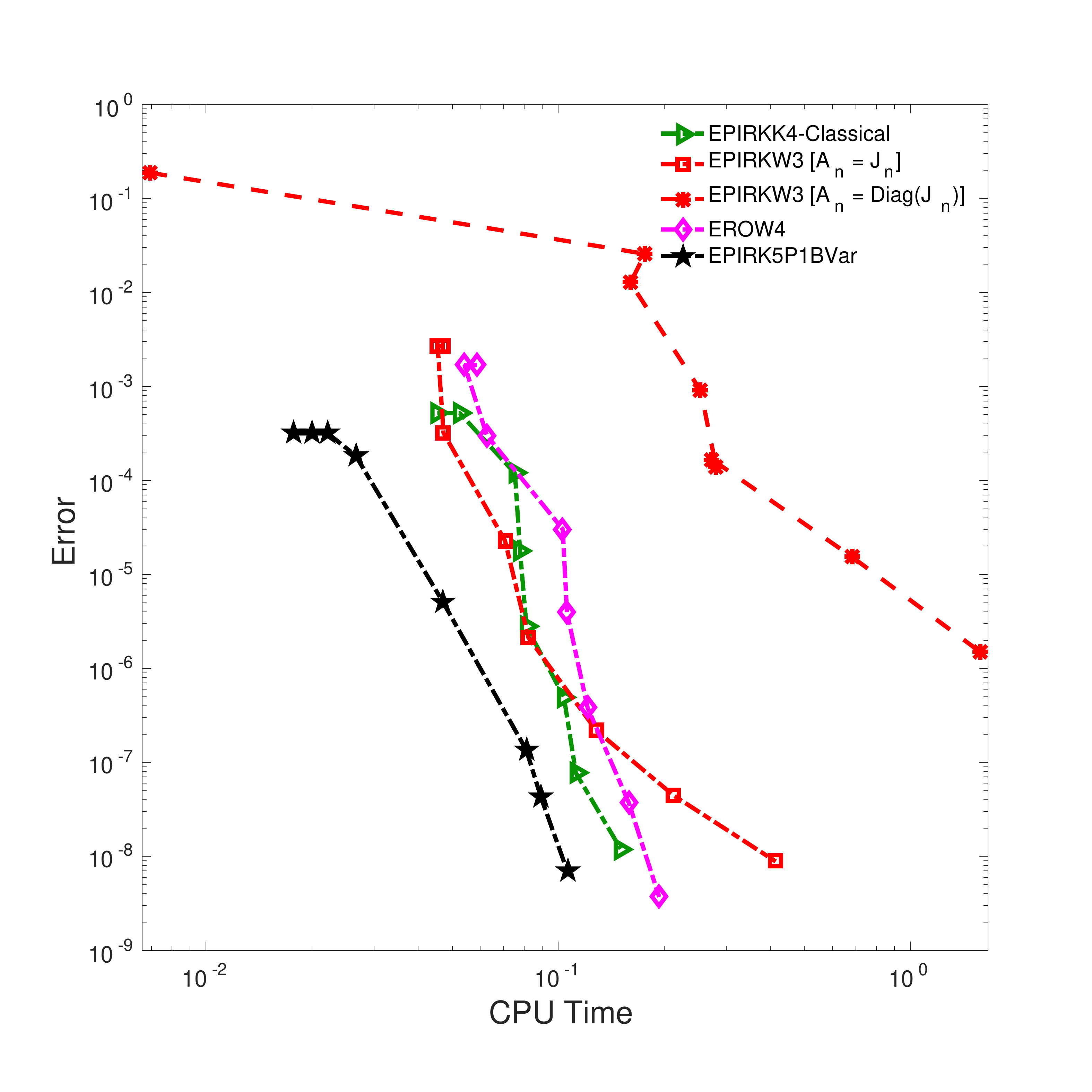}
	    \caption{W-type versus classical methods}
             \label{fig:ALCN-work-prec-diag-wc}
    \end{subfigure}%
    \caption{Work-precision diagrams for different methods applied to Allen-Cahn problem \eqref{eqn:AllenCahn} on a $64 \times 64$ grid.}
\end{figure}

\begin{figure}[!htb]
    \centering
    \begin{subfigure}{.45\textwidth}
        \centering
	    \includegraphics[scale=0.25]{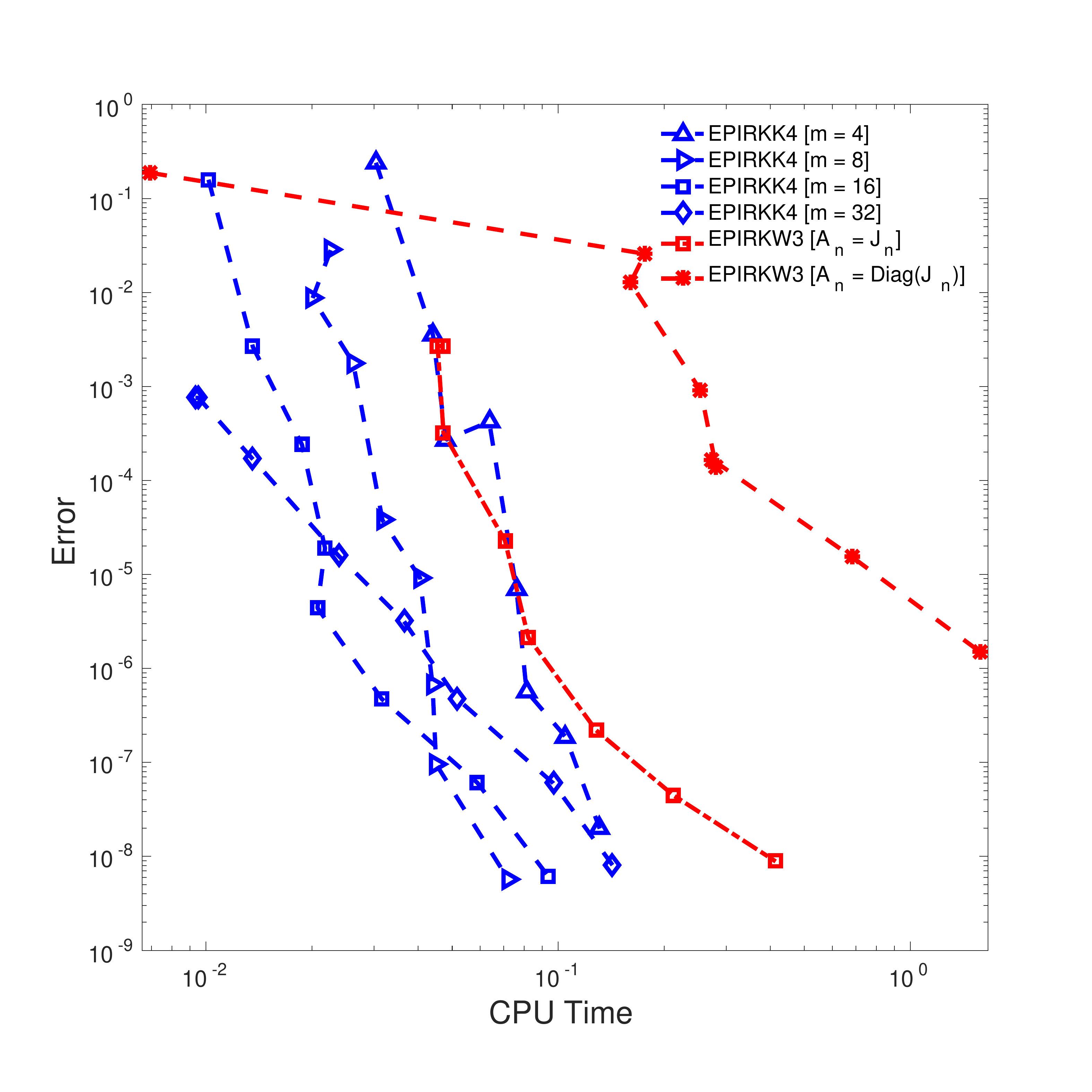}
  	    \caption{W-type versus K-type methods}
		\label{fig:ALCN-work-prec-diag-wk}
    \end{subfigure}%
    \begin{subfigure}{0.45\textwidth}
        \centering
		\includegraphics[scale=0.25]{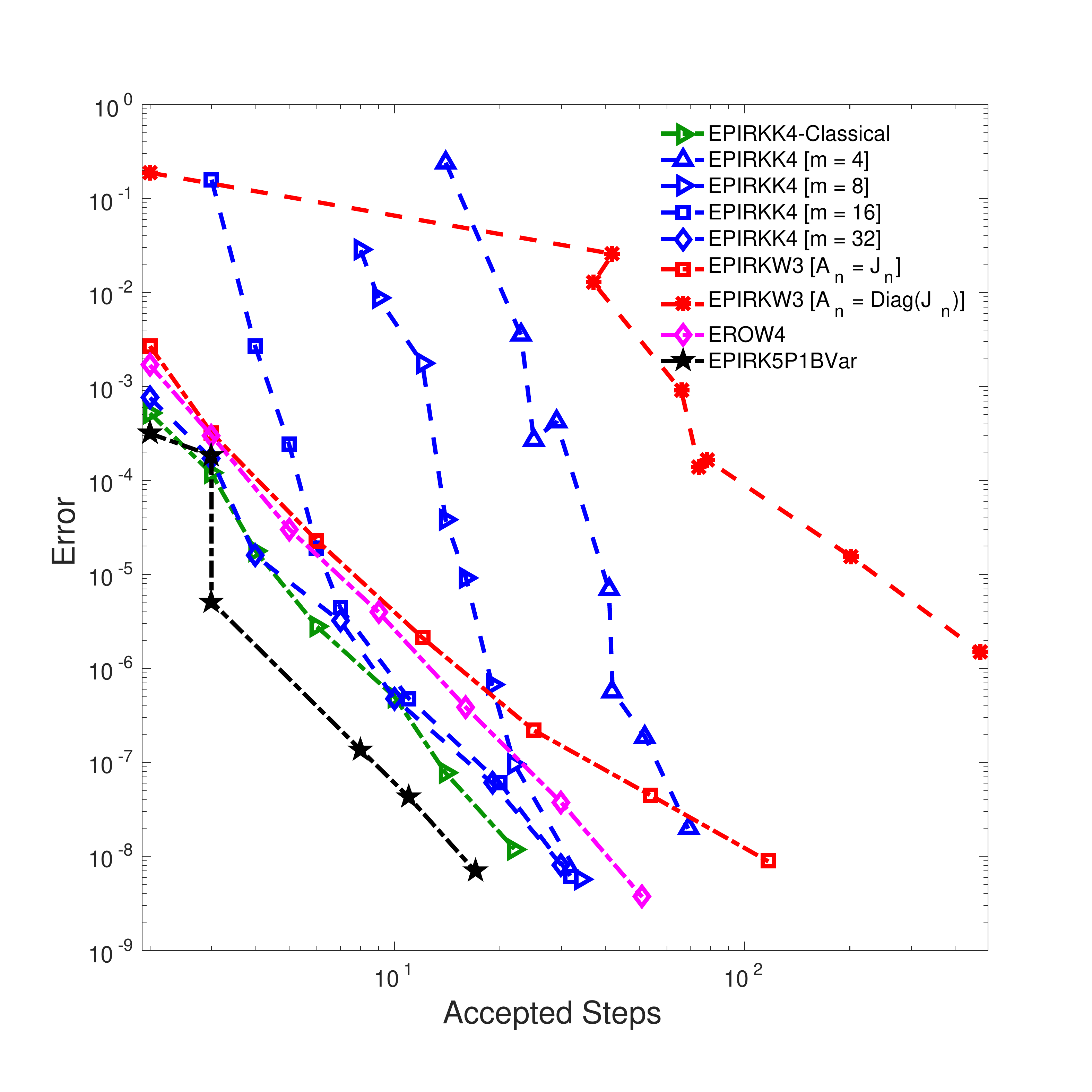}
	    \caption{Convergence diagram}
		\label{fig:ALCN-error-vs-steps}
    \end{subfigure}%
    \caption{Work-precision diagrams and convergence diagrams for different methods applied to Allen-Cahn problem \eqref{eqn:AllenCahn} on a $64 \times 64$ grid.}
\end{figure}

\begin{table}[]
\centering
\begin{tabular}{|>{\arraybackslash}m{4.2cm}|>{\centering\arraybackslash}m{0.75cm}|>{\centering\arraybackslash}m{0.75cm}|>{\centering\arraybackslash}m{0.75cm}|>{\centering\arraybackslash}m{0.75cm}|>{\centering\arraybackslash}m{0.75cm}|>{\centering\arraybackslash}m{0.75cm}|>{\centering\arraybackslash}m{0.75cm}|>{\centering\arraybackslash}m{0.75cm}|}
\hline
\backslashbox{\textbf{Integrator}}{\textbf{Tolerance}}  & $10^{-1}$ & $10^{-2}$ & $10^{-3}$ & $10^{-4}$ & $10^{-5}$ & $10^{-6}$ & $10^{-7}$ & $10^{-8}$ \\ \hline
{\sc epirkk4-classical}                          &  49  &  49  &  42  &  37  &  30  &  23  &  19  &  16   \\ \hline                                                                                                                        
{\sc epirkk4} [M = 4]                            &   4  &   4  &   4  &   4  &   4  &   4  &   4  &   4   \\ \hline                                                                                                                        
{\sc epirkk4} [M = 8]                            &   8  &   8  &   8  &   8  &   8  &   8  &   8  &   8   \\ \hline                                                                                                                        
{\sc epirkk4} [M = 16]                           &  16  &  16  &  16  &  16  &  16  &  16  &  16  &  16   \\ \hline                                                                                                                        
{\sc epirkk4} [M = 32]                           &  32  &  32  &  32  &  32  &  32  &  32  &  32  &  32   \\ \hline                                                                                                                        
{\sc epirkw3} [$\A_n = \J_n$]                    &  48  &  48  &  37  &  26  &  17  &  11  &   7  &   5   \\ \hline                                                                                                                        
{\sc epirkw3} [$\A_n = \textnormal{diag}(\J_n)$] &   0  &   0  &   0  &   0  &   0  &   0  &   0  &   0   \\ \hline                                                                                                                        
{\sc erow4}                                      &  57  &  57  &  47  &  39  &  29  &  21  &  15  &  12   \\ \hline                                                                                                                        
{\sc epirk5p1bvar}                               &  32  &  47  &  52  &  43  &  31  &  21  &  18  &  15   \\ \hline  
\end{tabular}
\caption{Root mean square number of Krylov vectors per projection for each integrator applied to Allen-Cahn problem \eqref{eqn:AllenCahn} on a $64 \times 64$ grid.{\sc epirkk4} uses fixed number of basis vectors in the Arnoldi process as indicated in brackets succeeding the name of the method. {\sc epirkw3} with $A_n = diag(J_n)$ directly computes the $\psi$ function product not needing to perform Arnoldi. All other methods, both {\sc epirkw3} with $A_n = J_n$ and the classical methods use an adaptive Arnoldi process to approximate the $\psi$ function product.\label{Table:alcn-kdim}}
\end{table}

\begin{table}[]
\centering
\begin{tabular}{|>{\arraybackslash}m{4.2cm}|>{\centering\arraybackslash}m{0.75cm}|>{\centering\arraybackslash}m{0.75cm}|>{\centering\arraybackslash}m{0.75cm}|>{\centering\arraybackslash}m{0.75cm}|>{\centering\arraybackslash}m{0.75cm}|>
{\centering\arraybackslash}m{0.75cm}|>{\centering\arraybackslash}m{0.75cm}|>{\centering\arraybackslash}m{0.75cm}|}
\hline
\backslashbox{\textbf{Integrator}}{\textbf{Tolerance}}  & $10^{-1}$ & $10^{-2}$ & $10^{-3}$ & $10^{-4}$ & $10^{-5}$ & $10^{-6}$ & $10^{-7}$ & $10^{-8}$ \\ \hline
{\sc epirkk4-classical}                           &   0  &   0  &   1  &   0  &   0  &   1  &   0  &   0  \\ \hline                                                                                                                        
{\sc epirkk4} [M = 4]                             &   4  &   2  &   2  &  11  &   3  &   6  &  14  &   8  \\ \hline                                                                                                                        
{\sc epirkk4} [M = 8]                             &   6  &   1  &   1  &   3  &   6  &   5  &   0  &   0  \\ \hline                                                                                                                        
{\sc epirkk4} [M = 16]                            &   1  &   1  &   2  &   2  &   0  &   0  &   0  &   0  \\ \hline                                                                                                                        
{\sc epirkk4} [M = 32]                            &   0  &   0  &   0  &   0  &   0  &   0  &   0  &   0  \\ \hline                                                                                                                        
{\sc epirkw3} [$\A_n = \J_n$]                     &   0  &   0  &   0  &   0  &   0  &   0  &   0  &   0  \\ \hline                                                                                                                        
{\sc epirkw3} [$\A_n = \textnormal{diag}(\J_n)$]  &   0  &  15  &  15  &  13  &  11  &   0  &   0  &   0  \\ \hline                                                                                                                        
{\sc erow4}                                       &   0  &   0  &   0  &   1  &   0  &   0  &   0  &   0  \\ \hline                                                                                                                        
{\sc epirk5p1bvar}                                &   0  &   0  &   0  &   0  &   0  &   1  &   0  &   0  \\ \hline  
\end{tabular}
\caption{Total number of rejected timesteps for each integrator applied to Allen-Cahn problem \eqref{eqn:AllenCahn} on a $64 \times 64$ grid.{\sc epirkk4} executions with basis sizes set to four or eight suffer from step rejections and take large number of timesteps for all different tolerance levels. {\sc epirkw3} with $\A_n = \textnormal{diag}(\J_n)$ also shows similar behavior. This hints at the poor stability of these methods. \label{Table:alcn-rejn}}
\end{table}

\begin{table}[]
\centering
\small
\begin{tabular}{|>{\arraybackslash}m{3.8cm}|>{\centering\arraybackslash}m{1cm}|>{\centering\arraybackslash}m{1cm}|>{\centering\arraybackslash}m{1cm}|>{\centering\arraybackslash}m{1cm}|>{\centering\arraybackslash}m{1cm}|>{\centering\arraybackslash}m{1cm}|>{\centering\arraybackslash}m{1cm}|>{\centering\arraybackslash}m{1cm}|}
\hline
\backslashbox{\textbf{Integrator}}{\textbf{Accuracy}}  & $10^{-1}$ & $10^{-2}$ & $10^{-3}$ & $10^{-4}$ & $10^{-5}$ & $10^{-6}$ & $10^{-7}$ & $10^{-8}$ \\ \hline
{\sc epirkk4-classical}                             &  ---     &    ---     &      .045  &      .078  &      .081  &      .104  &      .112  &      ---  \\ \hline
{\sc epirkk4} [M = 4]                               &  ---     &      .044  &      .048  &    ---     &      .076  &      .081  &      .131  &      ---  \\ \hline
{\sc epirkk4} [M = 8]                               &    .023  &      .02   &    ---     &      .032  &      .041  &      .044  &      .045  &     .072  \\ \hline
{\sc epirkk4} [M = 16]                              &  ---     &      .014  &      .019  &      .022  &      .021  &      .032  &      .059  &     .094  \\ \hline
{\sc epirkk4} [M = 32]                              &  ---     &    ---     &      .009  &      .024  &      .036  &      .052  &      .097  &     .143  \\ \hline
{\sc epirkw3} [$\A_n = \J_n$]                       &  ---     &      .045  &      .047  &      .071  &      .082  &      .128  &      .212  &     .414  \\ \hline
{\sc epirkw3} [$\A_n = \textnormal{diag}(\J_n)$]    &    .16   &    ---     &      .254  &      .683  &     1.583  &    ---     &    ---     &      ---  \\ \hline
{\sc erow4}                                         &  ---     &      .054  &      .063  &      .103  &      .106  &      .121  &      .159  &     .193  \\ \hline
{\sc epirk5p1bvar}                                  &  ---     &    ---     &      .018  &    ---     &      .047  &      .081  &      .089  &     .106  \\ \hline
\end{tabular}
\caption{CPU time in which accuracy is achieved by different integrators applied to Allen-Cahn problem \eqref{eqn:AllenCahn} on a $64 \times 64$ grid.}
\label{Table:alcn-accuracy-time}
\end{table}

\begin{figure}[!htb]
    \centering
    \begin{subfigure}{.45\textwidth}
        \centering
	    \includegraphics[scale=0.25]{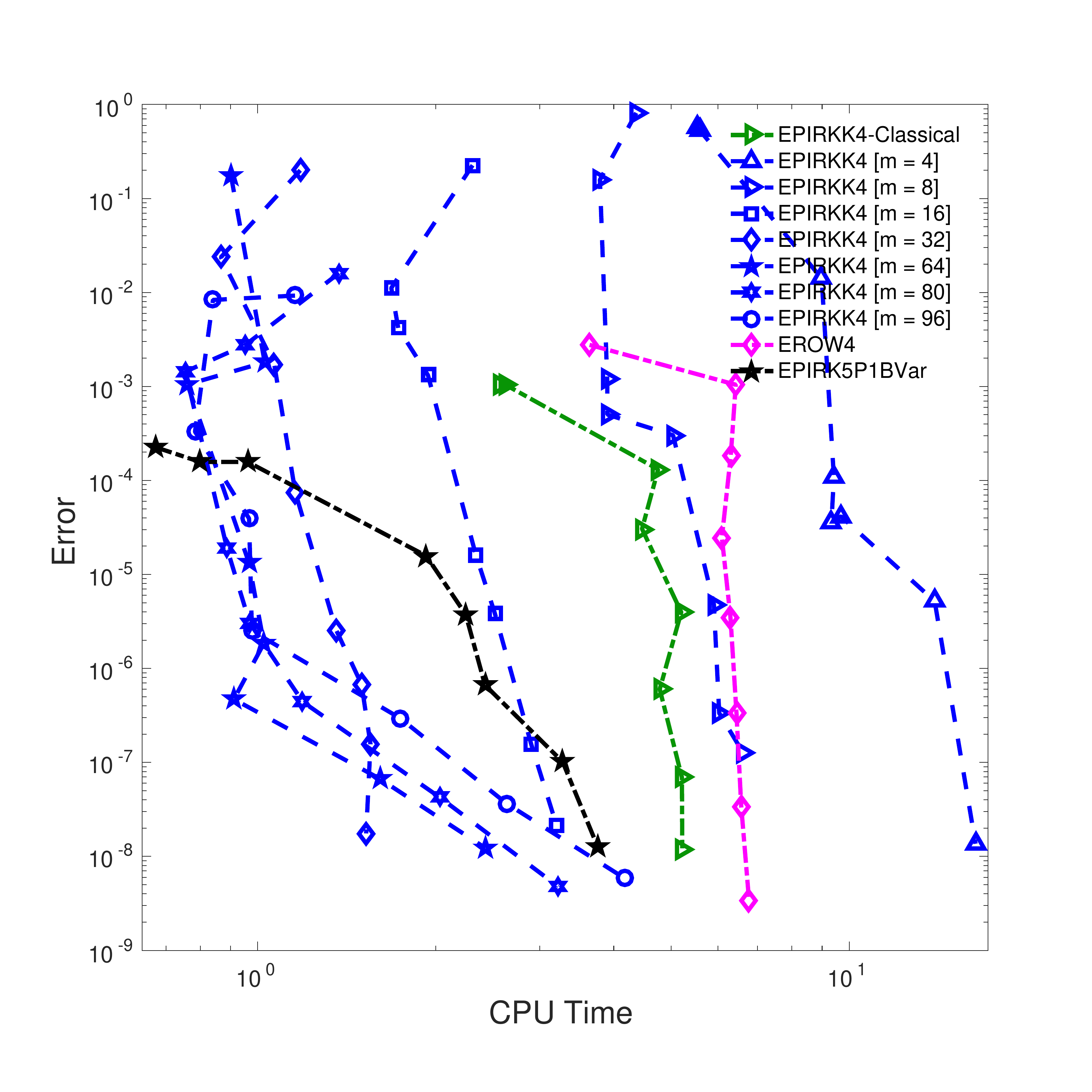}
	    \caption{K-type versus classical methods}
            \label{fig:ALCN-work-prec-diag-kc-256}
    \end{subfigure}
    ~~
    \begin{subfigure}{0.45\textwidth}
        \centering
	    \includegraphics[scale=0.25]{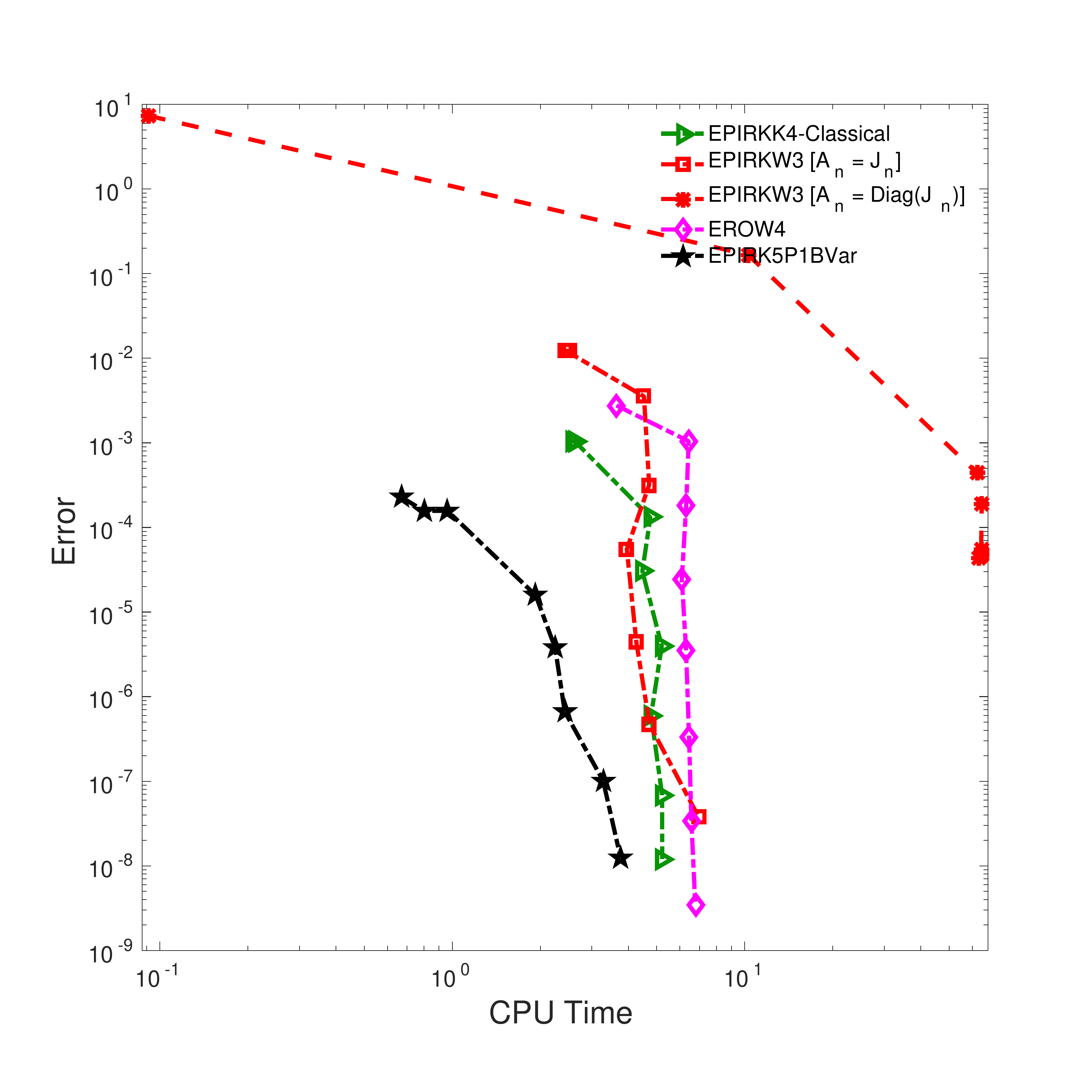}
	    \caption{W-type versus classical methods}
             \label{fig:ALCN-work-prec-diag-wc-256}
    \end{subfigure}%
    \caption{Work-precision diagrams for different methods applied to Allen-Cahn problem \eqref{eqn:AllenCahn} on a $256 \times 256$ grid.}
\end{figure}

\begin{figure}[!htb]
    \centering
    \begin{subfigure}{.45\textwidth}
        \centering
	    \includegraphics[scale=0.25]{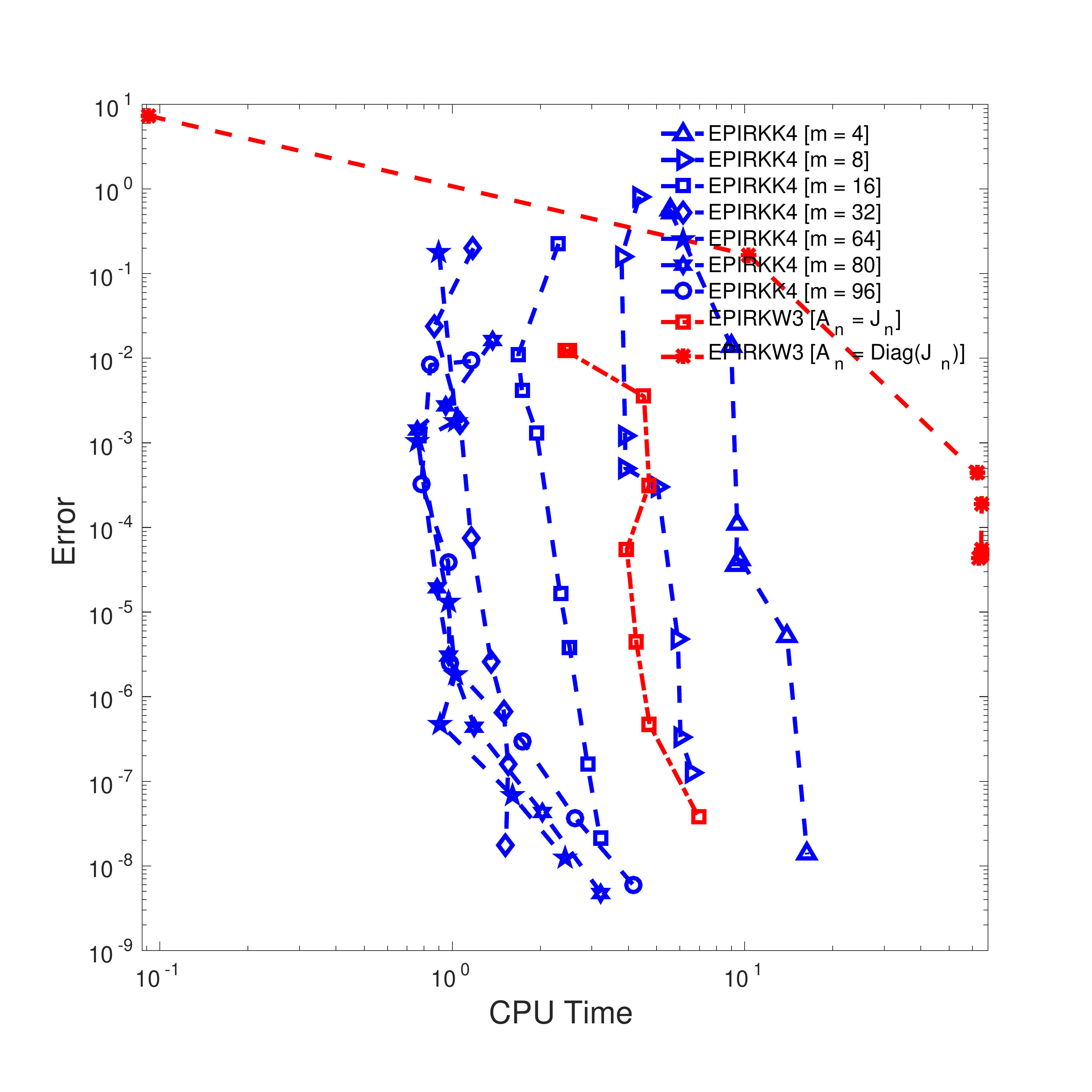}
  	    \caption{W-type versus K-type methods}
		\label{fig:ALCN-work-prec-diag-wk-256}
    \end{subfigure}%
    \begin{subfigure}{0.45\textwidth}
        \centering
		\includegraphics[scale=0.25]{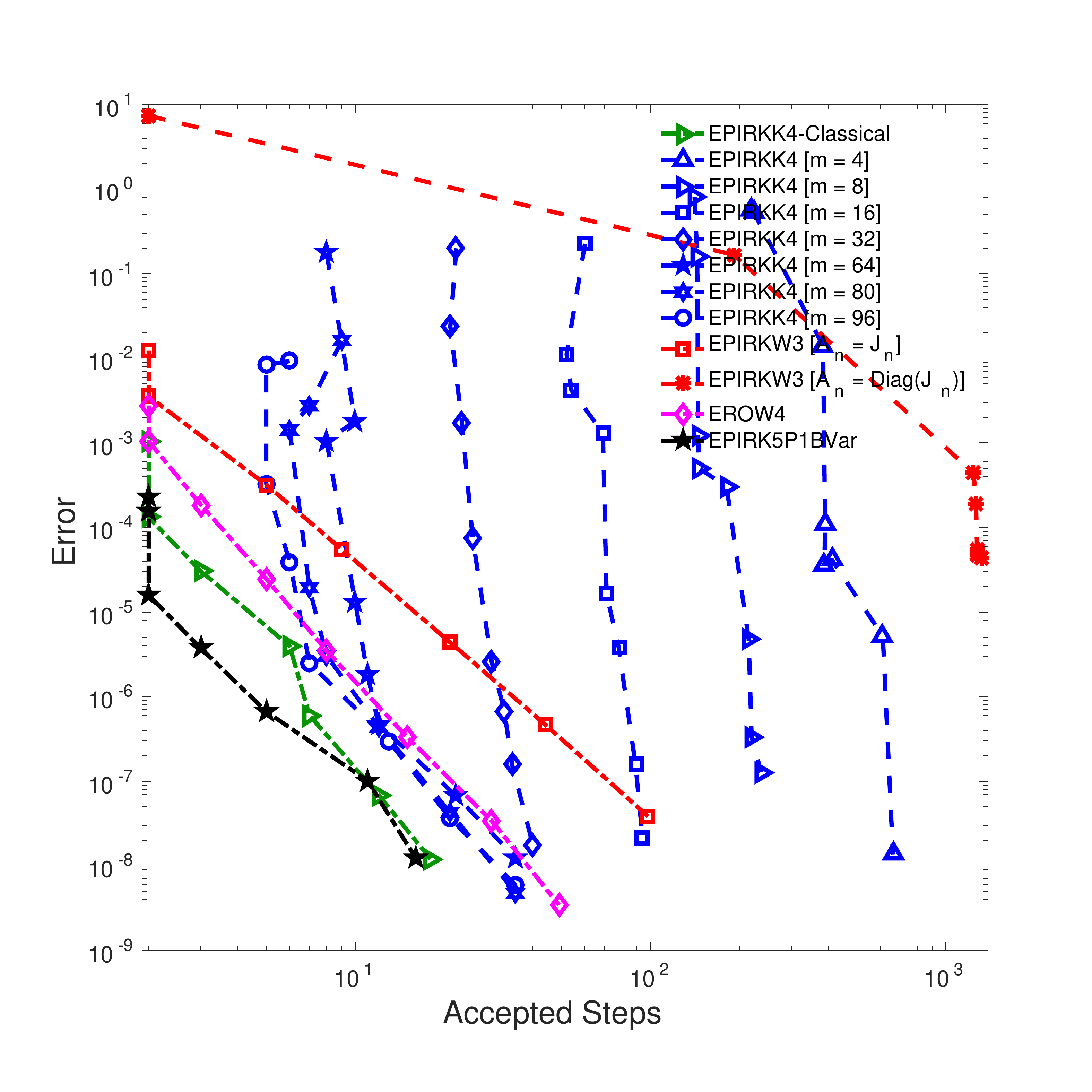}
	    \caption{Convergence diagram}
		\label{fig:ALCN-error-vs-steps-256}
    \end{subfigure}%
    \caption{Work-precision diagrams and convergence diagrams for different methods applied to Allen-Cahn problem \eqref{eqn:AllenCahn} on a $256 \times 256$ grid.}
\end{figure}

\begin{table}[tbh!]
\centering
\begin{tabular}{|>{\arraybackslash}m{4.2cm}|>{\centering\arraybackslash}m{0.75cm}|>{\centering\arraybackslash}m{0.75cm}|>{\centering\arraybackslash}m{0.75cm}|>{\centering\arraybackslash}m{0.75cm}|>{\centering\arraybackslash}m{0.75cm}|>{\centering\arraybackslash}m{0.75cm}|>{\centering\arraybackslash}m{0.75cm}|>{\centering\arraybackslash}m{0.75cm}|}
\hline
\backslashbox{\textbf{Integrator}}{\textbf{Tolerance}}  & $10^{-1}$ & $10^{-2}$ & $10^{-3}$ & $10^{-4}$ & $10^{-5}$ & $10^{-6}$ & $10^{-7}$ & $10^{-8}$ \\ \hline
{\sc epirkk4-classical}                          &  192  &  192  &  211  &  177  &  131  &  118  &   92  &   75   \\ \hline                                                                                                                
{\sc epirkk4} [M = 4]                            &    4  &    4  &    4  &    4  &    4  &    4  &    4  &    4   \\ \hline                                                                                                                
{\sc epirkk4} [M = 8]                            &    8  &    8  &    8  &    8  &    8  &    8  &    8  &    8   \\ \hline                                                                                                                
{\sc epirkk4} [M = 16]                           &   16  &   16  &   16  &   16  &   16  &   16  &   16  &   16   \\ \hline                                                                                                                
{\sc epirkk4} [M = 32]                           &   32  &   32  &   32  &   32  &   32  &   32  &   32  &   32   \\ \hline                                                                                                                
{\sc epirkk4} [M = 64]                           &   64  &   64  &   64  &   64  &   64  &   64  &   64  &   64   \\ \hline                                                                                                                
{\sc epirkk4} [M = 80]                           &   80  &   80  &   80  &   80  &   80  &   80  &   80  &   80   \\ \hline                                                                                                                
{\sc epirkk4} [M = 96]                           &   96  &   96  &   96  &   96  &   96  &   96  &   96  &   96   \\ \hline                                                                                                                
{\sc epirkw3} [$\A_n = \J_n$]                    &  186  &  186  &  204  &  125  &   93  &   61  &   41  &   28   \\ \hline                                                                                                                
{\sc epirkw3} [$\A_n = \textnormal{diag}(\J_n)$] &    0  &    0  &    0  &    0  &    0  &    0  &    0  &    0   \\ \hline                                                                                                                
{\sc erow4}                                      &  231  &  254  &  215  &  170  &  133  &  100  &   71  &   54   \\ \hline                                                                                                                
{\sc epirk5p1bvar}                               &  328  &  358  &  410  &  475  &  420  &  331  &  234  &  205   \\ \hline   
\end{tabular}
\caption{Root mean square number of Krylov vectors per projection for each integrator applied to Allen-Cahn problem \eqref{eqn:AllenCahn} on a $256 \times 256$ grid. {\sc epirkk4} uses fixed number of basis vectors in the Arnoldi process as indicated in brackets succeeding the name of the method. {\sc epirkw3} with $A_n = diag(J_n)$ directly computes the $\psi$ function product not needing to perform Arnoldi. All other methods, both {\sc epirkw3} with $A_n = J_n$ and the classical methods use an adaptive Arnoldi process to approximate the $\psi$ function product.\label{Table:alcn-kdim-256}}
\end{table}

\begin{table}[tbh!]
\centering
\begin{tabular}{|>{\arraybackslash}m{4.2cm}|>{\centering\arraybackslash}m{0.75cm}|>{\centering\arraybackslash}m{0.75cm}|>{\centering\arraybackslash}m{0.75cm}|>{\centering\arraybackslash}m{0.75cm}|>{\centering\arraybackslash}m{0.75cm}|>
{\centering\arraybackslash}m{0.75cm}|>{\centering\arraybackslash}m{0.75cm}|>{\centering\arraybackslash}m{0.75cm}|}
\hline
\backslashbox{\textbf{Integrator}}{\textbf{Tolerance}}  & $10^{-1}$ & $10^{-2}$ & $10^{-3}$ & $10^{-4}$ & $10^{-5}$ & $10^{-6}$ & $10^{-7}$ & $10^{-8}$ \\ \hline
{\sc epirkk4-classical}                          &    0  &    0  &    0  &    0  &    1  &    1  &    1  &    1   \\ \hline                                                                                                                
{\sc epirkk4} [M = 4]                            &   76  &   70  &   36  &   77  &   80  &   53  &    8  &  168   \\ \hline                                                                                                                
{\sc epirkk4} [M = 8]                            &  110  &   36  &   49  &   46  &   81  &   84  &   87  &  106   \\ \hline                                                                                                                
{\sc epirkk4} [M = 16]                           &   56  &   24  &   22  &    6  &   37  &   36  &   43  &   49   \\ \hline                                                                                                                
{\sc epirkk4} [M = 32]                           &   16  &    3  &   10  &   11  &   13  &   14  &   14  &    1   \\ \hline                                                                                                                
{\sc epirkk4} [M = 64]                           &    5  &    5  &    3  &    4  &    4  &    0  &    0  &    0   \\ \hline                                                                                                                
{\sc epirkk4} [M = 80]                           &    6  &    4  &    2  &    2  &    2  &    0  &    0  &    0   \\ \hline                                                                                                                
{\sc epirkk4} [M = 96]                           &    3  &    2  &    1  &    2  &    0  &    0  &    0  &    0   \\ \hline                                                                                                                
{\sc epirkw3} [$\A_n = \J_n$]                    &    0  &    0  &    0  &    2  &    1  &    1  &    0  &    0   \\ \hline                                                                                                                
{\sc epirkw3} [$\A_n = \textnormal{diag}(\J_n)$] &    0  &   46  &  156  &  181  &  169  &  169  &  164  &   41   \\ \hline                                                                                                                
{\sc erow4}                                      &    0  &    0  &    0  &    0  &    0  &    0  &    0  &    0   \\ \hline                                                                                                                
{\sc epirk5p1bvar}                               &    0  &    0  &    0  &    0  &    0  &    0  &    0  &    0   \\ \hline       
\end{tabular}
\caption{Total number of rejected timesteps for each integrator applied to Allen-Cahn problem \eqref{eqn:AllenCahn} on a $256 \times 256$ grid. {\sc epirkk4} executions with smaller basis sizes suffer severely from step rejections and take large number of timesteps for all different tolerance levels and so does {\sc epirkw3} with $\A_n = \textnormal{diag}(\J_n)$, signaling the poor stability of these methods. \label{Table:alcn-rejn-256}}
\end{table}

\begin{table}[tbh!]
\centering
\small
\begin{tabular}{|>{\arraybackslash}m{3.8cm}|>{\centering\arraybackslash}m{1cm}|>{\centering\arraybackslash}m{1cm}|>{\centering\arraybackslash}m{1cm}|>{\centering\arraybackslash}m{1cm}|>{\centering\arraybackslash}m{1cm}|>{\centering\arraybackslash}m{1cm}|>{\centering\arraybackslash}m{1cm}|>{\centering\arraybackslash}m{1cm}|}
\hline
\backslashbox{\textbf{Integrator}}{\textbf{Accuracy}}  & $10^{-1}$ & $10^{-2}$ & $10^{-3}$ & $10^{-4}$ & $10^{-5}$ & $10^{-6}$ & $10^{-7}$ & $10^{-8}$ \\ \hline
{\sc epirkk4-classical}                             &   ---    &      2.567 &      4.754 &      4.465 &      5.201 &      4.784 &      5.216 &      ---     \\ \hline
{\sc epirkk4} [M = 4]                               &    8.984 &     ---    &      9.422 &      9.333 &     13.949 &     ---    &     16.353 &      ---     \\ \hline
{\sc epirkk4} [M = 8]                               &   ---    &      3.899 &      3.903 &     ---    &      5.907 &      6.03  &     ---    &      ---     \\ \hline
{\sc epirkk4} [M = 16]                              &    1.683 &      1.729 &     ---    &      2.337 &      2.518 &      2.901 &      3.202 &      ---     \\ \hline
{\sc epirkk4} [M = 32]                              &     .867 &      1.067 &     ---    &      1.158 &      1.357 &      1.499 &      1.526 &      ---     \\ \hline
{\sc epirkk4} [M = 64]                              &   ---    &       .761 &     ---    &       .967 &      1.025 &       .912 &      1.609 &      ---     \\ \hline
{\sc epirkk4} [M = 80]                              &    1.373 &       .755 &     ---    &       .888 &       .97  &      1.192 &      2.033 &      3.222   \\ \hline
{\sc epirkk4} [M = 96]                              &   ---    &       .841 &       .784 &       .967 &       .978 &      1.742 &      2.635 &      4.172   \\ \hline
{\sc epirkw3} [$\A_n = \J_n$]                       &    2.444 &      4.522 &      4.712 &      3.952 &      4.26  &      4.727 &      6.932 &      ---     \\ \hline
{\sc epirkw3} [$\A_n = \textnormal{diag}(\J_n)$]    &   ---    &     ---    &     62.205 &     62.858 &     ---    &     ---    &     ---    &      ---     \\ \hline
{\sc erow4}                                         &   ---    &      3.641 &      6.308 &      6.093 &      6.3   &      6.45  &      6.57  &      6.772   \\ \hline
{\sc epirk5p1bvar}                                  &   ---    &     ---    &       .672 &      1.928 &      2.248 &      2.423 &      3.76  &      ---     \\ \hline
\end{tabular}
\caption{CPU time in which accuracy is achieved by different integrators applied to Allen-Cahn problem \eqref{eqn:AllenCahn} on a $256 \times 256$ grid.}
\label{Table:alcn-accuracy-time-256}
\end{table}

Figure \ref{fig:ALCN-work-prec-diag-kc} demonstrates that {\sc epirkk4} with basis size set to either sixteen ($m=16$) or thirty-two ($m=32$) is more efficient than classical exponential methods  for the Allen-Cahn problem \eqref{eqn:AllenCahn} on a spatial grid of size $64 \times 64$. {\sc epirkk4} with basis size four requires a large number of timesteps relative to both classical and {\sc epirkk4} executions with larger basis sizes, as shown in Figure \ref{fig:ALCN-error-vs-steps}. It also suffers from step rejections as indicated in Table \ref{Table:alcn-rejn}. As a consequence, despite being very cheap per timestep, {\sc epirkk4} performs the poorest among all {\sc epirkk4} executions. 

As the tolerance is tightened, all integrators, barring a few, are forced to take small steps that are approximately of the same size as seen in Figure \ref{fig:ALCN-error-vs-steps}. As a consequence, integrators that are cheap per timestep are more efficient for high accuracy solutions. {\sc epirkk4} with eight basis vectors ($m=8$) comes out on top in this regard.

{\sc epirkw3} with $\A_n=\J_n$ behaves like a classical method, but takes more steps in the high accuracy region due to its lower order, as shown in Figure \ref{fig:ALCN-error-vs-steps}. The same figure illustrates that {\sc epirkw3} with $\A_n = \textnormal{diag}(\J_n)$ takes significantly more steps than all other integrators and also suffers from step rejections in the low accuracy region (Table \ref{Table:alcn-rejn}). Both step rejection and small step size contribute to a poor performing integrator in its case. 

As the grid resolution is increased to $256 \times 256$, the picture is similar with some notable exceptions (see Figures \ref{fig:ALCN-work-prec-diag-kc-256}, \ref{fig:ALCN-work-prec-diag-wk-256} and \ref{fig:ALCN-error-vs-steps-256}). {\sc epirkk4} with basis size set to sixteen is no longer performing as well as before. Larger basis sizes are required to gain enough stability while remaining cheap per timestep. Although choosing the basis size adaptively for $K$-methods seems to be a natural next step, this has not been investigated as part of current work and will be considered as future extension.

Tables \ref{Table:alcn-accuracy-time} and \ref{Table:alcn-accuracy-time-256} summarize the results of the work precision diagrams in a tabular fashion where they show the best CPU time in which the integrators attain certain prescribed levels of accuracy for the Allen-Cahn problem \eqref{eqn:AllenCahn} on a $64 \times 64$ and $256 \times 256$ grid, respectively.

We also performed additional experiments with variable time stepping  on the Allen-Cahn problem \eqref{eqn:AllenCahn} ($64 \times 64$ grid) for different time spans ranging between [0, 0.2] units to [0, 9.6] units. Those results are not reported here in detail. We noticed that the relative performance of the classical methods and $K$-methods depends on the time span: in some instances the classical methods perform as well as the $K$-methods, while in other instances they perform more poorly than for the cases shown in the paper. We could not conclusively point to the cause of this variability in the relative performance.

\section{Conclusions}
\label{sec:Conclusions}

Exponential Propagation Iterative Methods of Runge-Kutta type (EPIRK) rely on the computation of exponential-like matrix functions of the Jacobian times vector products. This paper develops two new classes of EPIRK methods that allow the use of inexact Jacobians as arguments of the matrix functions. We derive a general order conditions theory for EPIRK-$W$-methods, which admit arbitrary approximations of the Jacobian, and for  EPIRK-$K$-methods, which use a specific Krylov-subspace based approximation of the Jacobian. We also provide a computational procedure to derive order conditions for methods with an arbitrary number of stages, and solve the order conditions of a three stage formulation to obtain coefficients for two third order EPIRK-$W$-methods, named {\sc epirkw3}, and a fourth order EPIRK-$K$-method, named {\sc epirkk4}. Several alternative implementations of  $W$-methods and $K$-methods are discussed, and their properties are studied.

Numerical experiments are conducted with three different test problems to study the performance of the new methods. The results confirm empirically that {\sc epirkw3}-method retains third order of convergence with different Jacobian approximations. The {\sc epirkk4}-method is computationally more efficient than {\sc epirkw3} and performs better than classical methods in a number of different scenarios considered in the paper. In particular {\sc epirkk4} outperforms {\sc epirkk4-classical}, a method with the same set of coefficients, but implemented in the classical framework for exponential integrators. More numerical experiments are needed to asses how these results generalize to different applications.

It was also observed that increasing the basis size can make {\sc epirkk4} more stable, but there is a cost to pay for it and that it is important to balance gains in stability against the increased computational cost. An adaptive approach to basis size selection for the $K$-methods is relevant in this regard and will be considered in future work.

\section*{Acknowledgements}

This work has been supported in part by NSF through awards NSF DMS-1419003, NSF DMS-1419105, NSF CCF-1613905, AFOSR FA9550-12-1-0293-DEF, and by the Computational Science Laboratory at Virginia Tech.

\appendix
\section{Derivation of $K$-methods}

The {\sc epirkk} method is derived from {\sc epirkw} method, where a specific Krylov-subspace approximation of the Jacobian is used instead of an arbitrary approximate Jacobian as admitted by the $W$-method. We start the derivation by first stating the general form of the {\sc epirkw} method:
\begin{equation}
\begin{split}
Y_i &= y_{n} + a_{i,1}\,  \psi_{i,1}(g_{i,1}\,\,h\,\A_n)\, hf(y_{n}) + \displaystyle \sum_{j = 2}^{i} a_{i,j}\, \psi_{i,j}(g_{i,j}\,h\,\A_{n})\, h\Delta^{(j-1)}r(y_{n}), \quad i = 1, \hdots, s - 1, \\
y_{n+1} &= y_{n}\, + b_{1}\, \psi_{s,1}(g_{s,1}\,h\,\A_{n})\, hf(y_{n}) + \displaystyle\sum_{j = 2}^{s} b_{j}\, \psi_{s,j}(g_{s,j}\,h\,\A_{n})\, h\Delta^{(j-1)}r(y_{n}).\nonumber
\end{split}
\end{equation}
In the above equation, the $\psi$ function is as defined in equation \eqref{eqn:psi_function_definition} and the following simplifying assumption is made about it: 
\begin{eqnarray}
\psi_{i,j}(z) &=& \psi_{j}(z) = \displaystyle\sum_{k=1}^{j} p_{j,k}\, \varphi_k(z), \nonumber
\end{eqnarray}
where $\varphi$ function is as defined in equation \eqref{eqn:phi_k}, \eqref{eqn:phi_k_recursive_formulation_and_phi_k_0}. Additionally, the remainder function ($r(y)$) and the forward difference operator ($\Delta^{(j)}r(Y_i)$) are defined accordingly below:
\begin{eqnarray}
r(y) &=& f(y) - f(y_{n}) - \mathbf{A}_{n} \, (y - y_{n}),\\ \nonumber
\Delta^{(j)}r(Y_i) &=& \Delta^{(j - 1)}r(Y_{i+1}) - \Delta^{(j - 1)}r(Y_{i}),\\ \nonumber
\Delta^{(1)}r(Y_i) &=& r(Y_{i+1}) - r(Y_i). \nonumber
\end{eqnarray}
The $K$-method uses a specific Krylov-subspace based approximation of the Jacobian ($\mathbf{J}_{n}$). An \textit{M}-dimensional Krylov-subspace is built as,
\begin{eqnarray}
\mathcal{K}_M = \text{span}\{f_n, \mathbf{J}_{n}f_n, \mathbf{J}_{n}^2f_n, \hdots, \mathbf{J}_{n}^{M-1}f_n\},
\end{eqnarray}
whose basis is the orthonormal matrix $\mathbf{V}$ and $\mathbf{H}$ is the upper-Hessenberg matrix obtained from Arnoldi iteration defined as
\begin{eqnarray}
\mathbf{H} &=& \mathbf{V}^{T} \mathbf{J}_{n} \,\mathbf{V}.
\end{eqnarray}
The corresponding Krylov-subspace based approximation of the Jacobian is built as 
\begin{eqnarray}
\mathbf{A}_{n} &=& \mathbf{V}\,\mathbf{H}\,\mathbf{V}^{T} = \mathbf{V}\mathbf{V}^{T} \mathbf{J}_{n} \,\mathbf{V}\mathbf{V}^{T}.
\end{eqnarray}
The use of Krylov-subspace based approximation of the Jacobian reduces the $\varphi$ and $\psi$ function in accordance with lemma \ref{Lemma:phik_reduced} and \ref{Lemma:psik_reduced} to the following 
\begin{eqnarray}
\varphi_k(h\,\gamma\,\mathbf{A}_{n}) &=& \frac{1}{k!} \left(\mathbf{I} - \mathbf{V}\mathbf{V}^{T}\right) + \mathbf{V}  \varphi_k(h\,\gamma\,\mathbf{H}) \mathbf{V}^{T},\\
\psi_{j}(h\gamma \mathbf{A}_{n}) &=& 
\widetilde{p}_j \left(\mathbf{I} - \mathbf{V}\mathbf{V}^{T}\right) +  \mathbf{V}\,\psi_{j}(h\gamma \mathbf{H}) \mathbf{V}^{T},
\end{eqnarray}
where $\widetilde{p}_j$ is defined as
\begin{eqnarray}
\widetilde{p}_j = \displaystyle\sum_{k=1}^{j}\frac{p_{j,k}}{k!}.
\end{eqnarray}
In order to derive the reduced stage formulation of the {\sc epirkk} method we need to resolve the vectors in the formulation into components in the Krylov-subspace and orthogonal to it. Repeating the splittings from the main text:

\begin{itemize}
\item Splitting the internal stage vectors noting that $Y_0 \equiv y_n$:
\begin{equation}
Y_i = \mathbf{V}\lambda_i + Y_i^{\bot} \qquad \textnormal{where} \quad
\mathbf{V}^{T} Y_i = \lambda_i, \quad
\left(\mathbf{I} - \mathbf{V}\mathbf{V}^{T}\right)\,Y_i = Y_i^{\bot}.
\end{equation}
\item  Splitting the right-hand side function evaluated at internal stage vectors while noting that $f_0 \equiv f(y_n)$:
\begin{equation}
f_i :=  f(Y_i) = \mathbf{V}\eta_i + f_i^{\bot} \qquad \textnormal{where} \quad
\mathbf{V}^{T} f_i = \eta_i, \quad
\left(\mathbf{I} - \mathbf{V}\mathbf{V}^{T}\right) f_i = f_i^{\bot}.
\end{equation}
\item Splitting the non-linear Taylor remainder of the right-hand side functions:
\begin{equation} 
\begin{split}
r(Y_i) &= f(Y_i) - f(y_{n}) - \mathbf{A}_{n}\, (Y_i - y_{n})  = f_i - f_0 - \mathbf{V}\, \mathbf{H}\, \mathbf{V}^{T}\, (Y_i - y_{n}), \\
\textnormal{where} \quad \mathbf{V}^{T}\, r(Y_i)  &= \eta_i - \eta_0 - \mathbf{H} \, (\lambda_i - \lambda_0), \\
\left(\mathbf{I} - \mathbf{V}\mathbf{V}^{T}\right)\, r(Y_i)  &= f_i^{\bot} - f_0^{\bot}.
\end{split}
\end{equation}
\item Splitting the forward differences of the non-linear remainder terms:
\begin{equation}
\begin{split}
\widetilde{r}_{(j-1)} &:=\Delta^{(j-1)}r(y_{n}) = \mathbf{V}\, d_{(j-1)} + \widetilde{r}_{(j-1)}^{\bot}, \\
\textnormal{where} &\quad \mathbf{V}^{T}\,\widetilde{r}_{(j-1)}   = d_{(j-1)}, \quad
\left(\mathbf{I} - \mathbf{V}\mathbf{V}^{T}\right)\, \widetilde{r}_{(j-1)}  = \widetilde{r}_{(j-1)}^{\bot}.
\end{split}
\end{equation}
\end{itemize}
Using these each internal stage of the {\sc epirkk} method can be expressed as below:
\begin{eqnarray}
\mathbf{V}\lambda_i + Y_i^{\bot} &=& y_{n}\, +  h\, a_{i,1} \bigg(\widetilde{p_1} \left(\mathbf{I} - \mathbf{V}\mathbf{V}^{T}\right) +  \mathbf{V}\,\psi_{1}(h\,g_{i,1}\, \mathbf{H}) \mathbf{V}^{T}\bigg) \bigg(\mathbf{V}\eta_0 + f_0^{\bot}\bigg)  \nonumber \\
&& + \displaystyle\sum_{j = 2}^{i} h\, a_{i,j}\, \bigg(\widetilde{p}_j \left(\mathbf{I} - \mathbf{V}\mathbf{V}^{T}\right) +  \mathbf{V}\,\psi_{j}(h\,g_{i,j}\, \mathbf{H}) \mathbf{V}^{T}\bigg) \bigg(\mathbf{V} d_{(j-1)} + \widetilde{r}_{(j-1)}^{\bot}\bigg) \nonumber \\
&=& y_{n}\, +  h\, a_{i,1} \bigg(\widetilde{p_1} \, f_0^{\bot} +  \mathbf{V}\,\psi_{1}(h\,g_{i,1}\, \mathbf{H}) \eta_0\bigg)  \nonumber \\
&& + \displaystyle\sum_{j = 2}^{i} h\, a_{i,j}\, \bigg(\widetilde{p}_j \, \widetilde{r}_{(j-1)}^{\bot} +  \mathbf{V}\,\psi_{j}(h\,g_{i,j}\, \mathbf{H}) \, d_{(j-1)}\bigg).
\label{eqn:stage_vector_expressed_in_K_space}
\end{eqnarray}
The reduced stage formulation of the {\sc epirkk} method is obtained by multiplying the above equation by $\mathbf{V}^{T}$ from the left
\begin{eqnarray}
\lambda_i &=& \mathbf{V}^{T} y_{n}\, + h \, a_{i,1} \psi_{1}(h\,g_{i,1}\,\mathbf{H}) \eta_0 + \displaystyle\sum_{j = 2}^{i} h\, a_{i,j}\, \psi_{j}(h\,g_{i,j}\, \mathbf{H}) \, d_{(j-1)} \nonumber \\
&=& \lambda_0 + h \, a_{i,1} \psi_{1}(h\,g_{i,1}\,\mathbf{H}) \eta_0 + \displaystyle\sum_{j = 2}^{i} h\, a_{i,j}\, \psi_{j}(h\,g_{i,j}\, \mathbf{H}) \, d_{(j-1)}.
\end{eqnarray}
And the full stage vector can be recovered by first computing the reduced stage vector and adding the orthogonal piece $Y_i^{\bot}$ obtained when multiplying equation \eqref{eqn:stage_vector_expressed_in_K_space} by $\left(\mathbf{I} - \mathbf{V}\mathbf{V}^{T}\right)$
\begin{eqnarray}
Y_i^{\bot} &=& \left(\mathbf{I} - \mathbf{V}\mathbf{V}^{T}\right) y_{n}\, + h \, a_{i,1} \, \widetilde{p_1} \, f_0^{\bot} + \displaystyle\sum_{j = 2}^{i} h\, a_{i,j}\, \widetilde{p}_j \, \widetilde{r}_{(j-1)}^{\bot} \nonumber \\
&=&(y_{n}\, - \mathbf{V} \lambda_0) + h \, a_{i,1} \, \widetilde{p_1} \, f_0^{\bot} + \displaystyle\sum_{j = 2}^{i} h\, a_{i,j}\, \widetilde{p}_j \, \widetilde{r}_{(j-1)}^{\bot}.
\end{eqnarray}
The final stage can also be written in the above form with multipliers $b_i$ in place of $a_{ij}$. Notice that the expensive computations are performed in the reduced space, i.e. the $\psi$ function is computed in the reduced space instead of the full space offering potential computational savings. In the above equations, 
the quantities $d_{(j-1)}$ and $\widetilde{r}_{(j-1)}^{\bot}$ can be shown to be 
\begin{eqnarray}
d_{(j-1)} &=& \displaystyle\sum_{k=0}^{j-1} \bigg((-1)^k {j-1 \choose k} \eta_{j-1-k} - \mathbf{H} \bigg((-1)^k {j-1 \choose k} \lambda_{j-1 -k}\bigg)\bigg),\\
\widetilde{r}_{(j-1)}^{\bot} &=& \displaystyle\sum_{k=0}^{j-1} \bigg((-1)^k {j-1 \choose k} f_{j-1-k}^{\bot}\bigg),
\end{eqnarray}
as is done in the following appendix.


\section{Proofs}\label{sec:AppendixB}

In order to prove equations \eqref{eqn:reduced_stage_vector_d} and \eqref{eqn:reduced_stage_vector_r}, we start with the definition of the remainder function and forward difference.

\begin{eqnarray}
r(y) &=& f(y) - f(y_{n}) - \mathbf{A}_{n} (y - y_{n}),
\end{eqnarray}

\begin{subequations}
\begin{eqnarray}
\Delta^{(j)}r(Y_i) &=& \Delta^{(j - 1)}r(Y_{i+1}) - \Delta^{(j - 1)}r(Y_{i}),\\
\Delta^{(1)}r(Y_i) &=& r(Y_{i+1}) - r(Y_i).
\end{eqnarray}
\end{subequations}

\begin{lemma}\label{Lemma:jth-forward-difference}
$\Delta^{(j)}r(Y_i) = \displaystyle\sum_{k=0}^{j} (-1)^k  {j \choose k} r(Y_{i+j-k})$.
\end{lemma}
\begin{proof} In order to prove the lemma, we resort to mathematical induction. Base case $j = 1$,
\begin{eqnarray}
\Delta^{(1)}r(Y_i) &=&  \displaystyle\sum_{k=0}^{1} (-1)^k  {1 \choose k} r(Y_{i+1-k}) \nonumber \\
&=& r(Y_{i+1}) - r(Y_i). \nonumber
\end{eqnarray}
The base case is true by definition. We now assume that the proposition holds true for all $j$ up to $k-1$. We have,
\begin{eqnarray}
\Delta^{(k-1)}r(Y_i) &=&  \displaystyle\sum_{l=0}^{k-1} (-1)^l  {k-1 \choose l} r(Y_{i+k-1-l}), \\
\Delta^{(k-1)}r(Y_{i+1}) &=&  \displaystyle\sum_{l=0}^{k-1} (-1)^l  {k-1 \choose l} r(Y_{i+k-l}).
\end{eqnarray}
Then for $j = k$,
\begin{eqnarray}
\Delta^{(k)}r(Y_i) &=& \Delta^{(k - 1)}r(Y_{i+1}) - \Delta^{(k - 1)}r(Y_{i}) \nonumber \\
&=& \displaystyle\sum_{l=0}^{k-1}(-1)^l {k-1 \choose l} r(Y_{i+k-l}) - \displaystyle\sum_{l=0}^{k-1}(-1)^l {k-1 \choose l} r(Y_{i+k -1 -l})\nonumber \\
&=& r(Y_{i+k}) + \displaystyle\sum_{l=1}^{k-1}(-1)^l {k-1 \choose l} r(Y_{i+k-l}) - \displaystyle\sum_{l=0}^{k-2}(-1)^l {k-1 \choose l} r(Y_{i+k -1 -l}) + (-1)^k r(Y_i). \nonumber
\end{eqnarray}

We perform a change of variable for the second summation, $m = l + 1 \; \implies l = (m - 1)$,

\begin{eqnarray}
\Delta^{(k)}r(Y_i) &=& r(Y_{i+k}) + \displaystyle\sum_{l=1}^{k-1}(-1)^l {k-1 \choose l} r(Y_{i+k-l}) - \displaystyle\sum_{m=1}^{k-1}(-1)^{m-1} {k-1 \choose m - 1} r(Y_{i+k -m}) + (-1)^k r(Y_i) \nonumber \\
&=& r(Y_{i+k}) + \displaystyle\sum_{l=1}^{k-1}(-1)^l {k-1 \choose l} r(Y_{i+k-l}) + \displaystyle\sum_{m=1}^{k-1}(-1)^{m} {k-1 \choose m - 1} r(Y_{i+k -m}) + (-1)^k r(Y_i). \nonumber
\end{eqnarray}

The summations run between the same start and end indices, and can be collapsed.

\begin{eqnarray}
\Delta^{(k)}r(Y_i)&=& r(Y_{i+k}) + \displaystyle\sum_{l=1}^{k-1}(-1)^l \bigg({k-1 \choose l}  + {k-1 \choose l - 1}\bigg) r(Y_{i+k-l}) + (-1)^k r(Y_i). \nonumber
\end{eqnarray}

We use the identity,

\begin{eqnarray}
{k-1 \choose l}  + {k-1 \choose l - 1} = {k \choose l}, \nonumber
\end{eqnarray}

and arrive at the desired result,

\begin{eqnarray}
\Delta^{(k)}r(Y_i)&=& r(Y_{i+k}) + \displaystyle\sum_{l=1}^{k-1}(-1)^l {k \choose l} r(Y_{i+k-l}) + (-1)^k r(Y_i) \nonumber\\
&=& (-1)^0 {k \choose 0} r(Y_{i+k}) + \displaystyle\sum_{l=1}^{k-1}(-1)^l {k \choose l} r(Y_{i+k-l}) + (-1)^k {k \choose k} r(Y_i) \nonumber\\
&=& \displaystyle\sum_{l=0}^{k} (-1)^l  {k \choose l} r(Y_{i+k-l}).
\end{eqnarray}

\end{proof}

\begin{lemma}\label{Lemma:jth-forward-difference-components}
Given \begin{eqnarray}
\Delta^{(j-1)}r(y_{n}) = \mathbf{V} d_{(j-1)} + \widetilde{r}_{(j-1)}^{\bot},
\end{eqnarray}
we need to prove
\begin{eqnarray}
d_{(j-1)} = \displaystyle\sum_{k=0}^{j-1} \bigg((-1)^k {j-1 \choose k} \eta_{j-1-k} - \mathbf{H} \bigg((-1)^k {j-1 \choose k} \lambda_{j-1 -k}\bigg)\bigg),\\
\widetilde{r}_{(j-1)}^{\bot} = \displaystyle\sum_{k=0}^{j-1} \bigg((-1)^k {j-1 \choose k} f_{j-1-k}^{\bot}\bigg).
\end{eqnarray}
\end{lemma}
\begin{proof}
We start with lemma \ref{Lemma:jth-forward-difference} where we have proven that 
\begin{equation*}
\Delta^{(j)}r(Y_i) = \displaystyle\sum_{k=0}^{j} (-1)^k  {j \choose k} r(Y_{i+j-k}).
\end{equation*}
We plugin the value $i = 0$ which corresponds to $\Delta^{(j)}r(Y_0) \equiv \Delta^{(j)}r(y_n)$ and we get,
\begin{equation*}
\Delta^{(j)}r(y_n) = \displaystyle\sum_{k=0}^{j} (-1)^k  {j \choose k} r(Y_{j-k}).
\end{equation*} WLOG replacing $j$ by $j-1$  yields 
\begin{equation}\label{eqn:forward-difference-remainder-function-sum}
\Delta^{(j-1)}r(y_n) = \displaystyle\sum_{k=0}^{j-1} (-1)^k  {j-1 \choose k} r(Y_{j-1-k}).
\end{equation}
Since the left-hand side of equation \eqref{eqn:forward-difference-remainder-function-sum} can be written as
\begin{eqnarray*}
\Delta^{(j-1)}r(y_{n}) = \mathbf{V} d_{(j-1)} + \widetilde{r}_{(j-1)}^{\bot},
\end{eqnarray*}
we have the following result
\begin{eqnarray}
\mathbf{V} d_{(j-1)} + \widetilde{r}_{(j-1)}^{\bot} = \displaystyle\sum_{k=0}^{j-1} (-1)^k  {j-1 \choose k} r(Y_{j-1-k}).
\label{eqn:alternating_sum_of_remainder_fn}
\end{eqnarray}
The remainder function $r(Y_i)$ is defined as
\begin{eqnarray*}
r(Y_i) = f(Y_i) - f(y_n) - \mathbf{A}_{n} * (Y_i - y_n).
\end{eqnarray*}
Plugging in the definition of remainder function in \eqref{eqn:alternating_sum_of_remainder_fn} and observing that $r(y_n) = 0$, we get
\begin{eqnarray}
\mathbf{V} d_{(j-1)} + \widetilde{r}_{(j-1)}^{\bot} &=& \displaystyle\sum_{k=0}^{j-2} (-1)^k  {j-1 \choose k} \bigg(f(Y_{j-1-k}) - f(y_n) - \mathbf{A}_{n} * \bigg(Y_{j-1-k} - y_n\bigg)\bigg)\nonumber\\
&=&  \displaystyle\sum_{k=0}^{j-2} (-1)^k  {j-1 \choose k} \bigg(f(Y_{j-1-k}) - \mathbf{A}_{n} * Y_{j-1-k}\bigg) - \nonumber\\ &&\displaystyle\sum_{k=0}^{j-2} (-1)^k  {j-1 \choose k} \bigg(f(y_{n}) - \mathbf{A}_{n} * y_{n}\bigg).
\label{eqn:subsitute_remainder_function_in_lemma_b1}
\end{eqnarray}
Consider the identity involving alternating sum and difference of binomial coefficients,
\begin{eqnarray}
\displaystyle \sum_{i = 0}^{k} (-1)^{i} {k \choose i} = 0.
\label{eqn:identity_of_binomial_coefficients}
\end{eqnarray}
Applying \eqref{eqn:identity_of_binomial_coefficients} to \eqref{eqn:subsitute_remainder_function_in_lemma_b1} we get,
\begin{eqnarray}
\mathbf{V} d_{(j-1)} + \widetilde{r}_{(j-1)}^{\bot} &=& \displaystyle\sum_{k=0}^{j-2} (-1)^k  {j-1 \choose k} \bigg(f(Y_{j-1-k}) - f(y_n) - \mathbf{A}_{n} * \bigg(Y_{j-1-k} - y_n\bigg)\bigg)\nonumber\\
&=&  \displaystyle\sum_{k=0}^{j-2} (-1)^k  {j-1 \choose k} \bigg(f(Y_{j-1-k}) - \mathbf{A}_{n} * Y_{j-1-k}\bigg) - \nonumber\\ && (-(-1)^{(j-1)}) \bigg(f(y_{n}) - \mathbf{A}_{n} * y_{n}\bigg) \nonumber\\
&=&  \displaystyle\sum_{k=0}^{j-2} (-1)^k  {j-1 \choose k} \bigg(f(Y_{j-1-k}) - \mathbf{A}_{n} * Y_{j-1-k}\bigg) + \nonumber\\ && (-1)^{(j-1)} {j-1 \choose j-1} \bigg(f(y_{n}) - \mathbf{A}_{n} * y_{n}\bigg) \nonumber\\
&=&  \displaystyle\sum_{k=0}^{j-1} (-1)^k  {j-1 \choose k} \bigg(f(Y_{j-1-k}) - \mathbf{A}_{n} * Y_{j-1-k}\bigg).
\end{eqnarray}
Additionally, since we are replacing the Jacobian by the approximation in the Krylov-subspace, i.e. $\mathbf{A}_{n} = \mathbf{V}\,\mathbf{H}\,\mathbf{V}^T$ we have
\begin{eqnarray}
\mathbf{V} d_{(j-1)} + \widetilde{r}_{(j-1)}^{\bot} &=& \displaystyle\sum_{k=0}^{j-1} (-1)^k  {j-1 \choose k} \bigg(f(Y_{j-1-k}) - \mathbf{V}\mathbf{H}\mathbf{V}^{T} * Y_{j-1-k}\bigg).
\label{eqn:penultimate_equation_of_lemma_d_r}
\end{eqnarray}
We get the expression for $d_{(j-1)}$ by multiplying equation \eqref{eqn:penultimate_equation_of_lemma_d_r} from the left by $V^{T}$ and that for $\widetilde{r}_{(j-1)}^{\bot}$ by multiplying by $(\mathbf{I} - \mathbf{V}\mathbf{V}^{T})$.

\end{proof}

\section{Structure of Jacobian evaluated at $y_0$ for the test problems}\label{sec:AppendixC}

\begin{figure}[!h]
  \centering
  \includegraphics[scale=0.35]{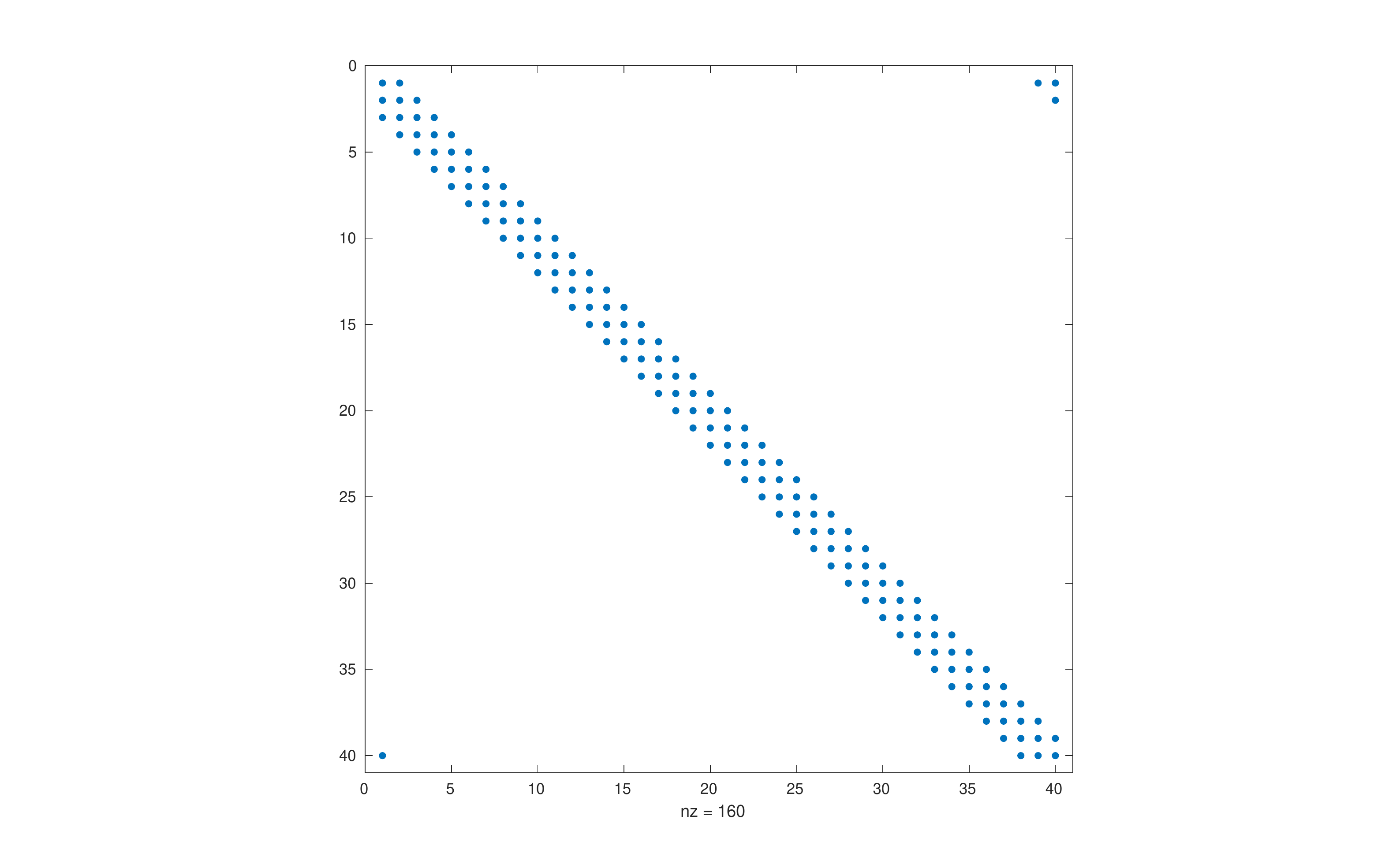}
  \caption{Structure of Jacobian evaluated at $y_0$ for the Lorenz-96 system \eqref{eqn:Lorenz}\label{fig:L96-Jac-y0}}
\end{figure}

\begin{figure}[!h]
  \centering
  \includegraphics[scale=0.35]{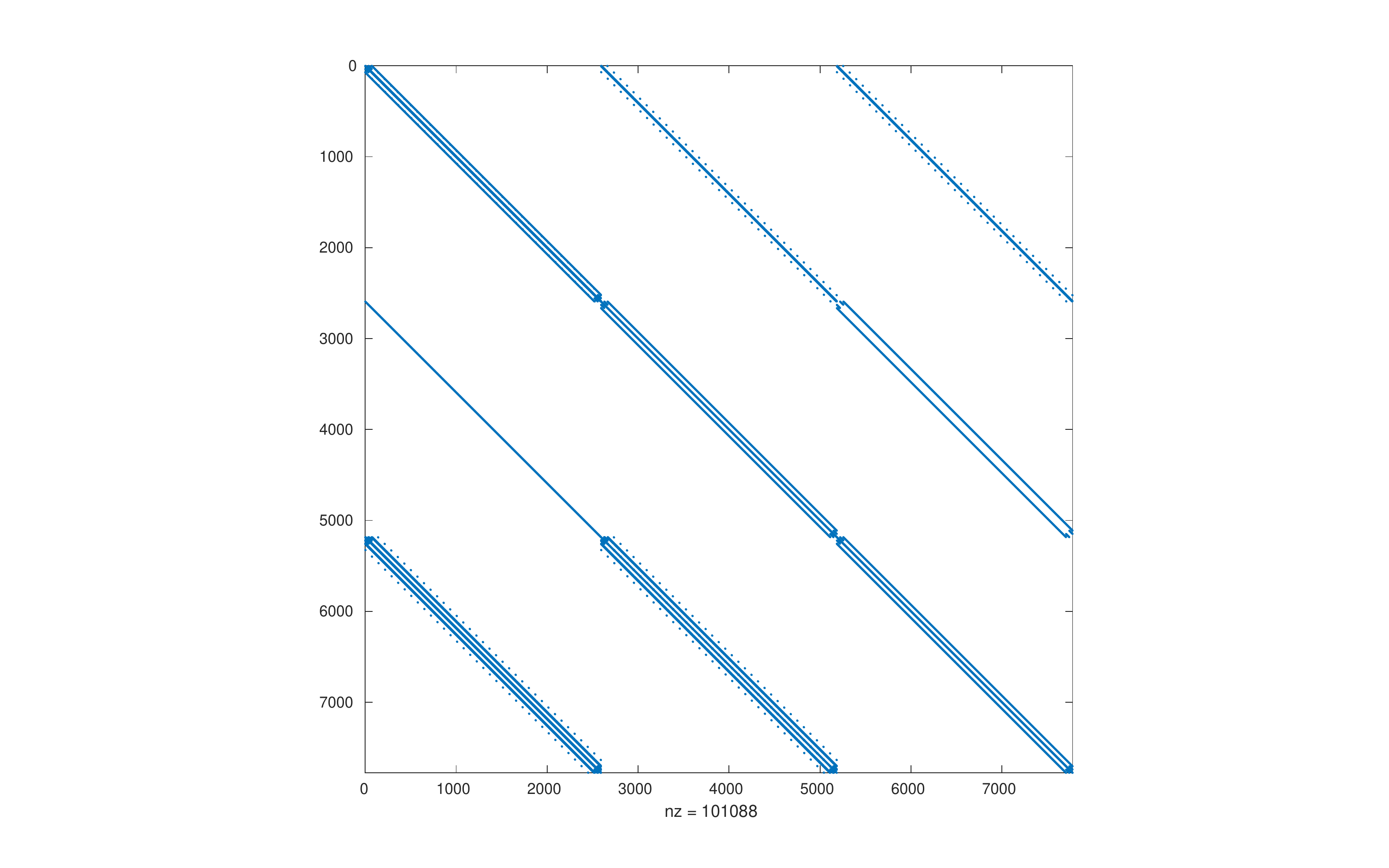}
  \caption{Structure of Jacobian evaluated at $y_0$ for the Shallow Water Equations \eqref{eqn:SWE}\label{fig:SWE-Jac-y0}}
\end{figure}

\begin{figure}[!h]
  \centering
  \includegraphics[scale=0.35]{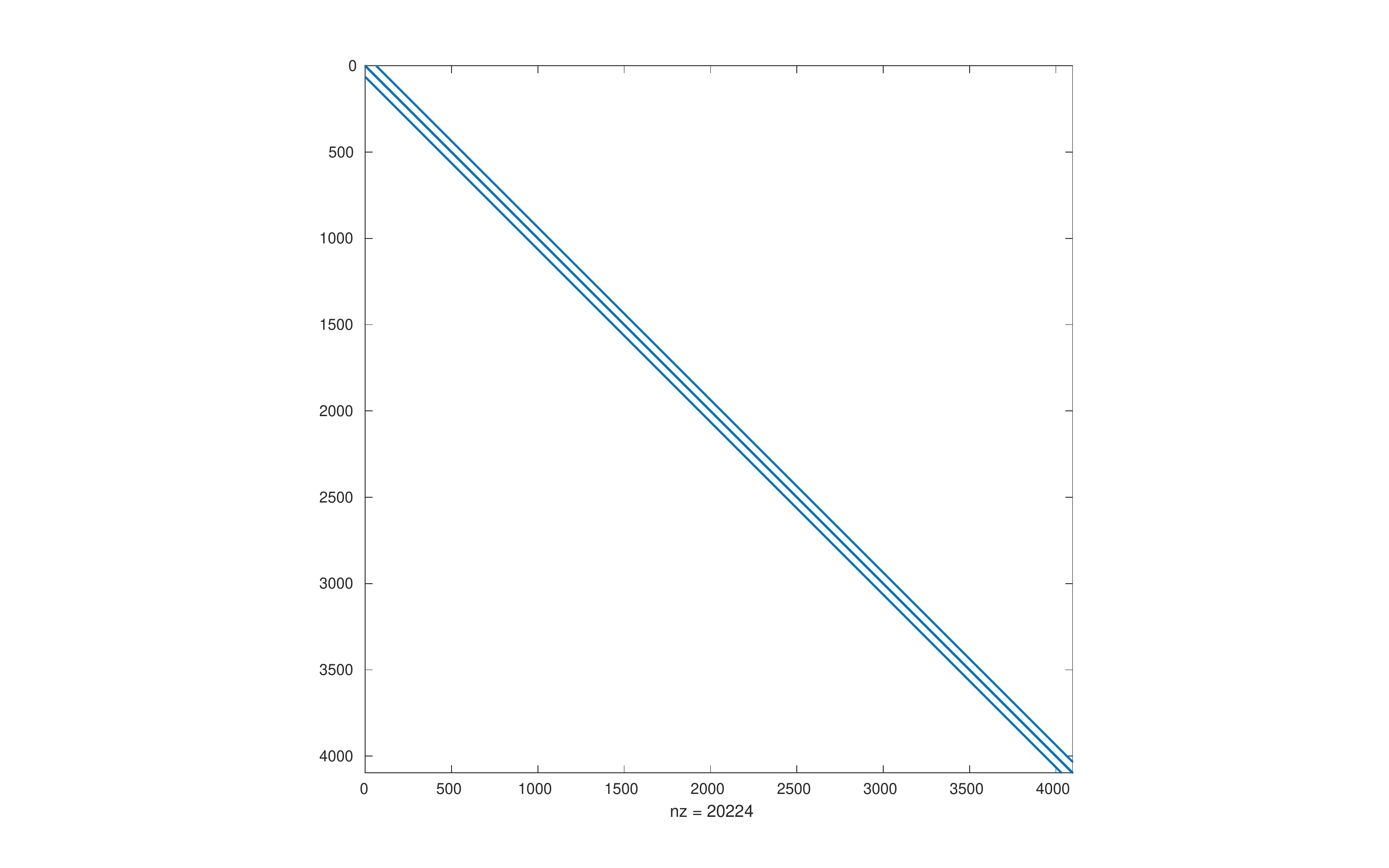}
  \caption{Structure of Jacobian evaluated at $y_0$ for the Allen-Cahn Equations \eqref{eqn:AllenCahn} on a $64 \times 64$ grid.\label{fig:Alcn-Jac-y0}}
\end{figure}

\clearpage

\section*{References}
\bibliographystyle{plain}
\bibliography{Master,ode_exponential,ode_general,ode_krylov}

\end{document}